\documentclass[11pt]{article}
\usepackage[margin=1in]{geometry}
\usepackage{amsthm,amsmath,amsfonts,amssymb,mathtools,bm}
\usepackage[authoryear,round]{natbib}
\usepackage{graphicx}
\usepackage{booktabs}
\usepackage{enumitem}
\usepackage[section]{placeins}
\usepackage{hyperref}
\hypersetup{colorlinks=true,citecolor=blue,linkcolor=blue,urlcolor=blue}

\numberwithin{equation}{section}
\theoremstyle{plain}
\newtheorem{theorem}{Theorem}[section]
\newtheorem{proposition}[theorem]{Proposition}
\newtheorem{lemma}[theorem]{Lemma}
\newtheorem{corollary}[theorem]{Corollary}

\theoremstyle{definition}
\newtheorem{assumption}[theorem]{Assumption}

\newtheorem{remark}[theorem]{Remark}

\newcommand{\R}{\mathbb{R}}
\newcommand{\E}{\mathbb{E}}
\newcommand{\Var}{\operatorname{Var}}
\newcommand{\Cov}{\operatorname{Cov}}

\newcommand{\Toep}{\operatorname{Toep}}
\newcommand{\diag}{\operatorname{diag}}
\newcommand{\sign}{\operatorname{sign}}

\newcommand{\Span}{\operatorname{span}}
\newcommand{\Id}{I}
\newcommand{\one}{\mathbf{1}}
\newcommand{\eps}{\varepsilon}
\newcommand{\calR}{\mathcal R}
\newcommand{\calI}{\mathcal I}
\newcommand{\calC}{\mathcal C}

\newcommand{\vect}[1]{\bm{#1}}
\newcommand{\norm}[1]{\left\lVert #1\right\rVert}
\newcommand{\ip}[2]{\left\langle #1,#2\right\rangle}

\newcommand{\parhead}[1]{\medskip\noindent\textbf{#1.}\quad}

\newcommand{\Prb}{\mathbb{P}}
\newcommand{\Om}{\Omega}
\newcommand{\ph}{\varphi}
\newcommand{\kobs}{\kappa_{\rm obs}}

\newcommand{\opnorm}[1]{\left\lVert #1\right\rVert_2}
\newcommand{\Fnorm}[1]{\left\lVert #1\right\rVert_F}

\title{Fixed-Threshold One-Bit Toeplitz Covariance Estimation under Sparse-Ruler Sampling}
\author{%
  Zhiyong Cheng\thanks{School of Computer and Artificial Intelligence, Chaohu University.}
  \and
  Shengyao Chen\thanks{School of Electronic and Optical Engineering, Nanjing University of Science and Technology.}%
}
\date{}

\begin{document}
\maketitle

\begin{abstract}
We estimate the Toeplitz covariance matrix of a centered Gaussian distribution from data that are both coarsely quantized and sparsely sampled.
Only the coordinates of a sparse ruler are recorded, and each recorded value is kept as a single bit: the sign of its comparison with a fixed threshold.
Such data arise in low-precision sensing front ends and sparse sensor arrays.
Because the threshold is nonzero, every bit has a common mean.
Each bit is also reused across many of the products that build the covariance, so one bit's error enters many of them.
Centering removes the shared error.
We prove a Gaussian variance contraction theorem for products of a centered, bounded nonlinearity of a Gaussian vector, the non-smooth one-bit sign included; it sets each lag's variance by how well the ruler covers that lag.
The resulting estimator needs neither the signal scale nor the bit mean in advance, since the nonzero threshold makes both identifiable from the marginal bits.
A matching minimax lower bound shows the resulting coverage rate is optimal up to constants over a neighborhood of white noise; the bound holds even for the unquantized real-valued samples, so one-bit quantization costs only a constant factor.
\end{abstract}

\par
\noindent\textbf{MSC2020 subject classifications:} Primary 62H12; secondary 60G15, 62M15, 62G20.

\bigskip

\section{Introduction and main results}
\label{sec:intro}

We study how accurately the Toeplitz covariance matrix of a Gaussian distribution can be estimated from data that are degraded in two ways at once.
Each snapshot is observed only on a deterministic sparse ruler \(\Omega\subset\{0,\ldots,d-1\}\), and every observed coordinate is kept as a single bit, the sign \(Y_j^{(\ell)}=\sign(X_j^{(\ell)}-\lambda)\) of its comparison with a fixed threshold \(\lambda\).
This is the regime of low-precision, large-scale acquisition, such as undithered analog-to-digital front ends and sparse sensor arrays, where coarse quantization and incomplete sampling occur together rather than in isolation.
We ask what operator-norm estimation rate is attainable here and which feature of the sampling design governs it.

On the quantization side, the analysis is usually made tractable by dithering.
A controlled dither signal added before each comparison makes the quantized products conditionally unbiased and simplifies the analysis \citep{DirksenMalyRauhut2022,DirksenMaly2024}, including in the closest ruler-based Toeplitz analyses \citep{xu2026bit}, but it requires a per-sample reference signal that many low-precision, high-channel-count front ends do not include.
A plain fixed comparator is a common alternative, as used in one-bit sparse arrays \citep{BarShalomWeiss2002,LiuVaidyanathan2017,ChengChen2020,Sedighi2021}, and a fixed nonzero threshold is the dither-free way to keep the scale information that a symmetric channel would lose \citep{LiuLin2021,Eamaz2023,Xiao2023}.
We therefore study this undithered fixed-threshold regime on a sparse ruler; the sparse-ruler and compressive side has otherwise been developed at full precision \citep{Eldar2020,RomeroLopezLeus2015}.
The difficulty is specific to this combination.
A nonzero threshold gives every observed bit a common nonzero mean, and a sparse ruler reuses each bit across many deterministic lag products, so the shared mean injects the same one-vertex fluctuation into all of them, a coherent error that dithering and random masking never create and that standard quadratic-form concentration does not control.

Centering removes this obstruction and reveals the coverage coefficient \(\varphi(\Omega)\) as the quantity that governs the leading operator-norm rate.
The paper turns this into matching upper and lower bounds, sharp up to constants in the natural operating regime for balanced designs, with a Gaussian variance-contraction theorem as the main probabilistic tool.

\subsection{Notation}

Vectors lie in $\mathbb R^d$ with Euclidean inner product $\ip{u}{v}=\sum_i u_iv_i$ and norm $\norm{v}_2=\ip{v}{v}^{1/2}$; $\Id_d$ is the identity matrix and $\one$ the all-ones vector, with the dimension dropped when it is clear from context.
For a matrix $M$, $M^\top$ is its transpose, $M^*$ its conjugate transpose, and $\operatorname{Re}M$, $\operatorname{Im}M$ its entrywise real and imaginary parts; $M\succeq0$ means that $M$ is positive semidefinite.
The operator, or spectral, norm $\norm{M}_2=\sup_{\norm{x}_2=1}\norm{Mx}_2$ equals, for a symmetric or Hermitian $M$, the largest magnitude of its eigenvalues, and the Frobenius norm is $\norm{M}_F=\bigl(\sum_{j,k}|M_{jk}|^2\bigr)^{1/2}$.
For a closed convex set $S$, $\Pi_S$ denotes the Euclidean nearest-point projection onto $S$.
We write $M\circ N$ for the entrywise, or Hadamard, product and $M^{\circ k}$ for its $k$-fold entrywise power.
A square matrix is hollow if its diagonal is zero, and a symmetric Toeplitz matrix is constant along each diagonal; for a lag sequence $b=(b_1,\ldots,b_{d-1})$, $T_d^0(b)$ is the hollow symmetric Toeplitz matrix with entry $b_{|j-k|}$ at position $(j,k)$ for $j\ne k$, and $\Toep_d(a)$ is the symmetric Toeplitz matrix with diagonal $a_0$ and lag sequence $a_1,\ldots,a_{d-1}$.
We write $\E$, $\Var$, $\Cov$, and $\Prb$ for expectation, variance, covariance, and probability; $\mathcal N(\mu,\Sigma)$ denotes the Gaussian law, $\Phi$ the standard normal distribution function, $Q=1-\Phi$ its survival function with inverse $Q^{-1}$, and $\sign(t)$ the sign of $t$, with $\sign(0)=0$.
Constants $C$ and $c$ denote positive numbers that depend only on the quantities indicated at each statement and may change from occurrence to occurrence; we write $a\lesssim b$ when $a\le Cb$ for such a constant, and $a\asymp b$ when both $a\lesssim b$ and $b\lesssim a$ hold.

\subsection{Model and main results}

Our setting is Toeplitz covariance estimation from a deterministic sparse ruler \citep{Moffet1968,LinebargerSudboroughTollis1993,PalVaidyanathan2010,VaidyanathanPal2011}.
Let \(X^{(1)},\ldots,X^{(n)}\in\mathbb R^d\) be independent centered Gaussian vectors with symmetric Toeplitz covariance \(\Gamma_{jk}=\gamma_{|j-k|}\), \(\gamma_0>0\), and normalized lags \(\rho_s=\gamma_s/\gamma_0\).
A ruler \(\Omega\subset\{0,\ldots,d-1\}\) selects the coordinates that are read, and at a fixed nonzero threshold \(\lambda\) each observation is a single bit \(Y_j^{(\ell)}=\sign(X_j^{(\ell)}-\lambda)\), \(j\in\Omega\), so the lags are estimated from products of these signs along the ruler.
For each lag \(s\) the observed pairs are \(\Omega_s=\{(j,k)\in\Omega^2:k-j=s\}\) with multiplicity \(q_s=|\Omega_s|\), assumed at least one for every \(s\).
The design enters through the coverage coefficient
\begin{equation}\label{eq:intro-phi}
    \varphi(\Omega)=\sum_{s=1}^{d-1}q_s^{-1}.
\end{equation}
We want the operator-norm rate this design allows and the feature of the ruler that sets it.
The answer is \(\varphi(\Omega)\) rather than the ruler size \(|\Omega|\), the same functional that already appears as an upper-bound scale in the dithered ruler line \citep{xu2026bit,XuZhengYang2025}; the rest of this subsection establishes it on a fixed one-bit channel.

At a nonzero threshold a direct sign-product estimator does not aggregate cleanly, because every sign carries a common mean.
Writing \(u_i=\sign(g_i-\tau)=\mu+v_i\) with \(\tau=\lambda/\sqrt{\gamma_0}\), \(\mu=\E u_i\), and \(\E v_i=0\), the raw product splits as
\begin{equation}\label{eq:intro-vertex-decomp}
    u_i u_j-\mathbb E(u_i u_j)
    =\{v_i v_j-\mathbb E(v_i v_j)\}+\mu(v_i+v_j),
\end{equation}
where the second term is a one-vertex projection that the ruler copies into every lag product touching that vertex, adding a row-sum variance that can dominate the edge-Frobenius scale.
We identify this vertex-projection obstruction and remove it by centering.
Proposition~\ref{prop:raw-obstruction} and Corollary~\ref{cor:sparse-ruler-obstruction} show that subtracting the marginal mean before aggregation cancels the projection and leaves a degenerate pair statistic that aggregates at the coverage scale \(\varphi(\Omega)\).

Centering does not by itself bound what remains.
The centered statistic is a degenerate quadratic form in hard-threshold signs, which are non-smooth with an infinite Hermite expansion, placing them beyond both the fixed-Hermite-rank limit theory for nonlinear functionals of Gaussian sequences \citep{BreuerMajor1983,Arcones1994} and order-dependent chaos concentration of Hanson--Wright type.
We bound its fluctuation directly, by a single-snapshot variance contraction that holds uniformly over all chaos orders (Theorem~\ref{thm:contraction}).
For \(g\sim\mathcal N(0,C)\), a bounded centered transform \(v_i=h(g_i)\), and a hollow matrix \(A\) whose active pairs stay away from \(|C_{ij}|=1\),
\[
    \Var(v^\top A v)\le K_{h,\varepsilon}\|C\|_2^2\|A\|_F^2 .
\]
Apart from the transform \(h\) and the boundary margin \(\varepsilon\), the constant is free of the dimension, the support size, the maximum degree, and the chaos order.
This order-uniformity is what lets the bound reach hard signs, and it comes from a fusion-frame Bessel inequality for degenerate pair-chaoses (Proposition~\ref{prop:fusion-frame-main}), in which a creation--annihilation factor cancels against the two-coordinate number operator.
The inequality is new in this non-asymptotic, arbitrary-support form, and we expect it to be of independent interest for the variance analysis of nonlinear functionals of Gaussian fields.

These ingredients give the estimator and its rate.
Whether that rate is optimal is settled by matching operator-norm bounds.
The oracle estimator forms centered products with the fixed-threshold Gaussian sign link \citep{Bussgang1952,Price1958,VanVleckMiddleton1966}, whose sign correlation \(c(\rho;\tau)=\Cov\{\sign(G_1-\tau),\sign(G_2-\tau)\}\) is strictly increasing on compact regularity intervals \citep{Plackett1954} and is inverted lag by lag, while the plug-in estimator first learns \((\gamma_0,\tau,\mu)\) from the \(n|\Omega|\) marginal signs.
Write \(t=\log(Cd/\delta)\), \(\kappa_{\rm obs}=\|\Gamma_{\Omega,\Omega}\|_2/\gamma_0\), and \(S_1(d;\rho)=\sum_{s=1}^{d-1}|\rho_s|\), and let \(L_1,L_2\ge1\) bound the first two derivatives of the inverse link.
Theorem~\ref{thm:toeplitz} shows that, with probability at least \(1-\delta\), the plug-in estimator satisfies
\begin{align*}
    \|\widehat\Gamma_{\rm plug}-\Gamma\|_2
    \le{}& C_\varepsilon\gamma_0L_1\kappa_{\rm obs}\sqrt{\frac{\varphi(\Omega)t}{n}}
    + C\gamma_0L_1d\,\frac{t}{n}
    + C_{\tau,\varepsilon}\gamma_0L_2\left\{\min\{\kappa_{\rm obs}^2\varphi(\Omega),d\}\,\frac{t}{n}+d\,\frac{t^2}{n^2}\right\}\\
    &{}+ C_{\tau,\varepsilon}\gamma_0\{1+S_1(d;\rho)\}\left\{\sqrt{\frac{\kappa_{\rm obs}t}{n|\Omega|}}+\frac{t}{n}\right\},
\end{align*}
and the oracle estimator obeys the same bound without the last term.
Centered pair coverage enters through \(\varphi(\Omega)\), spectral boundedness through the \(L_1d\,t/n\) term, inverse-link curvature at the coverage scale, and marginal calibration through the \(n|\Omega|\) one-coordinate signs.
A matching spectral-packing lower bound in the known-scale identity-neighborhood submodel \(\mathcal P_{\mathbb R}(c_0)\) of Section~\ref{subsec:lower-bounds} (Theorem~\ref{thm:spectral-packing-minimax}, Corollary~\ref{cor:coverage-log-lower}) shows that the minimax risk over this submodel is at least \(c\gamma_0\min\{1,\sqrt{\varphi(\Omega)\log d/n}\}\).
In the operating regime \(n\gtrsim\max\{\varphi(\Omega),d^2/\varphi(\Omega)\}\log d\) for balanced designs the two bounds meet, pinning down the minimax rate
\[
    \inf_{\widehat\Gamma}\sup_{\Gamma\in\mathcal P_{\mathbb R}(c_0)}\mathbb E_\Gamma\|\widehat\Gamma-\Gamma\|_2
    \asymp\gamma_0\sqrt{\frac{\varphi(\Omega)\log d}{n}},
\]
attained by the oracle with known scale and by the plug-in with unknown scale (Corollaries~\ref{cor:regime-minimax} and~\ref{cor:plugin-minimax}).
Because this lower bound is proved by data processing from full-precision sparse-ruler samples, for balanced designs \(\varphi(\Omega)\) is intrinsic to the sampling and fixed-threshold one-bit quantization costs only constant factors in this regime.
The statement is local, over this submodel and operating regime, and we do not claim the remaining factors are sharp.

\subsection{Relation to existing work}

Operator-norm covariance estimation has separate traditions for ordered, sparse, Toeplitz, and subsampled structure.
Banding, tapering, thresholding, and block-thresholding provide high-dimensional benchmarks \citep{BickelLevina2008Regularized,BickelLevina2008Thresholding,CaiZhangZhou2010,CaiYuan2012,CaiZhou2012}, while Toeplitz-specific rates reflect the ordered lag geometry \citep{CaiRenZhou2013,KlockmannKrivobokova2024}.
Sparse rulers and related sparse-array constructions provide deterministic lag coverage \citep{Moffet1968,LinebargerSudboroughTollis1993,PalVaidyanathan2010,VaidyanathanPal2011}, and full-precision compressive covariance sensing studies this geometry without one-bit quantization \citep{AriananadaLeus2012,RomeroLopezLeus2015,RomeroArianandaTianLeus2016,Eldar2020,Lawrence2020}.
Our lower bound refines the ruler endpoint to the design-adapted functional \(\varphi(\Omega)\) under balanced coverage.

Missing-data covariance estimation usually assumes random observation patterns independent of the data \citep{Lounici2014,CaiZhang2016,Abdalla2024}.
Masked and quantized covariance estimation gives closer comparisons, including unquantized masks, Toeplitz masks, and quantized masked estimators \citep{LevinaVershynin2012,ChenGittensTropp2012,KabanavaRauhut2017,DirksenMalyRauhut2022,MalyYangDirksenRauhutCaire2022}.
Here the mask is a fixed sparse ruler, the entries are one-bit signs, and deterministic reuse of mean-shifted bits creates the vertex-projection obstruction.

Dithered quantization changes the experiment by making quantized products conditionally unbiased \citep{ChenWangNgWang2023,DirksenMaly2024}.
Ruler-based dithered Toeplitz estimators already exhibit the coverage scale \(\varphi(\Omega)\) \citep{xu2026bit,XuZhengYang2025}.
Our fixed-threshold experiment keeps the hardware channel unchanged, so the statistical problem still involves link inversion, its curvature, scale calibration, and centering.

Nonzero and time-varying thresholds are standard tools for scale recovery in one-bit covariance and autocorrelation estimation \citep{LiuLin2021,Eamaz2023,Xiao2023,LiuChou2025}, and the sign-correlation link itself goes back to clipped-noise analysis \citep{Bussgang1952,Price1958,VanVleckMiddleton1966,Plackett1954}.
One-bit direction finding and sparse-array variants usually use zero-threshold signs and the arcsine law \citep{BarShalomWeiss2002,LiuVaidyanathan2017,ChengChen2020,Sedighi2021}.
A zero threshold avoids the mean shift but is scale-blind, whereas the fixed nonzero-threshold channel is scale-informative and therefore needs the centering theory developed here.

\subsection{Organization of the paper}

Section~\ref{sec:sampling-estimator} sets up the sampling model, the marginal calibration, the centered threshold-sign link, the oracle and plug-in estimators, and the general sparse-pair notation used in the proofs.
Section~\ref{sec:centering} establishes the vertex-projection obstruction and shows that centering removes it.
Section~\ref{sec:contraction} proves the variance-contraction theorem for centered sparse-pair statistics, together with the fusion-frame Bessel inequality that supplies its chaos-order uniformity.
Section~\ref{sec:estimator} assembles these into the operator-norm upper bounds and the matching spectral-packing lower bound, and draws the minimax consequences.
Section~\ref{sec:discussion} discusses the results and open questions, and the supplement contains all proofs and the settings of the numerical rate checks.

\section{Statistical model and estimators}
\label{sec:sampling-estimator}

This section sets up the statistical experiment before introducing the general sparse-pair notation used in the proofs.
The main object of the paper is an estimator and its risk under a deterministic sampling design.
The Gaussian contraction theorem in Section~\ref{sec:contraction} is a tool for controlling the centered estimator.
The construction separates three statistical components of the Toeplitz estimator: marginal calibration, centered pair interaction, and deterministic aggregation geometry.

\subsection{Sparse-ruler one-bit observations}
\label{subsec:toeplitz-observation}

Let $X^{(1)},\ldots,X^{(n)}\in\mathbb R^d$ be independent Gaussian vectors with mean zero and symmetric Toeplitz covariance $\Gamma$:
\[
    \Gamma_{jk}=\gamma_{|j-k|},
    \qquad \gamma_0>0 .
\]
Write $\rho_s=\gamma_s/\gamma_0$ for the normalized lags.
A deterministic sparse ruler $\Om\subset\{0,\ldots,d-1\}$ specifies the observed coordinates.
For a fixed threshold $\lambda>0$ we observe
\begin{equation}\label{eq:one-bit-observation}
    Y^{(\ell)}_j
    =\operatorname{sign}(X^{(\ell)}_j-\lambda),
    \qquad j\in\Om,
    \quad \ell=1,\ldots,n .
\end{equation}
The target is the full $d\times d$ Toeplitz covariance matrix $\Gamma$, while observations are restricted to the coordinates in $\Om$.

For positive lags define
\begin{equation}\label{eq:omega-s}
    \Om_s=\{(j,k)\in\Om^2:k-j=s\},
    \qquad q_s=|\Om_s|,
    \qquad s=1,\ldots,d-1 .
\end{equation}
We say that $\Om$ covers all lags when $q_s\ge1$ for every $s$.
Its coverage complexity is
\begin{equation}\label{eq:phi-coverage}
    \ph(\Om)=\sum_{s=1}^{d-1}q_s^{-1} .
\end{equation}
This quantity records how unevenly the deterministic ruler covers the lags.

The $n|\Om|$ marginal bits are used for pooled calibration of $(\gamma_0,\tau,\mu)$.
The lag products constructed from those bits estimate the covariance, but pairs within a snapshot share vertices and are not independent samples.
After centering, the lag coverage profile $q_s$ enters through the Frobenius scale $\ph(\Om)$.

\subsection{Marginal calibration}
\label{subsec:marginal-calibration}

For one coordinate $x\sim N(0,\gamma_0)$, put
\[
    \tau=\lambda/\sqrt{\gamma_0},
    \qquad q=\Prb(x\ge\lambda)=Q(\tau),
    \qquad \mu=\E\operatorname{sign}(G-\tau)=1-2\Phi(\tau),
\]
where $G\sim N(0,1)$.
A nonzero known threshold identifies the scale $\gamma_0$ through the marginal exceedance probability \citep{ChapeauBlondeau2008}.
Fix a compact regular threshold interval $0<\tau_{\min}<\tau_{\max}<\infty$ and write $Q_\tau=[Q(\tau_{\max}),Q(\tau_{\min})]$.
The pooled exceedance estimator is
\begin{equation}\label{eq:qhat}
    \widehat q_{\rm raw}
    =\frac{1}{n|\Om|}
      \sum_{\ell=1}^n\sum_{j\in\Om}
      \one\{Y^{(\ell)}_j=1\}.
\end{equation}
We set
\begin{equation}\label{eq:plugin-calibration}
    \widehat q=\Pi_{Q_\tau}(\widehat q_{\rm raw}),
    \qquad
    \widehat\tau=Q^{-1}(\widehat q),
    \qquad
    \widehat\gamma_0=(\lambda/\widehat\tau)^2,
    \qquad
    \widehat\mu=1-2\Phi(\widehat\tau).
\end{equation}
The projection makes the estimator globally defined and is inactive on the regular high-probability calibration event.

\subsection{Centered threshold-sign link}
\label{subsec:centered-link}

For a bivariate standard Gaussian pair $(G_1,G_2)$ with correlation $r$, define the centered link
\begin{equation}\label{eq:centered-link}
    c(r;\tau)
    =\operatorname{Cov}\{\operatorname{sign}(G_1-\tau),
                         \operatorname{sign}(G_2-\tau)\}
    =4\{\Prb(G_1>\tau,G_2>\tau)-Q(\tau)^2\} .
\end{equation}
On every compact interval $[-1+\varepsilon,1-\varepsilon]$, Plackett's identity \citep{Plackett1954} gives $\partial_r c(r;\tau)>0$, and the inverse map is denoted by $\psi(\cdot;\tau)$.
The estimator below uses this centered link rather than the raw sign-product link, removing the sampling obstruction of Section~\ref{sec:centering}.

\subsection{Oracle and plug-in lag estimators}
\label{subsec:lag-estimator}

First suppose that $(\gamma_0,\tau,\mu)$ are known.
Define the centered signs
\[
    V^{(\ell)}_j=Y^{(\ell)}_j-\mu .
\]
For $s\ge1$ set
\begin{equation}\label{eq:chat-oracle}
    \widehat c^{\rm or}_{s}
    =\frac{1}{nq_s}\sum_{\ell=1}^n
      \sum_{(j,k)\in\Om_s}
      V^{(\ell)}_{j}V^{(\ell)}_{k} .
\end{equation}
The normalized lag estimator is
\begin{equation}\label{eq:rho-oracle}
    \widehat\rho^{\rm or}_{s}
    =\psi\bigl(\Pi_{\calI_\tau}(\widehat c^{\rm or}_{s});\tau\bigr),
    \qquad s=1,\ldots,d-1,
\end{equation}
where $\calI_\tau=c([-1+\eps,1-\eps];\tau)$ is the compact inverse-link domain.
Clipping keeps the estimator globally defined and bounded; it is inactive on the regular high-probability event.
The oracle covariance estimator is the symmetric Toeplitz completion:
\begin{equation}\label{eq:Rhat-oracle}
    \widehat \Gamma_{\rm or}=\Toep_d(\gamma^{\rm or}),
    \qquad
    \widehat\gamma^{\rm or}_0=\gamma_0,
    \qquad
    \widehat\gamma^{\rm or}_s=\gamma_0\widehat\rho^{\rm or}_s .
\end{equation}

The plug-in estimator is defined by replacing $(\gamma_0,\tau,\mu)$ in \eqref{eq:chat-oracle}--\eqref{eq:Rhat-oracle} by $(\widehat\gamma_0,\widehat\tau,\widehat\mu)$ from \eqref{eq:plugin-calibration}; the arguments of $\psi$ are clipped to $\calI_{\widehat\tau}$ in the same way.
We denote the resulting estimator by $\widehat \Gamma_{\rm plug}$.

\subsection{Risk and regularity parameters}
\label{subsec:risk-regularity}

The performance criterion is operator-norm risk over a class $\mathcal P$ of Toeplitz covariance matrices:
\[
    \inf_{\widehat \Gamma}\sup_{\Gamma\in\mathcal P}
    \E_\Gamma\opnorm{\widehat \Gamma-\Gamma} .
\]
The upper bounds below are stated in high probability.
The observed-submatrix relative operator norm is
\begin{equation}\label{eq:kobs}
    \kobs=\opnorm{\Gamma_{\Om,\Om}}/\gamma_0,
\end{equation}
and the short-memory size is
\begin{equation}\label{eq:S1}
    S_1(d;\rho)=\sum_{s=1}^{d-1}|\rho_s| .
\end{equation}
The inverse-link constants $L_1,L_2\ge1$ are uniform bounds on the first two derivatives of $\psi(\cdot;\tau')$ over the buffered compact regularity domain of Assumption~\ref{ass:inverse}.

\subsection{General dependent sparse-pair notation}
\label{subsec:general-edge-notation}

Sections~3 and~4 use a graph abstraction of the same sampling mechanism.
Let $g=(g_1,\ldots,g_m)\sim N(0,C)$, where $C$ is a correlation matrix and $\opnorm{C}=\kappa$.
Let $E\subset\{(i,j):i\ne j\}$ be a deterministic edge set and let $h\in L^2(\gamma)\cap L^\infty(\gamma)$, where $\gamma$ is the standard $N(0,1)$ law, satisfy $\E h(G)=0$.
Put $v_i=h(g_i)$.
Weighted sparse-pair statistics have the form
\begin{equation}\label{eq:general-edge-stat}
    T_E(a,h)=\sum_{(i,j)\in E}a_{ij}\{v_i v_j-\E(v_i v_j)\} .
\end{equation}
For an undirected weighted design, let $a_{ij}=a_{ji}$ and define
\begin{equation}\label{eq:edge-and-row-geometry}
    \calC(a)=\sum_{i<j}a_{ij}^2,
    \qquad
    r_i=\sum_{j:j\ne i}a_{ij},
    \qquad
    \calR(a)=\sum_i r_i^2 .
\end{equation}
The centered theory controls $\calC(a)$; raw nonzero-threshold products also contain the row-sum term $\calR(a)$.

The local support-separation condition used in the contraction theorem is the following.
\begin{assumption}[Support separation]
\label{ass:support-separation}
\begin{equation}\label{eq:support-separation}
    A_{ij}\ne0,\ i\ne j
    \quad\Longrightarrow\quad
    |C_{ij}|\le1-\varepsilon .
\end{equation}
\end{assumption}
It excludes nearly duplicated Gaussian coordinates on the active edge support.
In the Toeplitz application, Assumption~\ref{ass:inverse} gives $|C_{jk}|=|\rho_{|j-k|}|\le1-2\eps$ on every observed nonzero-lag pair, so Assumption~\ref{ass:support-separation} holds with the theorem parameter $\eps$.
The Toeplitz model has the stronger separation margin $2\eps$, and the two occurrences of $\eps$ refer to the same constant.

For sparse-ruler spectral aggregation, deterministic weights $w_s$ generate the hollow frequency matrix $A_\theta(w)$ by
\[
    (A_\theta(w))_{jk}=q_s^{-1}w_s e^{2\pi i s\theta},
    \qquad k-j=s>0,
    \qquad (j,k)\in\Om_s,
\]
with Hermitian completion and zero diagonal. The corresponding centered spectral polynomial is
\begin{equation}\label{eq:spectral-polynomial}
    G_\theta(w)=v^*A_\theta(w)v-\E\,v^*A_\theta(w)v,
\end{equation}
a real trigonometric polynomial of degree $d-1$ in $\theta$.
Then
\begin{equation}\label{eq:freq-frobenius}
    \Fnorm{A_\theta(w)}^2
    =2\sum_{s=1}^{d-1}|w_s|^2q_s^{-1}
    \le 2W^2\ph(\Om),
    \qquad W=\max_s|w_s|.
\end{equation}
This identity is the bridge from the general edge theorem to the Toeplitz operator-norm rate.
Here $\Fnorm{A_\theta(w)}^2$ is the squared Frobenius norm of the hollow Toeplitz restriction to $\Omega$ with lag weights $q_s^{-1}w_se^{2\pi is\theta}$; the same sparse-pair restriction functional governs the centered obstruction variance of Section~\ref{sec:centering} and the sparse Frobenius metric of the lower bound, evaluated at different lag weights.

\section{The sampling obstruction: vertex projection under deterministic pair reuse}
\label{sec:centering}

The estimator in Section~\ref{sec:sampling-estimator} centers the one-bit signs before forming lag products.
To see why, it suffices to work at identity covariance.
Toeplitz approximation, inverse-link curvature and Gaussian conditioning then drop out, leaving deterministic vertex reuse as the only source of dependence.

Let $u_j=\sign(G_j-\tau)$ with $\tau\ne0$, $\mu=\E u_j$, and $v_j=u_j-\mu$.
The raw product decomposition \eqref{eq:intro-vertex-decomp} shows that raw sparse pair averages contain a linear statistic in the centered signs.
The following weighted-graph identity makes this one-vertex projection explicit.

\begin{proposition}[Vertex-projection obstruction on weighted edge designs]
\label{prop:raw-obstruction}
Let $G_i$, $i\in V$, be independent standard normal variables.
Let $u_i=\sign(G_i-\tau)=\mu+v_i$, where $\tau\ne0$, $\E v_i=0$ and $\Var(v_i)=\sigma_v^2$.
For real symmetric weights $a_{ij}=a_{ji}$ on an undirected edge design, define
\[
S_{\rm raw}=2\sum_{i<j}a_{ij}\{u_i u_j-\E(u_i u_j)\},
\qquad
S_{\rm cen}=2\sum_{i<j}a_{ij}v_i v_j,
\]
and $r_i=\sum_{j:j\ne i}a_{ij}$.
Then
\begin{equation}
S_{\rm raw}=S_{\rm cen}+2\mu\sum_i r_i v_i,
\label{eq:weighted-raw-decomposition}
\end{equation}
and the two terms on the right are uncorrelated.
Consequently,
\begin{align}
\Var(S_{\rm raw})
&=4\sigma_v^4\sum_{i<j}a_{ij}^2
  +4\mu^2\sigma_v^2\sum_i r_i^2,
\label{eq:weighted-raw-variance}\\
\Var(S_{\rm cen})
&=4\sigma_v^4\sum_{i<j}a_{ij}^2.
\label{eq:weighted-centered-variance}
\end{align}
Equivalently, if $A$ is the associated symmetric hollow matrix, then
\begin{equation}
\Var(S_{\rm cen})=2\sigma_v^4\norm{A}_F^2.
\label{eq:weighted-centered-frobenius}
\end{equation}
\end{proposition}

\begin{corollary}[Sparse-ruler coverage and row-sum separation]
\label{cor:sparse-ruler-obstruction}
Let $\Omega$ be any ruler covering all lags $1,\ldots,d-1$, set $m=|\Omega|$, and define
\[
    \zeta_s^{\rm raw}
      =q_s^{-1}\sum_{(j,k)\in\Omega_s}\{u_j u_k-\E(u_j u_k)\},
    \qquad
    \zeta_s^{\rm cen}=q_s^{-1}\sum_{(j,k)\in\Omega_s}v_jv_k .
\]
Then
\begin{equation}
    \Var\left(2\sum_{s=1}^{d-1}\zeta_s^{\rm raw}\right)
       \ge 16\mu^2\sigma_v^2\frac{(d-1)^2}{m},
\label{eq:raw-lower-bound}
\end{equation}
whereas
\begin{equation}
    \Var\left(2\sum_{s=1}^{d-1}\zeta_s^{\rm cen}\right)
       =4\sigma_v^4\varphi(\Omega).
\label{eq:centered-exact}
\end{equation}
\end{corollary}

The proofs are elementary and are given in the supplement.
Raw products retain the row-sum term $\calR(a)=\sum_i r_i^2$ of \eqref{eq:edge-and-row-geometry}, whereas centered products have only the edge-Frobenius scale $\calC(a)=\sum_{i<j}a_{ij}^2$.
When $\varphi(\Omega)\ll d^2/m$, the row-sum term can dominate even at identity covariance.
Centering removes it and recovers the sparse-pair scale.

\begin{figure}[t]
\centering
\includegraphics[width=.96\linewidth]{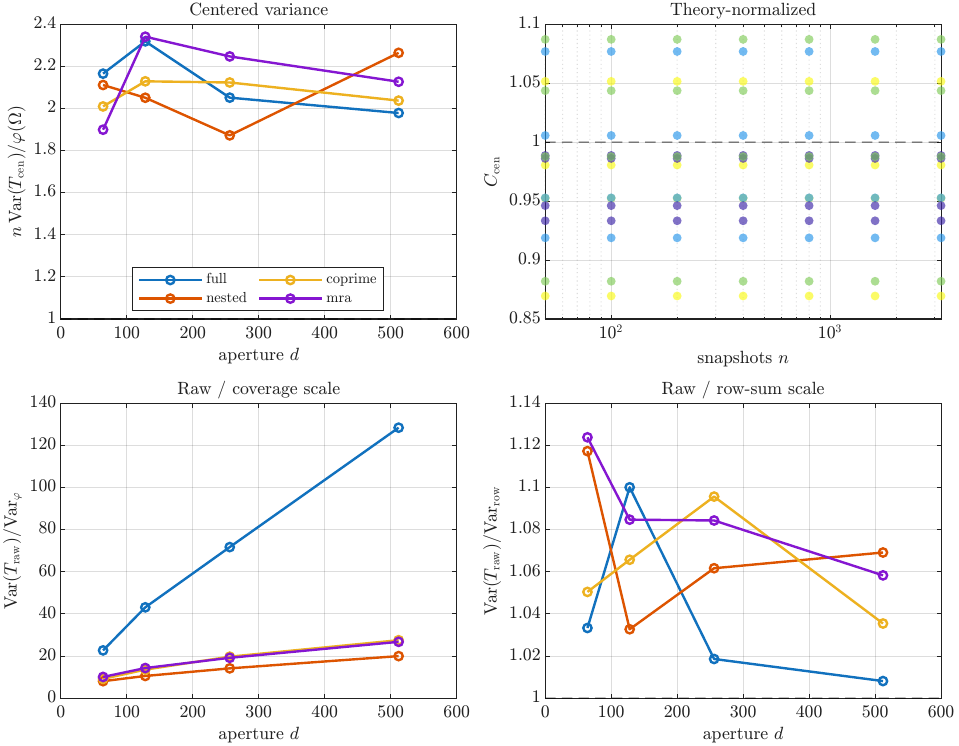}
\caption{Centering changes the variance scale.  Under identity covariance, centered products collapse at the predicted $\varphi(\Omega)/n$ scale across deterministic coverage profiles, whereas raw nonzero-threshold products retain the row-sum obstruction.}
\label{fig:obstruction-diagnostic}
\end{figure}

\section{A contraction theorem for centered sparse-pair statistics}
\label{sec:contraction}

Correlations between Gaussian coordinates could in principle reintroduce degree dependence after centering.
The next theorem rules this out on compact nondegenerate correlation domains by bounding the one-snapshot variance of centered lag averages in terms of the sampling geometry in \eqref{eq:freq-frobenius}.
Sample-level concentration is then obtained by averaging the independent snapshots.

Throughout this section,
\[
E_A=\{(i,j):i\ne j,\ A_{ij}\ne0\}
\]
denotes the ordered support of $A$, and $\norm{A}_F^2=\sum_{(i,j)\in E_A}|A_{ij}|^2$.
When an undirected edge convention is used, the symmetric matrix contains both orientations and the Frobenius norm includes both entries.

\begin{theorem}[Variance contraction for centered Gaussian quadratic forms]
\label{thm:contraction}
Let $g\sim N(0,C)$, where $C$ is a real correlation matrix and $\norm{C}_2=\kappa$.
Let $h:\R\to\R$ be real-valued with $h\in L^2(\gamma)\cap L^\infty(\gamma)$ and $\E h(G)=0$ for $G\sim N(0,1)$, and set $v_i=h(g_i)$.
If $A$ is a hollow Hermitian matrix satisfying Assumption~\ref{ass:support-separation}, then
\begin{equation}
    \Var(v^*Av)\le K_{h,\eps}\kappa^2\norm{A}_F^2 ,
\label{eq:contraction}
\end{equation}
where $K_{h,\eps}$ depends only on $\eps$, $\norm{h}_{L^2(\gamma)}$, and $\norm{h}_\infty$.
\end{theorem}

The constant in \eqref{eq:contraction} never depends on $d$, $|E_A|$, the maximum degree of the edge design, or the entries of $A$.
Theorem~\ref{thm:contraction} is a one-snapshot variance bound rather than a Hanson--Wright-type tail inequality; the sample-level tail in Theorem~\ref{thm:toeplitz} is obtained only after averaging the $n$ independent snapshots.
The centering assumption $\E h(G)=0$ is essential. Without it, Section~\ref{sec:centering} shows that deterministic vertex reuse can create row-sum variance.
For Hermitian $A$, since $v$ is real-valued, the skew-symmetric imaginary part drops out:
\[
v^*Av=v^\top(\operatorname{Re}A)v,
\qquad
v^\top(\operatorname{Im}A)v=0.
\]

Two immediate consequences are used in the Toeplitz proof.

\begin{corollary}[Threshold signs]
\label{cor:threshold}
For
\[
    h_\tau(t)=\sign(t-\tau)-\E\sign(G-\tau),
\]
Theorem~\ref{thm:contraction} gives
\[
    \Var(v^*Av)\le C_\eps\kappa^2\norm{A}_F^2
\]
with a constant independent of $\tau$.
\end{corollary}

\begin{corollary}[Sparse-ruler spectral variance]
\label{cor:ruler-variance}
Let $\Omega$ be a sparse ruler, let $v_i=h_\tau(g_i)$ be the centered threshold signs of Corollary~\ref{cor:threshold}, and let $G_\theta=G_\theta(w)$ be the centered spectral polynomial \eqref{eq:spectral-polynomial}.
If the support correlations satisfy $|C_{jk}|\le1-\eps$ for all nonzero-lag observed pairs and $|w_s|\le W$, then
\begin{equation}
    \sup_{\theta\in[0,1]}\Var(G_\theta)
       \le C_\eps \kappa^2 W^2\varphi(\Omega).
\label{eq:ruler-variance}
\end{equation}
\end{corollary}

The sparse-ruler corollary follows by applying Theorem~\ref{thm:contraction} to the frequency matrix $A_\theta(w)$, whose Frobenius norm is controlled by \eqref{eq:freq-frobenius}.
Theorem~\ref{thm:toeplitz} uses it through the chain
\[
\Fnorm{A_\theta(w)}^2\le2W^2\ph(\Om)
\quad\Longrightarrow\quad
\sup_{\theta\in[0,1]}\Var\{G_\theta(w)\}\le C_\eps\kobs^2W^2\ph(\Om);
\]
Bernstein's inequality, the trigonometric grid and the inverse-link Taylor expansion then complete the argument.
In the Toeplitz upper-bound proof, the separation condition is imposed only on observed nonzero-lag pairs.
Under the buffered regime of Section~\ref{sec:estimator} it holds there with margin $2\eps$, inside the domain where the centered inverse link is stable.

The proof of Theorem~\ref{thm:contraction} is given in the supplement.
Its two ingredients are a pairwise Hoeffding projection \citep{Hoeffding1948} and a Bessel-type inequality for a family of degenerate pair-chaos subspaces, in the spirit of fusion frames \citep{CasazzaKutyniokLi2008}.
The projection separates the residual one-vertex component from the genuinely degenerate two-vertex chaoses; centering makes the one-vertex projection vanish at independence and remain summable under Gaussian dependence.

To state the second ingredient, let $H$ be the Gaussian Hilbert space generated by $g_1,\ldots,g_m$, so that $g_i=W(e_i)$ for an isonormal process $W$ with $\ip{e_i}{e_j}_H=C_{ij}$.
For $N\ge2$, let $H^{\odot N}$ be the $N$-fold symmetric tensor power equipped with the homogeneous Fock norm $\norm{F}_{(N)}^2=N!\,\norm{F}_{H^{\odot N}}^2$, under which the multiple integral $I_N$ is an isometry onto the $N$th Wiener chaos \citep{Janson1997,Nualart2006,PeccatiTaqqu2011}.
For an ordered edge $(i,j)$, define the degenerate pair-chaos subspace
\[
    K_{ij,N}
    =\Span\{e_i,e_j\}^{\odot N}
     \ominus\Span\{e_i^{\odot N},e_j^{\odot N}\},
\]
and let $P^{\rm int}_{ij,N}$ denote the orthogonal projection onto $K_{ij,N}$.
The subtracted pure tensors are exactly the one-vertex directions that can add coherently across edges sharing a vertex.

\begin{proposition}[Fusion-frame bound for degenerate pair chaoses]
\label{prop:fusion-frame-main}
Let $E$ be a finite ordered edge set with $|C_{ij}|\le1-\eps$ for every $(i,j)\in E$.
Then, for every $N\ge2$ and every $F\in H^{\odot N}$,
\begin{equation}
    \sum_{(i,j)\in E}\norm{P^{\rm int}_{ij,N}F}_{(N)}^2
    \le C_\eps\kappa^2\norm{F}_{(N)}^2,
\label{eq:fusion-main}
\end{equation}
and, equivalently, for all $f_{ij}\in K_{ij,N}$,
\begin{equation}
    \norm{\sum_{(i,j)\in E}f_{ij}}_{(N)}^2
    \le C_\eps\kappa^2\sum_{(i,j)\in E}\norm{f_{ij}}_{(N)}^2.
\label{eq:fusion-main-synthesis}
\end{equation}
The constant $C_\eps$ depends only on $\eps$; it is independent of the chaos order $N$, the edge set, its maximum degree, and the dimension.
\end{proposition}

The uniformity in $N$ is what matters.
Creation--annihilation estimates produce a factor $N$, while the number operator on the two-coordinate chaos returns a factor $1/(N\eps)$, and their product removes any dependence on the Hermite order of the transform.
In the proof of Theorem~\ref{thm:contraction}, the interaction part of $h(g_i)h(g_j)$ is expanded over chaoses, each chaos component lies in $K_{ij,N}$, and \eqref{eq:fusion-main-synthesis} is applied with coefficients $A_{ij}$ at every order simultaneously.
Throughout, $H$ and its chaoses are real Hilbert spaces.
For complex coefficients $A_{ij}$ the synthesis bound extends to the complexification by applying the real estimate separately to the real and imaginary parts of the coefficients.
In the proof of Theorem~\ref{thm:contraction} a Hermitian $A$ is in any case reduced to its real symmetric part before the chaos argument is invoked.

\begin{remark}[Coordinate-dependent transforms]
\label{rem:heterogeneous}
Theorem~\ref{thm:contraction} holds verbatim for $v_i=h_i(g_i)$ with coordinate-dependent centered transforms $h_i\in L^2(\gamma)\cap L^\infty(\gamma)$, with $K$ depending only on $\eps$, $\sup_i\norm{h_i}_{L^2(\gamma)}$ and $\sup_i\norm{h_i}_\infty$: no step of the proof uses that the transforms at the two endpoints of an edge coincide (see the closing remark of Section~S1 of the supplement).
In particular, centered threshold signs with coordinate-dependent thresholds $\tau_i$ in a compact interval are covered, which supplies the one-snapshot variance input for the threshold-mismatch extension mentioned in Section~\ref{sec:discussion}.
\end{remark}

This theorem is not a direct consequence of standard quadratic-form tools.
Dependent Hanson--Wright bounds control quadratic forms of the original coordinates, or smooth concentration classes, rather than hollow forms of nonsmooth transforms $h(g_i)$ with infinite Hermite expansions \citep{RudelsonVershynin2013,Adamczak2015}.
Fixed-degree dependent concentration and direct Wick or diagram expansions have order-dependent constants \citep{GotzeSambaleSinulis2019}; decoupling for $U$-statistics is useful for tail comparison but erases the coupled vertex reuse that creates the obstruction in Section~\ref{sec:centering} \citep{DeLaPenaGine1999}.
Asymptotic theory for nonlinear functionals of Gaussian sequences \citep{BreuerMajor1983,Arcones1994} gives central limit theorems at fixed Hermite rank, and Toeplitz-weighted quadratic forms of nonlinear functions of stationary Gaussian sequences have a classical asymptotic theory under summability or regular-variation conditions \citep{TerrinTaqqu1990,GiraitisTaqqu1998}.
The estimator here instead requires a non-asymptotic variance inequality on arbitrary hollow supports, with constants depending only on the correlation operator norm and the local separation, uniformly over chaos order; Proposition~\ref{prop:fusion-frame-main} provides this uniformity.

\FloatBarrier
\section{Operator-norm bounds and coverage lower bound}
\label{sec:estimator}

We now prove operator-norm guarantees for the estimator defined in Section~\ref{sec:sampling-estimator}.
The proof treats the centered sparse-pair fluctuation, the nonlinear inverse-link Taylor error, and the pooled marginal calibration separately.
The contraction theorem controls the first directly and, re-applied lag by lag, the inverse-link Taylor error as well; the pooled marginal calibration is the one separate scalar effect.
The construction turns the centered sparse-pair geometry into a one-bit Toeplitz estimator with pooled marginal plug-in calibration and an operator-norm risk bound, whose leading coverage term is matched by the lower bound of Section~\ref{subsec:lower-bounds}.

We use the estimator and notation of Section~\ref{sec:sampling-estimator}.
The scalar centered link is regular on the buffered inverse-link domain introduced in Assumption~\ref{ass:inverse} below, which also defines the derivative bounds $L_1$ and $L_2$.

\subsection{Upper bounds over natural covariance classes}
We state the main covariance result over standard Toeplitz covariance classes rather than through an abstract calibration certificate.
The constants depend on compact threshold and inverse-link regularity domains but not on $d$, $n$ or the ruler geometry.
Subscripts as in $C_{\tau,\eps}$ record dependence on the threshold window $[\tau_{\min},\tau_{\max}]$ and the buffer $\eps$ of Assumption~\ref{ass:inverse}, uniformly over $\tau$ in the window, not on an individual threshold value.

\begin{assumption}[Buffered regular inverse-link regime]
\label{ass:inverse}
There exist $\eps\in(0,1/4)$ and $0<\tau_{\min}<\tau_{\max}<\infty$ such that
\[
    |\rho_s|\le1-2\eps,\qquad s=1,\ldots,d-1,
\]
and $\tau\in[\tau_{\min},\tau_{\max}]$, while the estimator clips to the larger compact inverse-link domain $\calI_{\tau'}=c([-1+\eps,1-\eps];\tau')$.
Set
\[
    L_k=1\vee\sup\bigl\{|\partial_u^k\psi(u;\tau')|:\
    \tau'\in[\tau_{\min}/2,2\tau_{\max}],\ u\in\calI_{\tau'}\bigr\},
    \qquad k=1,2,
\]
so that $L_1,L_2\ge1$ are uniform inverse-link derivative bounds on the buffered domain.
The two-layer structure (true lags in $[-1+2\eps,1-2\eps]$, clipping at the $\eps$ layer) keeps every Taylor expansion of the plug-in analysis inside the clipping domain; the supplement quantifies the resulting buffer.
\end{assumption}

Let
\[
    \kappa_{\rm obs}=\norm{\Gamma_{\Omega,\Omega}}_2/\gamma_0,
    \qquad
    S_1(d;\rho)=\sum_{s=1}^{d-1}|\rho_s|.
\]
The empirical centered pair fluctuation is governed by $\kappa_{\rm obs}$.
The plug-in population shift is controlled by the short-memory size $S_1(d;\rho)$.

\begin{assumption}[Short-memory Toeplitz class]
\label{ass:short-memory}
The normalized Toeplitz lags satisfy
\[
    S_1(d;\rho)=\sum_{s=1}^{d-1}|\rho_s|\le S_\star,
\]
where $S_\star$ is independent of $d$ and $n$.
\end{assumption}

\begin{theorem}[Sparse-ruler one-bit Toeplitz upper bound]
\label{thm:toeplitz}
Assume Assumption~\ref{ass:inverse}.
Let $t=\log(Cd/\delta)$ and $m=|\Omega|$.
There are constants $C,c,\eta_0>0$, depending only on the regularity domains in Assumption~\ref{ass:inverse}, such that the following hold.

\noindent\textup{(a) Oracle estimator.}
If the scale and sign mean are known, then with probability at least $1-\delta$,
\begin{equation}
    \norm{\widehat \Gamma_{\rm or}-\Gamma}_2
      \le C_\eps\gamma_0L_1\kappa_{\rm obs}
          \sqrt{\frac{\varphi(\Omega)t}{n}}
       + C\gamma_0L_1d\frac{t}{n}
       + C_\eps\gamma_0L_2\left\{
         \min\{\kappa_{\rm obs}^2\varphi(\Omega),d\}\frac{t}{n}
         +d\frac{t^2}{n^2}\right\}.
\label{eq:oracle-rate}
\end{equation}

\noindent\textup{(b) Plug-in estimator.}
In addition assume Assumption~\ref{ass:short-memory}.
Define
\[
    r_\mu=C_\tau\left\{\sqrt{\frac{\kappa_{\rm obs}t}{nm}}+\frac{t}{n}\right\}.
\]
If
\[
    n\ge C\frac{mt}{\kappa_{\rm obs}},
    \qquad
    r_\mu\le c\eta_0,
\]
then with probability at least $1-\delta$,
\begin{align}
    \norm{\widehat \Gamma_{\rm plug}-\Gamma}_2
      &\le C_\eps\gamma_0L_1\kappa_{\rm obs}
          \sqrt{\frac{\varphi(\Omega)t}{n}}
        + C\gamma_0L_1d\frac{t}{n}
        + C_{\tau,\eps}\gamma_0L_2\left\{
          \min\{\kappa_{\rm obs}^2\varphi(\Omega),d\}\frac{t}{n}
          +d\frac{t^2}{n^2}\right\} \nonumber\\
      &\quad
        + C_{\tau,\eps}\gamma_0\{1+S_1(d;\rho)\}r_\mu .
\label{eq:plugin-rate}
\end{align}
In particular, uniformly over the short-memory class $S_1(d;\rho)\le S_\star$, the last term is at most
\[
    C_{\tau,\eps,S_\star}\gamma_0
    \left\{\sqrt{\frac{\kappa_{\rm obs}t}{nm}}+\frac{t}{n}\right\}.
\]
\end{theorem}

\begin{corollary}[Sobolev spectral-density class]
\label{cor:sobolev-main}
Assume the conditions of Theorem~\ref{thm:toeplitz}\textup{(b)}.
If
\[
    \sum_{s=1}^{d-1}s^{2\beta}|\rho_s|^2\le A_\beta^2
\]
for some $\beta>1/2$, then $S_1(d;\rho)\le C_\beta A_\beta$.
Hence the plug-in term in \eqref{eq:plugin-rate} is bounded by
\[
    C_{\tau,\eps,\beta}\gamma_0(1+A_\beta)
    \left\{
    \sqrt{\frac{\kappa_{\rm obs}t}{n|\Omega|}}+\frac{t}{n}
    \right\}.
\]
\end{corollary}

Stable exponential-decay and finite-memory classes follow by the same substitution and are recorded in the supplement.

\begin{remark}[Why short memory enters only the plug-in bound]
\label{rem:short-memory-role}
Part \textup{(a)} requires only Assumption~\ref{ass:inverse}.
Part \textup{(b)} adds Assumption~\ref{ass:short-memory} because pooled calibration perturbs every lag through the nonlinear inverse link, and the resulting lagwise population perturbation is controlled in operator norm by a Toeplitz row-sum bound proportional to $1+S_1(d;\rho)$.
A bounded spectral density controls $\norm{\Gamma}_2/\gamma_0$ but not this perturbation, because a scalar nonlinear transformation of the lag sequence need not preserve a dimension-free Toeplitz spectral norm.
Short-memory and Sobolev-type classes supply the required summability.
On operator-norm balls, however, the row-sum bound can be bypassed altogether; see Corollary~\ref{cor:plugin-ball}.
The two conditions are incomparable: the operator-norm ball admits lag sequences with divergent $S_1(d;\rho)$, since a bounded Toeplitz spectral norm does not force absolutely summable lags, while a short-memory sequence can have operator norm above $c_0$.
\end{remark}

\begin{corollary}[Dropping short memory on operator-norm balls]
\label{cor:plugin-ball}
Assume Assumption~\ref{ass:inverse} and suppose in addition that $\norm{T_d^0(\rho)}_2\le c_0$ for some $c_0\le1-2\eps$, where $T_d^0(\rho)$ is the hollow Toeplitz matrix of the normalized lags; Assumption~\ref{ass:short-memory} is \emph{not} assumed.
Under the sample-size conditions of Theorem~\ref{thm:toeplitz}\textup{(b)}, after decreasing the constant $c$ in the admissibility condition $r_\mu\le c\eta_0$, the plug-in estimator satisfies, with probability at least $1-\delta$, the bound \eqref{eq:plugin-rate} with the calibration term $C_{\tau,\eps}\gamma_0\{1+S_1(d;\rho)\}r_\mu$ replaced by $C_{\tau,\eps}(1+B)\gamma_0 r_\mu$, where $B=B(c_0,\tau_{\min},\tau_{\max})$ is a finite constant independent of $d$, $n$ and the ruler.
\end{corollary}

The mechanism is recorded in the supplement.
When pooled calibration misestimates the scale by a relative error $\eta$, the recovered lag is shifted from $\rho_s$ to $\rho_s+a_\eta(\rho_s)$, where the perturbation $a_\eta$ is determined by the link.
This lagwise perturbation is expanded to first order in $\eta$.
The first-order coefficient is an explicit function of the link, analytic on the unit disk by Plackett's integral representation \citep{Plackett1954}, and its Maclaurin expansion converts the perturbation matrix into a series of entrywise powers $\{T_d^0(\rho)\}^{\circ k}$.
Each entrywise power contracts geometrically, $\norm{\{T_d^0(\rho)\}^{\circ k}}_2\le c_0^k$, because the Hadamard product is a compression of the Kronecker product \citep{Bhatia2007}.
Summing the series gives the constant $B$, and the quadratic remainder is absorbed by the curvature term.
The row-sum bound, and with it the short-memory size, never enters.

\begin{remark}[No positive-semidefinite projection]
\label{rem:no-psd}
The Toeplitz completion produced by link inversion is not projected onto the positive semidefinite cone, and $\widehat\Gamma_{\rm or}$ and $\widehat\Gamma_{\rm plug}$ need not be positive semidefinite, since clipping only keeps each lag estimate in the regular inverse-link range.
The loss is operator norm, so the risk bounds of Theorem~\ref{thm:toeplitz} do not require positive semidefiniteness.
If a positive semidefinite estimate is needed, projection costs at most a factor two.
For any $\widetilde\Gamma$ minimizing $\norm{\widehat\Gamma-M}_2$ over positive semidefinite Toeplitz matrices $M$, the matrix $\Gamma$ itself lies in that cone, so $\norm{\widehat\Gamma-\widetilde\Gamma}_2\le\norm{\widehat\Gamma-\Gamma}_2$ and hence
\[
\norm{\widetilde\Gamma-\Gamma}_2
\le\norm{\widetilde\Gamma-\widehat\Gamma}_2+\norm{\widehat\Gamma-\Gamma}_2
\le2\norm{\widehat\Gamma-\Gamma}_2 .
\]
\end{remark}

Theorem~\ref{thm:toeplitz} gives the same decomposition at the risk level.
The oracle term is controlled by centered pair coverage through $\varphi(\Omega)$.
The curvature of the nonlinear fixed-threshold inverse link enters at the coverage scale $\min\{\kobs^2\varphi(\Omega),d\}t/n+d(t/n)^2$, because the contraction theorem that controls the spectral polynomial also bounds the per-lag deviations; the only remainder at the $d$ scale is the boundedness term $L_1dt/n$ from the spectral Bernstein step.
The plug-in term reflects estimation of $(\gamma_0,\mu,\tau)$ from $n|\Omega|$ marginal bits before these quantities are reused across Toeplitz lags.

The proof in the supplementary material makes this separation explicit.
The oracle part linearizes the inverse link and applies the contraction theorem of Section~\ref{sec:contraction} to a sparse-ruler spectral polynomial, whose Frobenius scale is $\varphi(\Omega)$.
A trigonometric grid and Bernstein inequality convert this one-snapshot variance bound into operator-norm concentration.
Applied lag by lag, the same contraction theorem bounds the sum of squared lag deviations, which places the second-order Taylor term at the coverage scale.
The plug-in part combines pooled marginal calibration, empirical recentering control and a deterministic short-memory row-sum bound.
The supplement records standard substitutions for common short-memory classes.

\subsection{Lower bounds for deterministic coverage complexity}
\label{subsec:lower-bounds}

The lower bound below isolates the design-dependent part of the oracle rate.
Working in a known-scale identity-neighborhood submodel removes marginal calibration and conditioning effects, so the result certifies the intrinsic coverage complexity of deterministic sparse-pair sampling; combined with the upper bounds, it yields minimax rate optimality over this submodel for the known-scale oracle estimator and, over its unknown-scale extension, for the plug-in estimator (Corollaries~\ref{cor:regime-minimax} and~\ref{cor:plugin-minimax}).

In the known-scale real one-bit Toeplitz model, define
\begin{equation}
\mathcal P_{\mathbb R}(c_0)=
\left\{
\Gamma=\gamma_0\{I_d+T_d^0(\rho)\}:
I_d+T_d^0(\rho)\succeq0,\ 
\norm{T_d^0(\rho)}_2\le c_0
\right\},
\label{eq:global-parameter-class}
\end{equation}
where $0<c_0<1/2$ is fixed and observations are restricted to the sparse ruler $\Omega$.
The corresponding minimax risk is
\begin{equation}
\mathfrak R_n(\Omega,\mathcal P_{\mathbb R}(c_0))
=
\inf_{\widehat \Gamma}
\sup_{\Gamma\in\mathcal P_{\mathbb R}(c_0)}
\E_\Gamma\norm{\widehat \Gamma-\Gamma}_2.
\label{eq:minimax-risk}
\end{equation}

\begin{theorem}[Spectral-packing lower bound for the oracle sparse-pair submodel]
\label{thm:spectral-packing-minimax}
Let $\mathcal V=\{b^1,\ldots,b^M\}\subset\mathbb R^{d-1}$ with $M\ge3$, and define
\begin{align}
D_\Omega(\mathcal V)
&=
2\max_{1\le a\le M}\sum_{s=1}^{d-1}q_s(b_s^a)^2,\\
R_T(\mathcal V)
&=
\max_{1\le a\le M}\norm{T_d^0(b^a)}_2,\\
\Delta_T(\mathcal V)
&=
\min_{a\ne a'}\norm{T_d^0(b^a-b^{a'})}_2.
\end{align}
There is a constant $c>0$, depending only on $c_0$, such that
\begin{equation}
\mathfrak R_n(\Omega,\mathcal P_{\mathbb R}(c_0))
\ge
c\gamma_0
\Delta_T(\mathcal V)
\min\left\{
\frac{1}{R_T(\mathcal V)},\,
\sqrt{\frac{\log M}{nD_\Omega(\mathcal V)}}
\right\}.
\label{eq:spectral-packing-bound}
\end{equation}
Consequently, the same lower bound holds after taking the supremum over all finite packings $\mathcal V$.
\end{theorem}

Theorem~\ref{thm:spectral-packing-minimax} is a deterministic spectral-packing principle.
The sparse observation law is controlled by the weighted Frobenius metric $D_\Omega(\mathcal V)$, whereas the loss is governed by the Toeplitz operator separation $\Delta_T(\mathcal V)$.
The bound is well defined for every packing of distinct vectors.
The ruler is complete, so $q_s\ge1$ for all lags, and distinct vectors force $D_\Omega(\mathcal V)$, $R_T(\mathcal V)$ and $\Delta_T(\mathcal V)$ to be strictly positive.
The proof, given in the supplement, combines data processing from the unquantized observed Gaussian vector, a local Gaussian KL bound and Fano's inequality \citep{CoverThomas2006,Tsybakov2009}.

\begin{corollary}[Coverage-log lower bound under balanced real spectral packing]
\label{cor:coverage-log-lower}
Suppose there exist a lag set $S\subset\{1,\ldots,d-1\}$, a frequency set $\Theta\subset[0,1]$ with $|\Theta|=M\ge3$, and constants $a_0,b_0\in(0,1)$ and $\zeta\in(0,1/3)$ such that
\[
\varphi_S(\Omega):=\sum_{s\in S}q_s^{-1}\ge a_0\varphi(\Omega),
\qquad
\Phi_S(\Omega):=\sum_{s\in S}\Bigl(1-\frac sd\Bigr)q_s^{-1}\ge b_0\varphi_S(\Omega),
\]
and such that, with the taper kernel
\[
K_S(u)=\sum_{s\in S}\Bigl(1-\frac sd\Bigr)q_s^{-1}e^{2\pi i s u},
\qquad u\in[0,1],
\]
the coherence bounds
\[
\max_{\theta\in\Theta}\frac{|K_S(2\theta)|}{\Phi_S(\Omega)}\le\zeta,
\qquad
\max_{\theta\ne\theta'\in\Theta}
\max\left\{
\frac{|K_S(\theta-\theta')|}{\Phi_S(\Omega)},\,
\frac{|K_S(\theta+\theta')|}{\Phi_S(\Omega)}
\right\}\le\zeta
\]
hold.
Then
\begin{equation}
\mathfrak R_n(\Omega,\mathcal P_{\mathbb R}(c_0))
\ge
c\gamma_0
\min\left\{
1,\,
\sqrt{\frac{\varphi(\Omega)\log M}{n}}
\right\},
\label{eq:coverage-log-lower}
\end{equation}
where $c>0$ depends only on $a_0,b_0,\zeta$ and $c_0$.
In particular, if $M\ge d^\alpha$ for a fixed $\alpha>0$, then
\begin{equation}
\mathfrak R_n(\Omega,\mathcal P_{\mathbb R}(c_0))
\ge
c\gamma_0
\min\left\{
1,\,
\sqrt{\frac{\varphi(\Omega)\log d}{n}}
\right\}
\label{eq:coverage-logd-lower}
\end{equation}
where in this display $c=c(a_0,b_0,\zeta,c_0,\alpha)>0$.
\end{corollary}

Comparing Corollary~\ref{cor:coverage-log-lower} with Theorem~\ref{thm:toeplitz}\textup{(a)} matches the dependence on $\varphi(\Omega)$, $n$ and $\log d$ in the leading oracle term.
On the lower-bound class this comparison sharpens to minimax rate optimality of the coverage rate.

\begin{corollary}[Regime-restricted oracle minimax rate]
\label{cor:regime-minimax}
Fix $c_0\in(0,1/2)$ and suppose that the balanced real spectral-packing condition of Corollary~\ref{cor:coverage-log-lower} holds with $M\ge d^\alpha$ for some fixed $\alpha>0$.
There is a constant $C$, depending only on $c_0$, $\alpha$, the threshold window in Assumption~\ref{ass:inverse} and the packing constants $a_0,b_0,\zeta$, such that, whenever
\[
    n\ge C
    \max\left\{\varphi(\Omega),\,\frac{d^2}{\varphi(\Omega)}\right\}\log d,
\]
the known-scale oracle minimax risk \eqref{eq:minimax-risk} over the identity-neighborhood class obeys
\[
    \mathfrak R_n(\Omega,\mathcal P_{\mathbb R}(c_0))
    \asymp
    \gamma_0\sqrt{\frac{\varphi(\Omega)\log d}{n}},
\]
with implied constants of the same dependence, and the oracle sparse-ruler estimator attains this rate.
\end{corollary}

The lower bound is Corollary~\ref{cor:coverage-log-lower}; its minimum equals the square-root branch because the stated regime implies $\varphi(\Omega)\log d/n\le C^{-1}$.
For the upper bound, every $\Gamma=\gamma_0\{I_d+T_d^0(\rho)\}$ in $\mathcal P_{\mathbb R}(c_0)$ has $|\rho_s|\le\norm{T_d^0(\rho)}_2\le c_0$, so Assumption~\ref{ass:inverse} holds with $\eps=(1-c_0)/4\in(1/8,1/4)$, since $c_0\le1-2\eps=(1+c_0)/2$, and $L_1,L_2$ are bounded in terms of $c_0$ and the threshold domain; and since $\Gamma\succeq0$ has diagonal $\gamma_0$, $1\le\kappa_{\rm obs}\le1+c_0$.
The oracle bound \eqref{eq:oracle-rate} does not involve $S_1(d;\rho)$ and follows from Assumption~\ref{ass:inverse} alone---the short-memory size enters Theorem~\ref{thm:toeplitz} only through the plug-in term.
Integrating the high-probability bound over $\delta$, using the bounded regularity domain and the clipped inverse-link estimator, gives $\sup_{\Gamma\in\mathcal P_{\mathbb R}(c_0)}\E_\Gamma\norm{\widehat\Gamma_{\rm or}-\Gamma}_2\lesssim_{c_0}\gamma_0\sqrt{\varphi(\Omega)\log d/n}+\gamma_0 d\log d/n+\gamma_0\varphi(\Omega)\log d/n+\gamma_0 d\log^2d/n^2$.
The last three terms are dominated in the stated regime: the boundedness terms because $n\gtrsim d^2\log d/\varphi(\Omega)$, and the curvature term because $n\gtrsim\varphi(\Omega)\log d$, the factor $\kobs^2\le(1+c_0)^2$ being a constant on this class.
The integration step is recorded as a separate lemma in the supplement.

The minimax statement extends from the known-scale oracle estimator to the plug-in estimator itself, on the unknown-scale version of the same class.

\begin{corollary}[Regime-restricted plug-in minimax rate, unknown scale]
\label{cor:plugin-minimax}
Fix $c_0\in(0,1/2)$ and $0<\tau_{\min}<\tau_{\max}<\infty$, let $\Lambda=[\lambda^2/\tau_{\max}^2,\,\lambda^2/\tau_{\min}^2]$ be the induced scale window with $\gamma_{\max}=\max\Lambda$, and define
\[
\mathcal Q
=\bigl\{\gamma\{I_d+T_d^0(\rho)\}:\ \gamma\in\Lambda,\
I_d+T_d^0(\rho)\succeq0,\ \norm{T_d^0(\rho)}_2\le c_0\bigr\}.
\]
Suppose the balanced real spectral-packing condition of Corollary~\ref{cor:coverage-log-lower} holds with $M\ge d^\alpha$ for some fixed $\alpha>0$.
There is a constant $C_\star$, depending only on $c_0$, $\alpha$, the threshold window and the packing constants, such that, for all $d\ge2$, whenever
\[
    n\ge C_\star
    \max\left\{\varphi(\Omega),\,\frac{d^2}{\varphi(\Omega)}\right\}\log d,
\]
the minimax risk over $\mathcal Q$ satisfies
\[
    \inf_{\widehat\Gamma}\sup_{\Gamma\in\mathcal Q}
    \E_\Gamma\norm{\widehat\Gamma-\Gamma}_2
    \asymp
    \gamma_{\max}\sqrt{\frac{\varphi(\Omega)\log d}{n}},
\]
and the plug-in estimator attains this rate.
The threshold $\lambda$ enters only through the scale window $\Lambda$.
\end{corollary}

The proof is given in the supplement.
The lower bound is inherited from Corollary~\ref{cor:coverage-log-lower} on the known-scale slice at $\gamma_{\max}$.
The upper bound combines Corollary~\ref{cor:plugin-ball} with a truncated integration to expectation.
The sample-size conditions of Theorem~\ref{thm:toeplitz}\textup{(b)} depend on the confidence level $t$, so the tail integral is cut at $T_*\asymp n/m$, where both conditions hold.
The remaining tail is controlled by the deterministic bound for the clipped estimator.

\begin{remark}[Quantization costs only constants in the balanced regime]
\label{rem:quantization-free}
The lower bound \eqref{eq:coverage-logd-lower} is proved by data processing from the unquantized Gaussian observations on $\Omega$, so it certifies that no estimator, even with access to the full-precision sparse-ruler samples, can improve on the coverage rate.
Since the one-bit oracle estimator attains this rate in the regime of Corollary~\ref{cor:regime-minimax}, and the plug-in estimator attains it with unknown scale (Corollary~\ref{cor:plugin-minimax}), fixed-threshold one-bit quantization costs only constant factors there.
The argument does not yield a channel-intrinsic one-bit lower bound, and it leaves open whether quantization incurs additional losses outside this regime.
\end{remark}

The supplement gives concrete sufficient conditions for the balanced packing.
For example, its effective-support certificate shows that it suffices for a constant fraction of the inverse-coverage mass to be spread over $d^\alpha$ nonboundary lags with bounded effective concentration.
These conditions exclude pathological profiles whose inverse-coverage mass is concentrated near aperture-boundary lags, where Toeplitz operator separation is weak.

The following consequence, proved in the supplement, makes the condition checkable for standard designs.

\begin{corollary}[Bounded-redundancy rulers]
\label{cor:bounded-redundancy}
Let $\Omega$ cover all lags and define its redundancy
\[
    R(\Omega)=\binom{m}{2}\big/(d-1),
    \qquad m=|\Omega| .
\]
Fix $\zeta\in(0,1/3)$.
If $R(\Omega)\le\bar R$ and $d\ge C(\bar R,\zeta)$, then the balanced packing condition of Corollary~\ref{cor:coverage-log-lower} holds with $M\ge c(\bar R,\zeta)\,d$, and hence Corollary~\ref{cor:regime-minimax} applies with $\alpha=1$ and constants depending only on $\bar R$, $\zeta$ and $c_0$.
\end{corollary}

Minimum-redundancy rulers have $R(\Omega)$ bounded by an absolute constant \citep{Moffet1968,Leech1956}, and two-level nested rulers with balanced level sizes have $R(\Omega)\le2$ \citep{PalVaidyanathan2010}; both canonical families therefore satisfy the balanced packing condition.

\section{Discussion}
\label{sec:discussion}

At a nonzero threshold, raw products leak marginal bit fluctuations into all incident lag products and produce a row-sum variance term.
Centering removes this leakage and reveals the coverage complexity \(\varphi(\Omega)\) that governs the leading operator-norm risk.
Under balanced deterministic coverage geometry, the coverage term in the oracle operator-norm rate is intrinsic.
The plug-in analysis further shows that marginal scale calibration is compatible with the same leading pair-estimation rate over short-memory and Sobolev spectral-density classes, and, on operator-norm balls, without any summability assumption.
The theory therefore separates three effects: deterministic pair reuse, scalar inverse-link curvature and marginal scale learning.
In the non-saturated operating regime \(n\gtrsim\max\{\varphi(\Omega),d^2/\varphi(\Omega)\}\log d\) the oracle and plug-in estimators are minimax rate optimal over the identity-neighborhood class (Corollaries~\ref{cor:regime-minimax} and~\ref{cor:plugin-minimax}).
The supplement reports numerical checks of these rates.

Two quantitative gaps remain open.
The conditioning factor \(\kobs\) in the leading upper-bound term comes from the \(\kappa^2\) variance bound of Theorem~\ref{thm:contraction} applied to the observed correlation matrix, whereas the lower bound carries no such factor; the two match only because \(\kobs\le1+c_0\) on the identity-neighborhood class.
Whether the linear \(\kobs\) dependence is sharp over wider Toeplitz classes, or improvable to \(\sqrt{\kobs}\) or a weighted spectral-density functional, is open.
The inverse-link curvature remainder of Theorem~\ref{thm:toeplitz} sits at the coverage scale and the plug-in calibration remainder is dominated in this regime, so the only remainder above the coverage scale is the boundedness term \(\gamma_0L_1dt/n\) from the spectral Bernstein step.
On dense rulers this term drives the regime restriction; weakening the regime of Corollary~\ref{cor:regime-minimax} requires either improving this term or showing that it is unavoidable.

Different parts of the argument use Gaussianity in different ways.
The vertex-projection obstruction is a second-moment identity that needs only independent coordinates and mean-shifted bounded signs.
The link inversion needs an identifiable, strictly increasing bivariate sign-correlation link with a stable inverse, which is available for some non-Gaussian scatter models, for instance a centered bivariate Cauchy \citep[cf.][]{LinLiu2026}.
The contraction theorem is the genuinely Gaussian component, since its proof runs through Hermite expansions, the Mehler operator, and the chaos decomposition, and extending it beyond Gaussian snapshots is open.
Threshold mismatch and unknown thresholds \citep[cf.][]{SteinBarNossekTabrikian2018} are partially covered.
Remark~\ref{rem:heterogeneous} supplies the one-snapshot variance input for coordinate-dependent thresholds, while the corresponding estimator theory is not developed here.

The fixed one-bit channel is the extreme member of a family of fixed quantized channels to which the variance input of Theorem~\ref{thm:contraction} extends.
For any monotone quantizer \(Q\) with bounded output range, the centered transform \(Q(\cdot)-\E Q(G)\) is bounded and centered, so Theorem~\ref{thm:contraction} gives the same edge-Frobenius variance bound, with a constant depending on the output range but not on the number of levels.
The lower bound is proved from the unquantized observations and is channel-independent.
Decomposing \(Q\) into its jumps \(w_r>0\) at normalized thresholds \(\tau_r\) and applying Plackett's identity \citep{Plackett1954} gives
\[
\partial_\rho\Cov\{Q(G_1),Q(G_2)\}=\sum_{r,r'}w_r w_{r'}\,\phi_2(\tau_r,\tau_{r'};\rho)>0,
\]
with \(\phi_2\) the bivariate standard normal density, so the link is strictly increasing and, since every normalized threshold depends on the scale only through \(\gamma_0^{-1/2}\), calibration again reduces to the single parameter \(\gamma_0\).
The obstruction itself is a property of the channel mean, present exactly when the quantizer output is mean-shifted, so a symmetric one-bit channel avoids it but cannot recover scale.
Because Theorem~\ref{thm:contraction} and Proposition~\ref{prop:fusion-frame-main} place no Toeplitz restriction on the support, the same variance input applies to the multilevel and block-Toeplitz lag sets of planar sparse arrays; only the aggregation and grid arguments here are one-dimensional.
A detailed treatment of fixed few-bit quantizers, long-memory spectral classes, and adaptive thresholds is left to future work.

\FloatBarrier

\section*{Acknowledgments}
This work was supported by the Research Startup Fund of Chaohu University under Grant KYQD2024015.

\clearpage
{\centering\Large\bfseries Supplementary Material\par}
\medskip
\setcounter{section}{0}
\renewcommand{\thesection}{S\arabic{section}}

\paragraph{Terminology and proof order}
We use the terminology of the main paper.
The marginal bit budget is the $n|\Omega|$ one-coordinate signs used for plug-in calibration, while virtual pair coverage is the lag profile $q_s$ and the coefficient $\varphi(\Omega)=\sum_s q_s^{-1}$.

The proof of the main upper bound starts with the pairwise projection and fusion-frame estimates that give the contraction theorem.
That theorem supplies the one-snapshot variance input for oracle spectral concentration.
The plug-in proof then adds marginal-bit calibration, empirical recentering control and a deterministic population perturbation bound.
The lower bound is separate; it uses data processing, Gaussian KL control and deterministic Toeplitz spectral packing.

\section{Proof of the centered sparse-pair contraction theorem}
\label{sec:supp-contraction}

This section proves the contraction theorem from the main text.
The theorem is used only after the one-bit signs have been centered.
Its purpose is to show that, on an arbitrary deterministic hollow support, the variance of the centered sparse-pair statistic is controlled by the Frobenius norm of the edge weights and not by maximum degree or weighted row sums.
The constants may depend on the compact support-separation parameter $\varepsilon$ and on the bounded transform $h$, but not on dimension, support size, maximum degree or Hermite order.

The proof has two components.
First, a pairwise Hoeffding projection removes the one-coordinate parts of $h(g_i)h(g_j)-\E h(g_i)h(g_j)$; because $h$ is centered, the one-vertex projection is of order $|C_{ij}|$ and can be summed by Schur--Hadamard contractions.
Second, the remaining pure two-coordinate chaoses lie in degenerate pair spaces.
A fusion-frame Bessel bound controls the synthesis over all active edges uniformly over chaos order.
Combining these two estimates gives the contraction theorem.

Throughout this section, $C$ is a real Gaussian correlation matrix with $\norm{C}_2=\kappa$, $h:\R\to\R$ is real-valued, and $g_i=W(e_i)$ denotes the associated isonormal representation.
The support edges satisfy $|C_{ij}|\le1-\eps$ whenever the edge $(i,j)$ is active.
Constants denoted $C_\eps$ and $K_{h,\eps}$ may depend polynomially on $\eps^{-1}$ and on fixed bounds for $h$, but not on dimension, support size, maximum degree or chaos order.
\subsection*{Gaussian Hilbert space and Fock normalization}
We use standard Gaussian Hilbert-space and Wiener-chaos notation \citep{Janson1997,Nualart2006,PeccatiTaqqu2011}.
Let $H$ be the real Gaussian Hilbert space generated by $g_1,\ldots,g_m$.
Thus $g_i=W(e_i)$ for an isonormal Gaussian process $W$ over $H$, with
\begin{equation}
  \ip{e_i}{e_j}_H=C_{ij},\qquad \norm{e_i}_H=1.
\end{equation}
For $N\ge0$, let $H^{\odot N}$ denote the $N$-fold symmetric tensor space.
We use the homogeneous Fock norm
\begin{equation}
  \norm{F}_{(N)}^2=N!\norm{F}_{H^{\odot N}}^2.
\end{equation}
With this normalization, the multiple integral map $I_N$ is an isometry from $(H^{\odot N},\norm{\cdot}_{(N)})$ into the $N$th Wiener chaos:
\begin{equation}
  \E|I_N(F)|^2=\norm{F}_{(N)}^2.
\end{equation}
For $u\in H$, define creation and annihilation operators
\begin{equation}
  c(u):H^{\odot(N-1)}\to H^{\odot N},\qquad c(u)U=u\vee U,
\end{equation}
and let $a(u)=c(u)^*$ with respect to the Fock norms.
If $\widehat e_i^{(N)}$ is the unit vector proportional to $e_i^{\odot N}$ in the $N$th Fock norm, then
\begin{equation}
  a(e_i)\widehat e_j^{(N)}=\sqrt{N}C_{ij}\widehat e_j^{(N-1)}.
\end{equation}

\subsection*{Schur--Hadamard tools}
We use two elementary consequences of positivity under Schur products; see, for example, \citet{Bhatia2007}.
\begin{lemma}[Schur contractions for Gaussian correlations]
\label{supp-lem:schur}
For every integer $r\ge1$,
\begin{equation}
  \norm{C^{\circ r}}_2\le\kappa.
\end{equation}
Moreover, for every row index $i$,
\begin{equation}
  \sum_j |C_{ij}|^{2r}\le\kappa.
\end{equation}
Consequently, writing $\one$ for the all-ones vector, for every matrix $B$ and every $r\ge1$,
\begin{equation}
  \norm{(B\circ C^{\circ r})\one}_2^2\le\kappa\norm{B}_F^2.
\end{equation}
\end{lemma}

\begin{proof}
The Schur product theorem gives $C^{\circ r}\succeq0$.
Since $C^{\circ(r-1)}$ is a correlation matrix, the Schur multiplier $M\mapsto M\circ C^{\circ(r-1)}$ is unital completely positive and hence an operator-norm contraction on Hermitian matrices.
Therefore
\begin{equation}
  \norm{C^{\circ r}}_2=\norm{C\circ C^{\circ(r-1)}}_2\le\norm{C}_2=\kappa.
\end{equation}
Let $D=C^{\circ r}$.
Since $0\preceq D\preceq\norm{D}_2I$ and $D_{ii}=1$,
\begin{equation}
  \sum_j |C_{ij}|^{2r}=(D^2)_{ii}\le\norm{D}_2D_{ii}\le\kappa.
\end{equation}
Finally, Cauchy--Schwarz row by row gives
\begin{align}
  \norm{(B\circ C^{\circ r})\one}_2^2
  &=\sum_i\left|\sum_jB_{ij}C_{ij}^r\right|^2\\
  &\le\sum_i\left(\sum_j|B_{ij}|^2\right)\left(\sum_j|C_{ij}|^{2r}\right)\\
  &\le\kappa\norm{B}_F^2.
\end{align}
\end{proof}

\begin{lemma}[Heterogeneous one-coordinate Gaussian contraction]
\label{supp-lem:centered-transform-covariance}
Let $g_\Omega\sim N(0,C_\Omega)$ be a standard Gaussian vector with $\norm{C_\Omega}_2=\kappa$.
For each $i\in\Omega$ let $\phi_i\in L^2(\gamma)$, where $\gamma$ is the standard $N(0,1)$ law, satisfy $\E\phi_i(G)=0$, with orthonormal Hermite expansion $\phi_i=\sum_{q\ge1}a_{i,q}H_q$; write $a_q=(a_{i,q})_{i\in\Omega}$ and set $v_i=\phi_i(g_i)$.
Then for every $b\in\mathbb{C}^{\Omega}$,
\begin{equation}
  \Var\Big(\sum_{i\in\Omega}b_iv_i\Big)
  =\sum_{q\ge1}(b\circ a_q)^*C_\Omega^{\circ q}(b\circ a_q)
  \le \kappa\sum_{i\in\Omega}|b_i|^2\norm{\phi_i}_{L^2(\gamma)}^2 ,
  \label{supp-eq:centered-transform-covariance}
\end{equation}
where $\circ$ denotes the entrywise product.
In particular, taking all $\phi_i=h$ gives $\Cov(v_\Omega)=\sum_{q\ge1}a_q^2C_\Omega^{\circ q}$ and $\norm{\Cov(v_\Omega)}_2\le\norm{h}_{L^2(\gamma)}^2\kappa$; for centered threshold signs $h_\tau(t)=\sign(t-\tau)-\E\sign(G-\tau)$ this yields $\norm{\Cov(h_\tau(g_i))_{i\in\Omega}}_2\le4\kappa$.
\end{lemma}

\begin{proof}
For standard Gaussian coordinates, Hermite orthogonality gives $\E H_q(g_i)H_r(g_j)=\mathbf 1_{\{q=r\}}(C_\Omega)_{ij}^q$, so $\Cov(v_i,v_j)=\sum_{q\ge1}a_{i,q}a_{j,q}(C_\Omega)_{ij}^q$.
Hence, for any $b\in\mathbb{C}^{\Omega}$,
\[
  \Var\Big(\sum_{i\in\Omega}b_iv_i\Big)
  =\sum_{i,j}b_i\overline{b_j}\,\Cov(v_i,v_j)
  =\sum_{q\ge1}(b\circ a_q)^*C_\Omega^{\circ q}(b\circ a_q),
\]
the rearrangement being justified by $\sum_{q\ge1}\norm{a_q}_2^2=\sum_{i\in\Omega}\norm{\phi_i}_2^2<\infty$.
Each Schur power $C_\Omega^{\circ q}$ is positive semidefinite with $\norm{C_\Omega^{\circ q}}_2\le\kappa$ by Lemma~\ref{supp-lem:schur}, so
\[
  (b\circ a_q)^*C_\Omega^{\circ q}(b\circ a_q)
  \le\kappa\norm{b\circ a_q}_2^2
  =\kappa\sum_{i\in\Omega}|b_i|^2a_{i,q}^2 .
\]
Summing over $q$ and using $\sum_{q\ge1}a_{i,q}^2=\norm{\phi_i}_2^2$ gives \eqref{supp-eq:centered-transform-covariance}.
Taking all $\phi_i=h$ identifies $\Cov(v_\Omega)=\sum_{q\ge1}a_q^2C_\Omega^{\circ q}$ as the matrix of this quadratic form, and the threshold-sign bound follows from $\norm{h_\tau}_2\le\norm{h_\tau}_\infty\le2$.
\end{proof}

The bound \eqref{supp-eq:centered-transform-covariance} controls only additive one-coordinate statistics $\sum_i b_i\phi_i(g_i)$.
It does not bound the degenerate pair interactions $\sum_{i,j}A_{ij}\eta_{C_{ij}}^{\rm int}(g_i,g_j)$ of the contraction theorem, whose mixed chaos components do not lie in the pure one-coordinate Hermite subspaces; controlling those uniformly in the chaos order requires the fusion-frame estimate of Proposition~\ref{supp-prop:fusion-frame}.

\subsection*{Pairwise Hoeffding projection}
Let $(G_1,G_2)$ be standard bivariate Gaussian with correlation $\rho$, where $|\rho|\le1-\eps$.
Define
\begin{equation}
  \eta_\rho(G_1,G_2)=h(G_1)h(G_2)-\E[h(G_1)h(G_2)].
\end{equation}
Let $H_1=L_0^2(G_1)$ and $H_2=L_0^2(G_2)$ be the centered one-coordinate subspaces.
Let $\Pi_\rho$ be the orthogonal projection onto $H_1+H_2$, and write
\begin{equation}
  \Pi_\rho\eta_\rho=f_\rho(G_1)+\widetilde f_\rho(G_2),
  \qquad
  \eta_\rho^{\rm int}=\eta_\rho-\Pi_\rho\eta_\rho.
\end{equation}
Because $|\rho|<1$, $H_1\cap H_2=\{0\}$ inside centered $L^2$, so the representation is unique.

\begin{lemma}[Pairwise Hoeffding projection]
\label{supp-lem:pair-projection}
For $|\rho|\le1-\eps$,
\begin{equation}
  \norm{f_\rho}_2^2+\norm{\widetilde f_\rho}_2^2\le K_{h,\eps}\rho^2,
\end{equation}
and
\begin{equation}
  \norm{\eta_\rho^{\rm int}}_2\le2\norm{h}_\infty^2.
\end{equation}
\end{lemma}

The factor $|\rho|$ is the key gain.
At independence the one-vertex projection vanishes; hence centered pair products do not create the coherent row-sum fluctuation seen for raw nonzero-threshold products.

\begin{proof}
Let $T_\rho$ be the Mehler operator.
Since $h$ is centered, if $h=\sum_{q\ge1}c_qH_q$ in the orthonormal Hermite basis, then
\begin{equation}
  \norm{T_\rho h}_2^2=\sum_{q\ge1}\rho^{2q}c_q^2\le\rho^2\norm{h}_2^2.
\end{equation}
The conditional expectation of $\eta_\rho$ given $G_1$ is
\begin{equation}
  a_\rho(G_1)=h(G_1)T_\rho h(G_1)-\langle h,T_\rho h\rangle.
\end{equation}
Thus
\begin{equation}
  \norm{a_\rho}_2\le \norm{h}_\infty\norm{T_\rho h}_2+\norm{h}_2\norm{T_\rho h}_2\le K_h|\rho|.
\end{equation}
The analogous conditional expectation $b_\rho(G_2)$ satisfies the same bound.
The normal equations for the projection are
\begin{equation}
  f_\rho+T_\rho\widetilde f_\rho=a_\rho,
  \qquad
  T_\rho f_\rho+\widetilde f_\rho=b_\rho.
\end{equation}
On centered $L^2(\gamma)$, $\norm{T_\rho}\le|\rho|\le1-\eps$, so
\begin{equation}
  \norm{(I-T_\rho^2)^{-1}}\le(1-\rho^2)^{-1}\le C_\eps.
\end{equation}
Solving gives $\norm{f_\rho}_2+\norm{\widetilde f_\rho}_2\le K_{h,\eps}|\rho|$.
The second claim follows because orthogonal projection is contractive and
\begin{equation}
  \norm{\eta_\rho}_2\le\norm{h(G_1)h(G_2)}_2+|\E h(G_1)h(G_2)|\le2\norm{h}_\infty^2.
\end{equation}
\end{proof}

\subsection*{One-vertex aggregate}
For a hollow real symmetric matrix $A$, define
\begin{equation}
  Q(A)=v^\top Av-\E[v^\top Av]=\sum_{i,j}A_{ij}\eta_{C_{ij}}(g_i,g_j).
\end{equation}
Use the pairwise projection to decompose
\begin{equation}
  Q(A)=Q^{(1)}(A)+Q^{(2)}(A),
\end{equation}
where
\begin{align}
  Q^{(1)}(A)&=\sum_{i,j}A_{ij}\{f_{C_{ij}}(g_i)+\widetilde f_{C_{ij}}(g_j)\},\\
  Q^{(2)}(A)&=\sum_{i,j}A_{ij}\eta_{C_{ij}}^{\rm int}(g_i,g_j).
\end{align}

\begin{lemma}[Aggregate bound for the one-vertex projection]
\label{supp-lem:one-vertex}
Under the support regularity $A_{ij}\neq0\Rightarrow |C_{ij}|\le1-\eps$,
\begin{equation}
  \Var(Q^{(1)}(A))\le K_{h,\eps}\kappa^2\norm{A}_F^2.
\end{equation}
\end{lemma}

\begin{proof}
Collect all one-coordinate terms depending on $g_i$:
\begin{equation}
  Q^{(1)}(A)=\sum_iF_i(g_i),
\end{equation}
where
\begin{equation}
  F_i=\sum_j A_{ij}f_{C_{ij}}+\sum_j A_{ji}\widetilde f_{C_{ji}}.
\end{equation}
The two sums are estimated identically; consider $\sum_j A_{ij}f_{C_{ij}}$.
By Lemma~\ref{supp-lem:pair-projection}, $\norm{f_{C_{ij}}}_2\le K_{h,\eps}^{1/2}|C_{ij}|$, so the triangle inequality followed by Cauchy--Schwarz over $j$ gives
\begin{equation}
  \norm{\sum_j A_{ij}f_{C_{ij}}}_2
  \le\sum_j|A_{ij}|\,\norm{f_{C_{ij}}}_2
  \le K_{h,\eps}^{1/2}\Big(\sum_j|A_{ij}|^2\Big)^{1/2}\Big(\sum_j|C_{ij}|^2\Big)^{1/2}.
\end{equation}
The companion sum $\sum_j A_{ji}\widetilde f_{C_{ji}}$ obeys the same bound, so
\begin{equation}
  \norm{F_i}_2^2\le K_{h,\eps}
  \left(\sum_j|A_{ij}|^2\right)
  \left(\sum_j|C_{ij}|^2\right)
  \le K_{h,\eps}\kappa\sum_j|A_{ij}|^2,
\end{equation}
where Lemma~\ref{supp-lem:schur} was used in the last step.
Each $F_i$ is a centered function of the single coordinate $g_i$, so Lemma~\ref{supp-lem:centered-transform-covariance}, applied with $\phi_i=F_i$ and $b_i=1$, gives
\begin{equation}
  \Var(Q^{(1)}(A))=\Var\Big(\sum_iF_i(g_i)\Big)
  \le\kappa\sum_i\norm{F_i}_2^2
  \le K_{h,\eps}\kappa^2\norm{A}_F^2,
\end{equation}
where the last step uses the bound $\norm{F_i}_2^2\le K_{h,\eps}\kappa\sum_j|A_{ij}|^2$ established above together with $\sum_i\sum_j|A_{ij}|^2=\norm{A}_F^2$.
\end{proof}

\subsection*{Degenerate pair-chaos fusion frame}
For $N\ge2$ and an ordered support edge $(i,j)$, define
\begin{equation}
  K_{ij,N}=\Span\{e_i,e_j\}^{\odot N}\ominus
  \Span\{e_i^{\odot N},e_j^{\odot N}\}
\end{equation}
inside the $N$th homogeneous Fock space.
Let $P_{ij,N}^{\rm int}$ be the orthogonal projection onto $K_{ij,N}$.

The subtraction of the two pure tensors is the geometric point of the argument.
The full two-coordinate space $\Span\{e_i,e_j\}^{\odot N}$ contains directions concentrated on a single vertex, and these directions can add coherently over many incident edges.
Those coherent one-vertex modes are exactly the modes already controlled by the Hoeffding projection.
After they are removed, every remaining vector contains a genuine residual component in the direction of $e_j$ orthogonal to $e_i$, or symmetrically in the direction of $e_i$ orthogonal to $e_j$.
The local separation $|C_{ij}|\le1-\eps$ makes these residual directions uniformly well conditioned.
The proof below turns this geometry into a Bessel bound by applying creation--annihilation estimates and Schur contractions to the residual Gram matrices.

\begin{proposition}[Fusion-frame bound for degenerate pair chaoses]
\label{supp-prop:fusion-frame}
Let $E$ be any finite ordered edge set satisfying
\begin{equation}
  (i,j)\in E\quad\Longrightarrow\quad |C_{ij}|\le1-\eps.
\end{equation}
Then, for every $N\ge2$ and every $F\in H^{\odot N}$,
\begin{equation}
  \sum_{(i,j)\in E}\norm{P_{ij,N}^{\rm int}F}_{(N)}^2
  \le C_\eps\kappa^2\norm{F}_{(N)}^2.
  \label{supp-eq:fusion-analysis}
\end{equation}
Equivalently, by Hilbert-space duality, for all $f_{ij}\in K_{ij,N}$,
\begin{equation}
  \norm{\sum_{(i,j)\in E} f_{ij}}_{(N)}^2
  \le C_\eps\kappa^2\sum_{(i,j)\in E}\norm{f_{ij}}_{(N)}^2.
  \label{supp-eq:fusion-synthesis}
\end{equation}
The same estimate with coefficients $A_{ij}$ follows by replacing $f_{ij}$ with $A_{ij}f_{ij}$; for complex $A_{ij}$, applying the real estimate separately to the real and imaginary parts of the coefficients gives \eqref{supp-eq:fusion-synthesis} on the complexified chaos with the same constant.
\end{proposition}

All Hilbert spaces here are real; complex coefficients enter only through the complexification just described, and in the proof of the contraction theorem a Hermitian $A$ is in any case reduced to its real symmetric part first.
The estimate is uniform in the chaos order $N$.
The apparent factor $N$ from creation--annihilation bounds is canceled by the $1/N$ lower bound from the number operator on the two-coordinate chaos; this cancellation is why the final theorem has no dependence on the Hermite order.

\begin{lemma}[Creation synthesis bound]
\label{supp-lem:creation}
For arbitrary $U_i\in H^{\odot(N-1)}$,
\begin{equation}
  \norm{\sum_i c(e_i)U_i}_{(N)}^2
  \le N\kappa\sum_i\norm{U_i}_{(N-1)}^2.
\end{equation}
\end{lemma}

\begin{proof}
Let $B\{U_i\}=\sum_i c(e_i)U_i$.
Its adjoint is $B^*F=\{a(e_i)F\}_i$.
Hence
\begin{equation}
  \norm{B}^2=\norm{B^*}^2
  =\sup_{\norm{F}_{(N)}=1}\sum_i\norm{a(e_i)F}_{(N-1)}^2.
\end{equation}
The last sum is $\langle F,d\Gamma(S)F\rangle_{(N)}$, where $d\Gamma(\cdot)$ is the second-quantization (number) operator on Fock space and $S=\sum_i e_i\otimes e_i$ is the frame operator of the family $\{e_i\}$.
The operator norm of $S$ is $\norm{C}_2=\kappa$, so $d\Gamma(S)\preceq N\kappa I$ on the $N$th homogeneous Fock space.
\end{proof}

For a fixed edge $(i,j)$, set
\begin{equation}
  r_{j|i}=\frac{e_j-C_{ij}e_i}{(1-C_{ij}^2)^{1/2}}.
\end{equation}
The support condition gives $1-C_{ij}^2\ge c_\eps>0$.
Define the one-sided residual subspace
\begin{equation}
  R_{j|i,N-1}=\Span\{e_i,r_{j|i}\}^{\odot(N-1)}\ominus
  \Span\{e_i^{\odot(N-1)}\}.
\end{equation}

\begin{lemma}[One-sided residual synthesis]
\label{supp-lem:one-sided}
For each fixed $i$ and arbitrary $U_{ij}\in R_{j|i,N-1}$,
\begin{equation}
  \norm{\sum_j U_{ij}}_{(N-1)}^2
  \le C_\eps\kappa\sum_j\norm{U_{ij}}_{(N-1)}^2.
\end{equation}
\end{lemma}

\begin{proof}
Write $n=N-1$.
By construction $\ip{r_{j|i}}{e_i}_H=0$, $\norm{r_{j|i}}_H=1$ and
\begin{equation}
  \ip{r_{j|i}}{r_{k|i}}_H
  =\frac{C_{jk}-C_{ij}C_{ik}}{(1-C_{ij}^2)^{1/2}(1-C_{ik}^2)^{1/2}}
  =:(R_i)_{jk},
\end{equation}
the correlation matrix of the residual vectors.
Under the decomposition $H=\Span\{e_i\}\oplus e_i^\perp$, the symmetric power splits orthogonally as $H^{\odot n}=\bigoplus_{b=0}^{n}e_i^{\odot(n-b)}\vee(e_i^\perp)^{\odot b}$, and accordingly
\begin{equation}
  R_{j|i,n}=\bigoplus_{b=1}^{n}
  \Span\{e_i^{\odot(n-b)}\vee r_{j|i}^{\odot b}\}.
\end{equation}
We first record the Fock normalization.
If $u\perp x$, $u\perp y$ and $\norm{u}_H=\norm{x}_H=\norm{y}_H=1$, then in the homogeneous Fock norm
\begin{equation}
  \ip{u^{\odot(n-b)}\vee x^{\odot b}}{u^{\odot(n-b)}\vee y^{\odot b}}_{(n)}
  =(n-b)!\,b!\,\ip{x}{y}_H^{\,b},
  \label{supp-eq:fock-normalization}
\end{equation}
because the Gram permanent factors over the two orthogonal blocks.
In particular $\norm{e_i^{\odot(n-b)}\vee r_{j|i}^{\odot b}}_{(n)}^2=(n-b)!\,b!$, so the vectors
\begin{equation}
  E_{j,b}^{(i,n)}
  =\frac{e_i^{\odot(n-b)}\vee r_{j|i}^{\odot b}}{\sqrt{(n-b)!\,b!}}
\end{equation}
are unit vectors satisfying
\begin{equation}
  \ip{E_{j,b}^{(i,n)}}{E_{k,b}^{(i,n)}}_{(n)}=(R_i)_{jk}^{\,b},
  \qquad
  \ip{E_{j,b}^{(i,n)}}{E_{k,b'}^{(i,n)}}_{(n)}=0
  \quad(b\ne b').
\end{equation}
At fixed $b$ the Gram matrix of $\{E_{j,b}^{(i,n)}\}_j$ is therefore exactly $R_i^{\circ b}$, with no residual combinatorial factor.

Expand $U_{ij}=\sum_{b=1}^{n}z_{ijb}E_{j,b}^{(i,n)}$, so that $\norm{U_{ij}}_{(n)}^2=\sum_{b=1}^n|z_{ijb}|^2$.
Orthogonality across $b$ gives
\begin{equation}
  \norm{\sum_jU_{ij}}_{(n)}^2
  =\sum_{b=1}^{n}\norm{\sum_jz_{ijb}E_{j,b}^{(i,n)}}_{(n)}^2
  =\sum_{b=1}^{n}z_{ib}^*R_i^{\circ b}z_{ib},
  \qquad z_{ib}=(z_{ijb})_j .
\end{equation}
Moreover,
\begin{equation}
  R_i=D_i^{-1}(C_{J,J}-C_{J,i}C_{i,J})D_i^{-1},
  \qquad D_i=\diag\{(1-C_{ij}^2)^{1/2}\}_j.
\end{equation}
Because $0\preceq C_{J,J}-C_{J,i}C_{i,J}\preceq C_{J,J}$ and $\norm{D_i^{-1}}_2\le\eps^{-1/2}$ by the support condition, we have $\norm{R_i}_2\le C_\eps\kappa$.
Since $R_i$ is a correlation matrix, so is $R_i^{\circ(b-1)}$ for every $b\ge1$ by the Schur product theorem, and writing $R_i^{\circ b}=R_i\circ R_i^{\circ(b-1)}$ exhibits $R_i^{\circ b}$ as the image of $R_i$ under the unital completely positive Schur multiplier $M\mapsto M\circ R_i^{\circ(b-1)}$, which is an operator-norm contraction on Hermitian matrices, exactly as in the proof of Lemma~\ref{supp-lem:schur}.
Hence $\norm{R_i^{\circ b}}_2\le\norm{R_i}_2\le C_\eps\kappa$ for every $b\ge1$.
Therefore
\begin{equation}
  \norm{\sum_jU_{ij}}_{(n)}^2
  \le C_\eps\kappa\sum_{b=1}^{n}\norm{z_{ib}}_2^2
  =C_\eps\kappa\sum_j\norm{U_{ij}}_{(n)}^2.
\end{equation}
No factor of $N$ appears because the sum over $b$ is an orthogonal sum.
\end{proof}

\begin{lemma}[Pure projection correction]
\label{supp-lem:pure}
For arbitrary $U_{ij}\in R_{j|i,N-1}$,
\begin{equation}
  \norm{\sum_{(i,j)\in E}P_{P_{ij,N}}c(e_i)U_{ij}}_{(N)}^2
  \le C_\eps N\kappa^2\sum_{(i,j)\in E}\norm{U_{ij}}_{(N-1)}^2,
\end{equation}
where
\begin{equation}
  P_{ij,N}=\Span\{e_i^{\odot N},e_j^{\odot N}\}.
\end{equation}
\end{lemma}

\begin{proof}
Let $\widehat e_i^{(N)}=e_i^{\odot N}/\sqrt{N!}$ be the unit pure tensor in the Fock norm.
Since $U_{ij}\in R_{j|i,N-1}$, we have $U_{ij}\perp\widehat e_i^{(N-1)}$, whence
\begin{equation}
  \langle c(e_i)U_{ij},\widehat e_i^{(N)}\rangle_{(N)}=0,
  \qquad
  |\langle c(e_i)U_{ij},\widehat e_j^{(N)}\rangle_{(N)}|
  =\sqrt N |C_{ij}|\,u_{ij},
\end{equation}
where
\begin{equation}
  u_{ij}
  =|\langle U_{ij},\widehat e_j^{(N-1)}\rangle_{(N-1)}|
  \le\norm{U_{ij}}_{(N-1)} .
\end{equation}
The two unit pure tensors $\widehat e_i^{(N)}$ and $\widehat e_j^{(N)}$ have inner product $C_{ij}^N$, so the projection
\begin{equation}
  P_{P_{ij,N}}c(e_i)U_{ij}=\alpha_{ij}\widehat e_i^{(N)}+\beta_{ij}\widehat e_j^{(N)}
\end{equation}
has coefficients solving the normal equations
\begin{equation}
\begin{pmatrix}1&C_{ij}^N\\ C_{ij}^N&1\end{pmatrix}
\begin{pmatrix}\alpha_{ij}\\ \beta_{ij}\end{pmatrix}
=
\begin{pmatrix}0\\ \sqrt N\,C_{ij}\langle U_{ij},\widehat e_j^{(N-1)}\rangle_{(N-1)}\end{pmatrix}.
\end{equation}
Since $N\ge2$ and $|C_{ij}|\le1-\eps$, the Gram matrix on the left has inverse of operator norm at most $\{1-(1-\eps)^2\}^{-1}\le C_\eps$, whence
\begin{equation}
  |\alpha_{ij}|\le C_\eps\sqrt N |C_{ij}|\,u_{ij},
  \qquad
  |\beta_{ij}|\le C_\eps\sqrt N |C_{ij}|\,u_{ij}.
\end{equation}
Write the total correction as $\sum_pc_p\widehat e_p^{(N)}$ with $c_p=a_p+b_p$, where
\begin{equation}
  a_p=\sum_{j:(p,j)\in E}\alpha_{pj},
  \qquad
  b_p=\sum_{i:(i,p)\in E}\beta_{ip}.
\end{equation}
For the outgoing part, Cauchy--Schwarz and Lemma~\ref{supp-lem:schur} with $r=1$ yield
\begin{align}
  \sum_p|a_p|^2
  &\le C_\eps N\sum_p
  \left(\sum_{j:(p,j)\in E} |C_{pj}|\,u_{pj}\right)^2\\
  &\le C_\eps N\sum_p\left(\sum_j|C_{pj}|^2\right)
  \left(\sum_{j:(p,j)\in E}u_{pj}^2\right)
  \le C_\eps N\kappa\sum_{(p,j)\in E}\norm{U_{pj}}_{(N-1)}^2.
\end{align}
For the incoming part, the same chain with the roles of the two indices exchanged gives
\begin{align}
  \sum_p|b_p|^2
  &\le C_\eps N\sum_p
  \left(\sum_{i:(i,p)\in E} |C_{ip}|\,u_{ip}\right)^2\\
  &\le C_\eps N\sum_p\left(\sum_i|C_{ip}|^2\right)
  \left(\sum_{i:(i,p)\in E}u_{ip}^2\right)
  \le C_\eps N\kappa\sum_{(i,p)\in E}\norm{U_{ip}}_{(N-1)}^2.
\end{align}
Therefore $\sum_p|c_p|^2\le2\sum_p|a_p|^2+2\sum_p|b_p|^2\le C_\eps N\kappa\sum_{(i,j)\in E}\norm{U_{ij}}_{(N-1)}^2$.
The Gram matrix of the pure tensors $\{\widehat e_p^{(N)}\}$ is $C^{\circ N}$, whose operator norm is at most $\kappa$.
Thus
\begin{equation}
  \norm{\sum_pc_p\widehat e_p^{(N)}}_{(N)}^2
  =c^*C^{\circ N}c
  \le\kappa\sum_p|c_p|^2
  \le C_\eps N\kappa^2\sum_{(i,j)\in E}\norm{U_{ij}}_{(N-1)}^2.
\end{equation}
\end{proof}

The duality step in the proposition uses the following reduction, which we record separately.

\begin{lemma}[Reduction to residual inputs]
\label{supp-lem:wlog-residual}
Fix $N\ge2$ and an ordered edge $(i,j)$ with $|C_{ij}|<1$, and let $\Pi_{ij}$ denote the orthogonal projection of $H^{\odot(N-1)}$ onto $R_{j|i,N-1}$.
Then, for every $U\in H^{\odot(N-1)}$,
\begin{equation}
  P_{ij,N}^{\rm int}\,c(e_i)U
  =P_{ij,N}^{\rm int}\,c(e_i)\Pi_{ij}U .
  \label{supp-eq:wlog-residual}
\end{equation}
\end{lemma}

\begin{proof}
Write $U=\Pi_{ij}U+U_1+U_2$, where $U_1$ is the component along $e_i^{\odot(N-1)}$ and $U_2$ is orthogonal to $\Span\{e_i,e_j\}^{\odot(N-1)}$; this decomposition is orthogonal because $\Span\{e_i,e_j\}^{\odot(N-1)}=\Span\{e_i^{\odot(N-1)}\}\oplus R_{j|i,N-1}$.
For every $W\in\Span\{e_i,e_j\}^{\odot N}$,
\[
\ip{c(e_i)U_2}{W}_{(N)}=\ip{U_2}{a(e_i)W}_{(N-1)}=0,
\]
since $a(e_i)W\in\Span\{e_i,e_j\}^{\odot(N-1)}$.
As $K_{ij,N}\subset\Span\{e_i,e_j\}^{\odot N}$, this gives $P_{ij,N}^{\rm int}c(e_i)U_2=0$.
Finally, $c(e_i)U_1$ is a multiple of $e_i^{\odot N}$, which lies in the subtracted space $\Span\{e_i^{\odot N},e_j^{\odot N}\}$, so $P_{ij,N}^{\rm int}c(e_i)U_1=0$.
\end{proof}

\begin{proof}[Proof of Proposition~\ref{supp-prop:fusion-frame}]
Fix $N\ge2$.
Equip the Hilbert direct sum $\mathcal H_\oplus=\bigoplus_{(i,j)\in E}H^{\odot(N-1)}$ with the norm $\norm{U}_\oplus^2=\sum_{(i,j)\in E}\norm{U_{ij}}_{(N-1)}^2$, and define the oriented analysis operator
\begin{equation}
  T\colon H^{\odot N}\to\mathcal H_\oplus,
  \qquad
  TF=\bigl\{a(e_i)P_{ij,N}^{\rm int}F\bigr\}_{(i,j)\in E},
\end{equation}
whose adjoint is the synthesis map $T^*U=\sum_{(i,j)\in E}P_{ij,N}^{\rm int}c(e_i)U_{ij}$ on arbitrary inputs $U\in\mathcal H_\oplus$.
By Lemma~\ref{supp-lem:wlog-residual}, $T^*U=T^*\Pi U$, where $(\Pi U)_{ij}=\Pi_{ij}U_{ij}$ and $\norm{\Pi U}_\oplus\le\norm{U}_\oplus$; since residual families are themselves admissible inputs,
\begin{equation}
  \norm{T^*}
  =\sup\Bigl\{\norm{\textstyle\sum_{(i,j)\in E}P_{ij,N}^{\rm int}c(e_i)U_{ij}}_{(N)}
  :\ \norm{U}_\oplus\le1,\ U_{ij}\in R_{j|i,N-1}\Bigr\}.
\end{equation}
It therefore suffices to prove, for residual inputs $U_{ij}\in R_{j|i,N-1}$, the synthesis bound
\begin{equation}
  \norm{\sum_{(i,j)\in E}P_{ij,N}^{\rm int}c(e_i)U_{ij}}_{(N)}^2
  \le C_\eps N\kappa^2\sum_{(i,j)\in E}\norm{U_{ij}}_{(N-1)}^2.
  \label{supp-eq:oriented-synthesis}
\end{equation}
Since $c(e_i)U_{ij}\in \Span\{e_i,e_j\}^{\odot N}$,
\begin{equation}
  P_{ij,N}^{\rm int}c(e_i)U_{ij}
  =c(e_i)U_{ij}-P_{P_{ij,N}}c(e_i)U_{ij}.
\end{equation}
The main term is
\begin{equation}
  S_0=\sum_{i,j}c(e_i)U_{ij}=\sum_i c(e_i)U_i,
  \qquad U_i=\sum_jU_{ij}.
\end{equation}
By Lemmas~\ref{supp-lem:creation} and \ref{supp-lem:one-sided},
\begin{equation}
  \norm{S_0}_{(N)}^2\le N\kappa\sum_i\norm{U_i}_{(N-1)}^2
  \le C_\eps N\kappa^2\sum_{i,j}\norm{U_{ij}}_{(N-1)}^2.
\end{equation}
The correction term
\begin{equation}
  S_1=\sum_{(i,j)\in E}P_{P_{ij,N}}c(e_i)U_{ij}
\end{equation}
is bounded by Lemma~\ref{supp-lem:pure}.
Therefore \eqref{supp-eq:oriented-synthesis} follows from $\norm{S_0-S_1}^2\le2\norm{S_0}^2+2\norm{S_1}^2$.
Consequently $\norm{T}^2=\norm{T^*}^2\le C_\eps N\kappa^2$, which is the oriented analysis estimate
\begin{equation}
  \sum_{(i,j)\in E}\norm{a(e_i)P_{ij,N}^{\rm int}F}_{(N-1)}^2
  \le C_\eps N\kappa^2\norm{F}_{(N)}^2.
  \label{supp-eq:oriented-analysis}
\end{equation}
The same bound holds with the roles of $i$ and $j$ interchanged:
\begin{equation}
  \sum_{(i,j)\in E}\norm{a(e_j)P_{ij,N}^{\rm int}F}_{(N-1)}^2
  \le C_\eps N\kappa^2\norm{F}_{(N)}^2.
\end{equation}
Now let $f=P_{ij,N}^{\rm int}F\in K_{ij,N}$.
On $\Span\{e_i,e_j\}$ the frame operator
\begin{equation}
  S_{ij}=e_i\otimes e_i+e_j\otimes e_j
\end{equation}
has smallest eigenvalue $1-|C_{ij}|\ge\eps$.
Hence $d\Gamma(S_{ij})\succeq N\eps I$ on $\Span\{e_i,e_j\}^{\odot N}$, while
\begin{equation}
  \ip{f}{d\Gamma(S_{ij})f}_{(N)}
  =\norm{a(e_i)f}_{(N-1)}^2+\norm{a(e_j)f}_{(N-1)}^2 ,
\end{equation}
so that
\begin{equation}
  \norm{f}_{(N)}^2
  \le \frac{1}{N\eps}
  \{\norm{a(e_i)f}_{(N-1)}^2+\norm{a(e_j)f}_{(N-1)}^2\}.
\end{equation}
Summing over $(i,j)\in E$ and inserting the two oriented analysis bounds gives \eqref{supp-eq:fusion-analysis}: the factor $N$ from \eqref{supp-eq:oriented-analysis} is canceled exactly by the $1/(N\eps)$ from the number operator.
The synthesis form \eqref{supp-eq:fusion-synthesis} follows by Hilbert-space duality.
\end{proof}

\subsection*{Final assembly}
\paragraph*{Two-vertex residual}
The fusion-frame proposition gives the corresponding bound for the genuinely two-vertex part.
Fix a support pair $(i,j)$ and write $\eta^{\rm int}_{ij}=\eta_{C_{ij}}^{\rm int}(g_i,g_j)$.
Since $\eta^{\rm int}_{ij}\in L^2(\sigma(g_i,g_j))$, it has a chaos expansion $\eta^{\rm int}_{ij}=\sum_{N\ge0}I_N(F_{ij,N})$ with kernels $F_{ij,N}\in\Span\{e_i,e_j\}^{\odot N}$.
Centering kills the constant term, $F_{ij,0}=0$.
Orthogonality of $\eta^{\rm int}_{ij}$ to $L_0^2(g_i)+L_0^2(g_j)$ forces, for every $N\ge1$,
\begin{equation}
  F_{ij,N}\perp\Span\{e_i^{\odot N},e_j^{\odot N}\}.
\end{equation}
At $N=1$ this reads $F_{ij,1}\perp\Span\{e_i,e_j\}$, while also $F_{ij,1}\in\Span\{e_i,e_j\}$, so $F_{ij,1}=0$.
For $N\ge2$,
\begin{equation}
  F_{ij,N}\in\Span\{e_i,e_j\}^{\odot N}\ominus\Span\{e_i^{\odot N},e_j^{\odot N}\}
  =K_{ij,N}.
\end{equation}
Hence
\begin{equation}
  \eta_{ij}^{\rm int}=\sum_{N\ge2}I_N(F_{ij,N}),
  \qquad F_{ij,N}\in K_{ij,N}.
\end{equation}
Define the truncations $Q_M^{(2)}(A)=\sum_{(i,j)\in E}A_{ij}\sum_{N=2}^MI_N(F_{ij,N})$.
For each $M$, chaos orthogonality and the synthesis form \eqref{supp-eq:fusion-synthesis} with coefficients $A_{ij}$ give
\begin{align}
  \Var(Q_M^{(2)}(A))
  &=\sum_{N=2}^M\norm{\sum_{i,j}A_{ij}F_{ij,N}}_{(N)}^2\\
  &\le C_\eps\kappa^2\sum_{i,j}|A_{ij}|^2
  \sum_{N=2}^M\norm{F_{ij,N}}_{(N)}^2\\
  &\le C_\eps\kappa^2\sum_{i,j}|A_{ij}|^2
  \norm{\eta_{ij}^{\rm int}}_2^2.
\end{align}
By Lemma~\ref{supp-lem:pair-projection}, $\norm{\eta_{ij}^{\rm int}}_2\le2\norm{h}_\infty^2$.
Because the edge set is finite and $\sum_{N\ge2}\norm{F_{ij,N}}_{(N)}^2=\norm{\eta^{\rm int}_{ij}}_2^2<\infty$ for each edge,
\begin{equation}
  \norm{Q^{(2)}(A)-Q_M^{(2)}(A)}_2
  \le\sum_{(i,j)\in E}|A_{ij}|\,
  \norm{\eta^{\rm int}_{ij}-\sum_{N=2}^MI_N(F_{ij,N})}_2
  \longrightarrow0,
\end{equation}
so $Q_M^{(2)}(A)\to Q^{(2)}(A)$ in $L^2$ and, since all terms are centered, $\Var(Q_M^{(2)}(A))\to\Var(Q^{(2)}(A))$.
Passing to the limit yields
\begin{equation}
  \Var(Q^{(2)}(A))\le K_{h,\eps}\kappa^2\norm{A}_F^2.
  \label{eq:two-vertex-main-bound}
\end{equation}

\paragraph*{Assembly}
\begin{proof}[Proof of the contraction theorem]
First assume that $A$ is real symmetric and hollow.
The decomposition
\begin{equation}
  Q(A)=Q^{(1)}(A)+Q^{(2)}(A)
\end{equation}
gives
\begin{equation}
  \Var(Q(A))\le2\Var(Q^{(1)}(A))+2\Var(Q^{(2)}(A)).
\end{equation}
Lemma~\ref{supp-lem:one-vertex} and \eqref{eq:two-vertex-main-bound} yield
\begin{equation}
  \Var(v^\top Av)\le K_{h,\eps}\kappa^2\norm{A}_F^2.
\end{equation}
For Hermitian $A$, $\operatorname{Re}A$ is real symmetric and $\operatorname{Im}A$ is real skew-symmetric.
Since $v$ is real-valued,
\begin{equation}
  v^*Av=v^\top(\operatorname{Re}A)v,
  \qquad v^\top(\operatorname{Im}A)v=0.
\end{equation}
The matrix $\operatorname{Re}A$ is real symmetric and hollow, and $\norm{\operatorname{Re}A}_F\le\norm{A}_F$.
Applying the real symmetric result to $\operatorname{Re}A$ proves the theorem.
\end{proof}

\begin{remark}[Coordinate-dependent transforms]
\label{supp-rem:heterogeneous}
The proof above never uses that the two transforms at the endpoints of an edge coincide.
If $v_i=h_i(g_i)$ with centered $h_i\in L^2(\gamma)\cap L^\infty(\gamma)$, then the pairwise projection bounds of Lemma~\ref{supp-lem:pair-projection} hold for $\eta(G_1,G_2)=h_i(G_1)h_j(G_2)-\E[h_i(G_1)h_j(G_2)]$, with the constant determined by $\sup_i\norm{h_i}_{L^2(\gamma)}$ and $\sup_i\norm{h_i}_\infty$; the one-vertex aggregate of Lemma~\ref{supp-lem:one-vertex} is already heterogeneous in $i$; the inclusion $F_{ij,N}\in K_{ij,N}$ uses only orthogonality to $L_0^2(g_i)+L_0^2(g_j)$ and chaos grading; and Proposition~\ref{supp-prop:fusion-frame} does not involve the transforms.
Hence the contraction theorem holds verbatim for coordinate-dependent centered transforms, in particular for threshold signs $h_{\tau_i}$ with coordinate-dependent thresholds $\tau_i$ in a compact interval.
\end{remark}

\section{Proof of the one-bit sparse-ruler Toeplitz upper bound}
\label{sec:supp-upper}

\paragraph{Proof map for the main upper-bound theorem}
The proof follows the estimator decomposition in Section~2 of the main paper.
For the oracle estimator, Theorem~4.1 and the identity $\norm{A_\theta(w)}_F^2\le2W^2\varphi(\Omega)$ give pointwise variance control for centered sparse-ruler spectral polynomials.
A trigonometric grid and Bernstein's inequality then turn pointwise concentration into Toeplitz operator-norm control.
A Taylor expansion of the inverse centered link $\psi(\cdot;\tau)$ gives the first-order coverage term; a per-lag application of the same contraction theorem bounds the sum of squared lag deviations and places the second-order curvature term at the coverage scale, leaving the spectral boundedness term as the only $d$-scale remainder.

The plug-in proof uses the same oracle bound after accounting for marginal-bit calibration of $(\gamma_0,\mu,\tau)$ from $n|\Omega|$ signs, empirical recentering by $\widehat\mu$, and the deterministic population perturbation controlled by $1+S_1(d;\rho)$; on operator-norm balls a Hadamard-power argument removes the short-memory factor from this last term.

The assumptions and notation below are kept aligned with the main theorem so the supplement is separately compilable and can be checked against the main-text statements it imports.

\subsection{Upper bounds over natural covariance classes}
We state the main covariance result over standard Toeplitz covariance classes rather than through an abstract calibration certificate.
The constants depend on compact threshold and inverse-link regularity domains but not on $d$, $n$ or the ruler geometry.

\begin{assumption}[Buffered regular inverse-link regime]
\label{ass:inverse-s}
There exist $\eps\in(0,1/4)$ and $0<\tau_{\min}<\tau_{\max}<\infty$ such that
\[
    |\rho_s|\le1-2\eps,\qquad s=1,\ldots,d-1,
\]
and $\tau\in[\tau_{\min},\tau_{\max}]$, while the estimator clips to the larger compact inverse-link domain
\[
    \mathcal I_{\tau'}=c([-1+\eps,1-\eps];\tau').
\]
Set
\[
    L_k=1\vee\sup\bigl\{|\partial_u^k\psi(u;\tau')|:\
    \tau'\in[\tau_{\min}/2,2\tau_{\max}],\ u\in\mathcal I_{\tau'}\bigr\},
    \qquad k=1,2,
\]
so that $L_1,L_2\ge1$ are uniform inverse-link derivative bounds on the buffered domain; they are finite by Lemma~\ref{supp-lem:inverse-link}(i).
The two-layer structure (true lags in $[-1+2\eps,1-2\eps]$, clipping at the $\eps$ layer) keeps every Taylor expansion of the plug-in analysis inside the clipping domain (Lemma~\ref{supp-lem:inverse-link}(ii)).
\end{assumption}

Let
\[
    \kappa_{\rm obs}=\norm{\Gamma_{\Omega,\Omega}}_2/\gamma_0,
    \qquad
    S_1(d;\rho)=\sum_{s=1}^{d-1}|\rho_s|.
\]
The empirical centered pair fluctuation is governed by $\kappa_{\rm obs}$.
The plug-in population shift is controlled by the short-memory size $S_1(d;\rho)$.

\begin{assumption}[Short-memory Toeplitz class]
\label{ass:short-memory-s}
The normalized Toeplitz lags satisfy
\[
    S_1(d;\rho)=\sum_{s=1}^{d-1}|\rho_s|\le S_\star,
\]
where $S_\star$ is independent of $d$ and $n$.
\end{assumption}

\paragraph{Main-text theorem notation}
Assume Assumption~\ref{ass:inverse-s}.
Let $t=\log(Cd/\delta)$ and $m=|\Omega|$.
There are constants $C,c,\eta_0>0$, depending only on the regularity domains in Assumption~\ref{ass:inverse-s}, such that the following hold.

\noindent\textup{(a) Oracle estimator.}
If the scale and sign mean are known, then with probability at least $1-\delta$,
\begin{equation}
    \norm{\widehat\Gamma_{\rm or}-\Gamma}_2
      \le C_\eps\gamma_0L_1\kappa_{\rm obs}
          \sqrt{\frac{\varphi(\Omega)t}{n}}
       + C\gamma_0L_1d\frac{t}{n}
       + C_\eps\gamma_0L_2\left\{
         \min\{\kappa_{\rm obs}^2\varphi(\Omega),d\}\frac{t}{n}
         +d\frac{t^2}{n^2}\right\}.
\label{eq:oracle-rate-s}
\end{equation}

\noindent\textup{(b) Plug-in estimator.}
In addition assume Assumption~\ref{ass:short-memory-s}.
Define
\[
    r_\mu=C_\tau\left\{\sqrt{\frac{\kappa_{\rm obs}t}{nm}}+\frac{t}{n}\right\}.
\]
If
\[
    n\ge C\frac{mt}{\kappa_{\rm obs}},
    \qquad
    r_\mu\le c\eta_0,
\]
then with probability at least $1-\delta$,
\begin{align}
    \norm{\widehat\Gamma_{\rm plug}-\Gamma}_2
      &\le C_\eps\gamma_0L_1\kappa_{\rm obs}
          \sqrt{\frac{\varphi(\Omega)t}{n}}
        + C\gamma_0L_1d\frac{t}{n}
        + C_{\tau,\eps}\gamma_0L_2\left\{
          \min\{\kappa_{\rm obs}^2\varphi(\Omega),d\}\frac{t}{n}
          +d\frac{t^2}{n^2}\right\} \nonumber\\
      &\quad
        + C_{\tau,\eps}\gamma_0\{1+S_1(d;\rho)\}r_\mu .
\label{eq:plugin-rate-s}
\end{align}
In particular, uniformly over the short-memory class $S_1(d;\rho)\le S_\star$, the last term is at most
\[
    C_{\tau,\eps,S_\star}\gamma_0
    \left\{\sqrt{\frac{\kappa_{\rm obs}t}{nm}}+\frac{t}{n}\right\}.
\]

The bound separates the centered pair-coverage term through $\varphi(\Omega)$, the spectral boundedness term, the curvature term at the coverage scale, and the marginal-bit calibration term from $n|\Omega|$ one-coordinate signs.
Over short-memory spectral classes the calibration term remains lower-dimensional.

The lower bound below isolates the design-dependent part of the leading oracle term.
It matches the $\sqrt{\varphi(\Omega)\log d/n}$ dependence in a known-scale identity-neighborhood submodel under balanced spectral packing, without asserting sharpness of the correlation-stability, inverse-link, Taylor-remainder, or plug-in calibration factors.

\subsection*{Proof of the main upper-bound theorem}

The oracle proof uses inverse-link stability and sparse-ruler concentration.
The plug-in proof then adds pooled marginal calibration and a deterministic population-calibration perturbation bound over short-memory lag sequences.

\subsection{Marginal calibration and inverse-link stability}
\begin{lemma}[Pooled marginal calibration]
\label{supp-lem:pooled-calibration}
Let $S_{j}^{(\ell)}=\sign(x_j^{(\ell)}-\lambda)$ be threshold signs generated by vectors $\vect x^{(\ell)}$ and observed on $\Omega$, where each snapshot is independent and the observed correlation matrix on $\Omega$ has spectral norm at most $C\kappa_{\rm obs}$.
If $\tau\in[\tau_{\min},\tau_{\max}]$, then with probability at least $1-e^{-t}$,
\begin{equation}
  |\widehat\tau-\tau|+
  |\widehat\mu-\mu|+
  |\widehat\gamma_0/\gamma_0-1|
  \le C_\tau\left(
  \sqrt{\frac{\kappa_{\rm obs} t}{n|\Omega|}}+\frac{t}{n}
  \right).
  \label{supp-eq:pooled-cal-event}
\end{equation}
\end{lemma}

\begin{proof}
Let
\[
    \widehat q_{\rm raw}=(n|\Omega|)^{-1}\sum_{\ell,j}\mathbf 1\{S_j^{(\ell)}=1\},
    \qquad
    \mathcal Q_\tau=[Q(\tau_{\max}),Q(\tau_{\min})],
    \qquad
    \widehat q=\Pi_{\mathcal Q_\tau}(\widehat q_{\rm raw}),
\]
and let $q=Q(\tau)$.
For one snapshot, write $I_j=\mathbf 1\{S_j=1\}=\phi(g_j)+q$ with the centered indicator $\phi(t)=\mathbf 1\{t>\tau\}-q$, which is bounded with $\norm{\phi}_{L^2(\gamma)}^2=q(1-q)\le\tfrac14$.
Since the observed correlation matrix has spectral norm at most $C\kappa_{\rm obs}$, Lemma~\ref{supp-lem:centered-transform-covariance}, applied with $\phi_j=\phi$ and $b_j=|\Omega|^{-1}$, gives
\begin{equation}
\Var\left(|\Omega|^{-1}\sum_{j\in\Omega}I_j\right)
  \le C\kappa_{\rm obs}\norm{\phi}_{L^2(\gamma)}^2|\Omega|^{-1}
  \le C_\tau\frac{\kappa_{\rm obs}}{|\Omega|}.
\end{equation}
The summands are bounded and snapshots are independent.
Bernstein's inequality therefore yields
\begin{equation}
|\widehat q_{\rm raw}-q|
  \le C_\tau\left(
  \sqrt{\frac{\kappa_{\rm obs} t}{n|\Omega|}}+\frac{t}{n}
  \right)
\end{equation}
with probability at least $1-e^{-t}$.
Since the true $q=Q(\tau)$ belongs to $\mathcal Q_\tau$, the projection step is non-expansive:
\[
    |\widehat q-q|
    =
    |\Pi_{\mathcal Q_\tau}(\widehat q_{\rm raw})-q|
    \le |\widehat q_{\rm raw}-q|.
\]
Define $\widehat\tau=Q^{-1}(\widehat q)$.
On $\mathcal Q_\tau$, the map $q\mapsto Q^{-1}(q)$ has bounded derivative, and $\widehat\tau\in[\tau_{\min},\tau_{\max}]$.
The maps $\tau\mapsto1-2\Phi(\tau)$ and $\tau\mapsto(\lambda/\tau)^2$ are also Lipschitz on this compact interval. The mean-value theorem therefore transfers the displayed bound to $\widehat\tau$, $\widehat\mu$ and $\widehat\gamma_0/\gamma_0$.
\end{proof}

\begin{lemma}[Quantitative inverse-link calculus]
\label{supp-lem:inverse-link}
Write $\psi(\cdot;\tau')=c(\cdot;\tau')^{-1}$ and $\mathcal I_{\tau'}=c([-1+\eps,1-\eps];\tau')$.
By Plackett's identity \citep{Plackett1954}, $\partial_\rho c(\rho;\tau')=4\phi_2(\tau',\tau';\rho)>0$, where $\phi_2$ is the bivariate standard normal density, so on the compact set $\rho\in[-1+\eps,1-\eps]$, $\tau'\in[\tau_{\min}/2,2\tau_{\max}]$ the link $c$ is $C^2$ with $\partial_\rho c$ bounded above and below by positive constants and bounded second derivatives.
The following then hold, with all constants depending only on $\eps$, $\tau_{\min}$ and $\tau_{\max}$.
\begin{enumerate}
\item[(i)] $\psi(\cdot;\tau')$ has first and second derivatives on $\mathcal I_{\tau'}$ bounded by constants.
\item[(ii)] There exist $b_\eps,\eta_\eps>0$ such that, if $\tau\in[\tau_{\min},\tau_{\max}]$, $|\tau'-\tau|\le\eta_\eps$ and $|r|\le1-2\eps$, then
\begin{equation}
\operatorname{dist}\bigl\{c(r;\tau),\,\mathcal I_{\tau'}^{\,c}\bigr\}\ge b_\eps;
\label{supp-eq:buffer}
\end{equation}
if moreover $|u-c(r;\tau)|\le b_\eps/2$, the segment $\{c(r;\tau)+\theta(u-c(r;\tau)):0\le\theta\le1\}$ lies in $\mathcal I_{\tau'}$.
\item[(iii)] With $\tau_\eta=\tau(1+\eta)^{-1/2}$ and $a_\eta(r)=(1+\eta)\psi(c(r;\tau);\tau_\eta)-r$, there exist $C_{\tau,\eps},\eta_0>0$ such that $a_\eta$ is $C^2$ jointly in $(\eta,r)$ on $|\eta|\le\eta_0$, $|r|\le1-2\eps$, with $a_0\equiv0$, and
\begin{equation}
|a_\eta(r)|\le C_{\tau,\eps}|\eta|\,|r|.
\label{eq:scalar-pop-cal}
\end{equation}
\end{enumerate}
\end{lemma}

\begin{proof}
The Plackett expression and compactness give the stated two-sided bound on $\partial_\rho c$ and the bound on its second derivatives; the inverse-function theorem then yields (i).

For (ii), $c(\cdot;\tau')$ increasing gives $\mathcal I_{\tau'}=[c(-1+\eps;\tau'),c(1-\eps;\tau')]$.
With
\[
m_\eps=\inf_{\mathcal K}\partial_\rho c,
\qquad
M_\eps=\sup_{\mathcal K}|\partial_{\tau'} c|,
\qquad
\mathcal K=[-1+\eps,1-\eps]\times[\tau_{\min}/2,2\tau_{\max}],
\]
set $\eta_\eps=\min\{\tau_{\min}/2,\,m_\eps\eps/(2(M_\eps\vee1))\}$ and $b_\eps=m_\eps\eps/2$.
If $|\tau'-\tau|\le\eta_\eps$ then $\tau'\in[\tau_{\min}/2,2\tau_{\max}]$, and for $|r|\le1-2\eps$ integrating $\partial_\rho c(\cdot;\tau)\ge m_\eps$ gives $c(r;\tau)-c(-1+\eps;\tau)\ge m_\eps\eps$ and $c(1-\eps;\tau)-c(r;\tau)\ge m_\eps\eps$, while threshold perturbation moves each endpoint by at most $|c(\pm(1-\eps);\tau')-c(\pm(1-\eps);\tau)|\le M_\eps|\tau'-\tau|\le m_\eps\eps/2$.
Subtracting gives \eqref{supp-eq:buffer}, and the segment statement follows since every segment point is within $b_\eps/2$ of $c(r;\tau)$.

For (iii), shrink $\eta_0$ so that $|\tau_\eta-\tau|\le\eta_\eps$ for $|\eta|\le\eta_0$; then (ii) places $c(r;\tau)$ in $\mathcal I_{\tau_\eta}$ with margin $b_\eps$, so by (i) the map $a_\eta(r)$ is $C^2$ jointly in $(\eta,r)$.
Moreover $a_0(r)=\psi(c(r;\tau);\tau)-r=0$, and $a_\eta(0)=0$ because $c(0;\tau)=0$ and $\psi(0;\tau_\eta)=0$.
The two-variable fundamental theorem of calculus gives $a_\eta(r)=\int_0^\eta\int_0^r\partial_{\eta r}a_u(v)\,dv\,du$, and the mixed derivative is uniformly bounded on the compact domain, proving \eqref{eq:scalar-pop-cal}.
\end{proof}

\begin{corollary}[Short-memory population calibration]
\label{lem:population-calibration}
Under Assumption~\ref{ass:inverse-s}, for any Toeplitz lag sequence $\rho_1,\ldots,\rho_{d-1}$ with $|\rho_s|\le1-2\eps$ and all $|\eta|\le\eta_0$, the scale perturbation $a_\eta$ of Lemma~\ref{supp-lem:inverse-link}(iii) satisfies
\begin{equation}
    \norm{\eta\Id_d+\Toep_d^0\{a_\eta(\rho_s)\}_{s=1}^{d-1}}_2
    \le
    C_{\tau,\eps}\{1+S_1(d;\rho)\}|\eta|.
    \label{eq:short-memory-pop-cal}
\end{equation}
\end{corollary}

\begin{proof}
The row-sum bound and \eqref{eq:scalar-pop-cal} give
\[
    \norm{\Toep_d^0\{a_\eta(\rho_s)\}_{s=1}^{d-1}}_2
    \le
    2\sum_{s=1}^{d-1}|a_\eta(\rho_s)|
    \le
    2C_{\tau,\eps}|\eta|S_1(d;\rho),
\]
and $\norm{\eta\Id_d}_2=|\eta|$ proves \eqref{eq:short-memory-pop-cal} after adjusting constants.
\end{proof}

The oracle and plug-in upper bounds both reduce to one deterministic step: expanding the clipped, scale-perturbed inverse link around the true lag.
The next lemma performs this expansion once, in the two variables (entry error, scale error) simultaneously, so that the oracle case is the zero-scale specialization.

\begin{lemma}[Perturbed-threshold completion]
\label{supp-lem:completion}
Let $\gamma_0>0$, fix a lag $s$ with $|\rho_s|\le1-2\eps$ and $c_s=c(\rho_s;\tau)$, set $w_s=\gamma_0\psi'(c_s;\tau)$, and let $a_\eta$ be the scale perturbation of Lemma~\ref{supp-lem:inverse-link}(iii).
There are constants $C_{\tau,\eps},\eta_0>0$ such that, for every $|\eta|\le\eta_0$ and every entry error $e_s$ for which the segment from $c_s$ to $c_s+e_s$ lies in $\mathcal I_{\tau_\eta}$,
\begin{equation}
\gamma_0\bigl[(1+\eta)\psi(c_s+e_s;\tau_\eta)-\rho_s\bigr]
=w_se_s+\gamma_0a_\eta(\rho_s)+r_s,
\qquad
|r_s|\le C_{\tau,\eps}\gamma_0\bigl(e_s^2+|\eta|\,|e_s|\bigr).
\label{supp-eq:completion}
\end{equation}
\end{lemma}

\begin{proof}
Write $G(e,\eta)=(1+\eta)\psi(c_s+e;\tau_\eta)-\rho_s$, so the left side of \eqref{supp-eq:completion} is $\gamma_0G(e_s,\eta)$.
Every evaluation $\psi(c_s+\theta e_s;\tau_\eta)$, $\theta\in[0,1]$, lies in $\mathcal I_{\tau_\eta}$, where by Lemma~\ref{supp-lem:inverse-link}(i) $\psi(\cdot;\tau_\eta)$ is $C^2$ with $|\psi'|\le L_1$ and $|\psi''|\le L_2$; hence $G$ is $C^2$ jointly in $(e,\eta)$ for $|\eta|\le\eta_0$.
Now $G(0,\eta)=(1+\eta)\psi(c_s;\tau_\eta)-\rho_s=a_\eta(\rho_s)$ because $c_s=c(\rho_s;\tau)$, and $\partial_eG(0,0)=\psi'(c_s;\tau)$.
Taylor's theorem in $e$ at fixed $\eta$ gives $G(e_s,\eta)=a_\eta(\rho_s)+\partial_eG(0,\eta)e_s+\tfrac12\partial_e^2G(\zeta,\eta)e_s^2$ with $|\partial_e^2G|\le(1+\eta_0)L_2$, while $|\partial_eG(0,\eta)-\partial_eG(0,0)|=\bigl|\int_0^\eta\partial_\eta\partial_eG(0,u)\,du\bigr|\le C_{\tau,\eps}|\eta|$ since $\partial_\eta\partial_eG$ is bounded on the compact domain.
Collecting terms, $G(e_s,\eta)=a_\eta(\rho_s)+\psi'(c_s;\tau)e_s+R$ with $|R|\le C_{\tau,\eps}(e_s^2+|\eta||e_s|)$; multiplying by $\gamma_0$ gives \eqref{supp-eq:completion}.
\end{proof}

\subsection{Population perturbation on operator-norm balls}
\label{subsec:supp-hadamard-perturbation}

The row-sum bound \eqref{eq:short-memory-pop-cal} is the only step of the plug-in analysis where the short-memory size enters.
On operator-norm balls $\norm{T_d^0(\rho)}_2\le c_0$ it can be replaced by a bound carrying no lag-summability factor.
Two ingredients are needed.

\begin{lemma}[Kronecker-compression bound for the Hadamard product]
\label{supp-lem:kronecker-hadamard}
For all real or complex $d\times d$ matrices $A$ and $B$,
\[
\opnorm{A\circ B}\le\opnorm{A}\,\opnorm{B},
\qquad\text{and hence}\qquad
\opnorm{A^{\circ k}}\le\opnorm{A}^{k}\ \ (k\ge1),
\]
where $A^{\circ k}$ denotes the entrywise $k$-th power.
\end{lemma}

\begin{proof}
Let $P\colon\mathbb C^d\to\mathbb C^{d^2}$ be the isometry with $Pe_i=e_i\otimes e_i$.
Then $(P^*(A\otimes B)P)_{ij}=(e_i\otimes e_i)^*(A\otimes B)(e_j\otimes e_j)=A_{ij}B_{ij}$, so $A\circ B=P^*(A\otimes B)P$ is a compression of $A\otimes B$, and $\opnorm{A\circ B}\le\opnorm{A\otimes B}=\opnorm{A}\opnorm{B}$.
Iterating gives the second claim.
\end{proof}

No positivity is required, in contrast with the Schur-multiplier contraction of Lemma~\ref{supp-lem:schur}; the two statements should not be conflated.
Since the entrywise power of a hollow Toeplitz matrix is the hollow Toeplitz matrix of the powered lags,
\begin{equation}
T_d^0(\rho_s^k)=\bigl\{T_d^0(\rho)\bigr\}^{\circ k},
\qquad
\opnorm{T_d^0(\rho_s^k)}\le\opnorm{T_d^0(\rho)}^{k}\le c_0^{\,k},
\qquad k\ge1 ,
\label{supp-eq:hollow-hadamard-power}
\end{equation}
and this geometric decay in $k$ is what makes the series argument below unconditional on the ball.

The second ingredient is an analytic continuation of the centered link.
Plackett's identity $\partial_rc(r;\tau)=4\phi_2(\tau,\tau;r)$ gives, for real $r\in(-1,1)$,
\begin{equation}
c(r;\tau)
=\frac2\pi\int_0^r
\exp\Bigl(-\frac{\tau^2}{1+u}\Bigr)(1-u^2)^{-1/2}\,du .
\label{supp-eq:plackett-integral}
\end{equation}
Fix on the unit disk $D(0,1)\subset\mathbb C$ the analytic branch of $(1-z^2)^{-1/2}$ equal to $1$ at $z=0$; the singularities $z=\pm1$ of the integrand lie on the boundary.
The integrand is then analytic on the simply connected domain $D(0,1)$, so the integral in \eqref{supp-eq:plackett-integral}, taken along any path from $0$ to the argument inside $D(0,1)$, is path-independent and extends $c(\cdot;\tau)$ analytically to $D(0,1)$, with
\[
\partial_r c(r;\tau)
=\frac2\pi e^{-\tau^2/(1+r)}(1-r^2)^{-1/2}\ne0
\quad\text{on }D(0,1),
\qquad
\partial_\tau c(r;\tau)
=-\frac{4\tau}\pi\int_0^r\frac{e^{-\tau^2/(1+u)}}{(1+u)\sqrt{1-u^2}}\,du ,
\]
the latter analytic on $D(0,1)$ with $\partial_\tau c(0;\tau)=0$; differentiation under the integral sign is justified by uniform convergence of the integrand and its $\tau$-derivative on compact subsets of $D(0,1)\times(0,\infty)$.

\begin{lemma}[Population perturbation on operator-norm balls]
\label{supp-lem:hadamard-perturbation}
Assume Assumption~\ref{ass:inverse-s} and let $c_0\le1-2\eps$.
For $\tau\in[\tau_{\min},\tau_{\max}]$ define
\begin{equation}
b_\tau(r)
:=r+\frac\tau2\cdot
\frac{\partial_\tau c(r;\tau)}{\partial_r c(r;\tau)} .
\label{supp-eq:btau}
\end{equation}
Then the following hold.
\begin{enumerate}
\item[(i)] $b_\tau$ is analytic on $D(0,1)$ with $b_\tau(0)=0$, and, writing $b_{k,\tau}$ for its Maclaurin coefficients and setting $\varrho=(1+c_0)/2$,
\begin{equation}
B:=\sup_{\tau\in[\tau_{\min},\tau_{\max}]}
\sum_{k\ge1}|b_{k,\tau}|\,c_0^{\,k}
\;\le\;
\Bigl(\sup_{\tau}\sup_{|z|=\varrho}|b_\tau(z)|\Bigr)
\frac{c_0/\varrho}{1-c_0/\varrho}
\;<\;\infty .
\label{supp-eq:B-const}
\end{equation}
\item[(ii)] $b_\tau(r)=\partial_\eta a_\eta(r)\big|_{\eta=0}$ for real $|r|\le1-2\eps$.
\item[(iii)] There exist $\eta_0'\in(0,\eta_0]$ and $C_{\tau,\eps}$, depending only on the regularity domains and $c_0$, such that for all $|\eta|\le\eta_0'$ and all lag sequences with $\opnorm{T_d^0(\rho)}\le c_0$,
\begin{equation}
\bigl\lVert
\eta \Id_d+\Toep_d^0\{a_\eta(\rho_s)\}_{s=1}^{d-1}
\bigr\rVert_2
\;\le\;
(1+B)\,|\eta|
\;+\;
C_{\tau,\eps}\,d\,\eta^2 .
\label{supp-eq:no-S1}
\end{equation}
\end{enumerate}
\end{lemma}

\begin{proof}
\emph{(i).}
The denominator $\partial_rc(\cdot;\tau)$ is analytic and zero-free on $D(0,1)$ and the numerator $\partial_\tau c(\cdot;\tau)$ is analytic there, so $b_\tau$ is analytic on $D(0,1)$; $b_\tau(0)=0$ because $\partial_\tau c(0;\tau)=0$.
Cauchy's estimate on $|z|=\varrho$ gives $|b_{k,\tau}|\le M_\tau(\varrho)\varrho^{-k}$ with $M_\tau(\varrho)=\sup_{|z|=\varrho}|b_\tau(z)|$, and summing the geometric series yields \eqref{supp-eq:B-const}.
The supremum over $\tau$ is finite because $(z,\tau)\mapsto b_\tau(z)$ is continuous on the compact set $\{|z|=\varrho\}\times[\tau_{\min},\tau_{\max}]$; in particular the denominator is bounded below there, $|\partial_rc(z;\tau)|\ge\tfrac2\pi e^{-\tau_{\max}^2/(1-\varrho)}(1+\varrho^2)^{-1/2}>0$.

\emph{(ii).}
Write $w=c(r;\tau)$, which does not depend on $\eta$, and $g(\eta)=(1+\eta)\psi(w;\tau_\eta)$, so that $a_\eta(r)=g(\eta)-r$ and $g(0)=\psi(w;\tau)=r$.
Differentiating the identity $c(\psi(w;\tau');\tau')=w$ in $\tau'$ at $\tau'=\tau$ gives $\partial_\tau\psi(w;\tau)=-\partial_\tau c(r;\tau)/\partial_rc(r;\tau)$.
Since $\frac{d}{d\eta}\tau_\eta\big|_{\eta=0}=-\tau/2$, the chain rule gives
\[
g'(0)=\psi(w;\tau)+\partial_\tau\psi(w;\tau)\cdot\Bigl(-\frac\tau2\Bigr)
=r+\frac\tau2\cdot
\frac{\partial_\tau c(r;\tau)}{\partial_r c(r;\tau)}
=b_\tau(r).
\]

\emph{(iii).}
Since $|\rho_s|\le\opnorm{T_d^0(\rho)}\le c_0<\varrho$, the series $b_\tau(\rho_s)=\sum_{k\ge1}b_{k,\tau}\rho_s^k$ converges absolutely for every $s$, and by \eqref{supp-eq:hollow-hadamard-power} the matrix series $\sum_{k\ge1}b_{k,\tau}\,T_d^0(\rho_s^k)$ converges in operator norm; its limit agrees entrywise with $\Toep_d^0\{b_\tau(\rho_s)\}$, so
\[
\bigl\lVert \Toep_d^0\{b_\tau(\rho_s)\}\bigr\rVert_2
\le\sum_{k\ge1}|b_{k,\tau}|\,c_0^{\,k}\le B .
\]
For the remainder $R_\eta(r):=a_\eta(r)-\eta\,b_\tau(r)$, the joint $C^2$ regularity of $(\eta,r)\mapsto a_\eta(r)$ on $\{|\eta|\le\eta_0'\}\times\{|r|\le1-2\eps\}$ --- established in the proof of Lemma~\ref{supp-lem:inverse-link}(iii), all evaluations staying inside the clipping domain by Lemma~\ref{supp-lem:inverse-link}(ii) --- together with $a_0\equiv0$ and \textup{(ii)} gives, by Taylor's theorem in $\eta$,
\[
\sup_{|r|\le1-2\eps}|R_\eta(r)|\le C_{\tau,\eps}\,\eta^2,
\qquad |\eta|\le\eta_0' .
\]
The crude row-sum bound suffices for this quadratically small part, $\opnorm{\Toep_d^0\{R_\eta(\rho_s)\}}\le2(d-1)C_{\tau,\eps}\eta^2$.
Adding the diagonal $\eta\Id_d$ proves \eqref{supp-eq:no-S1}.
\end{proof}

\begin{remark}
\label{supp-rem:btau}
Expanding at $r=0$ gives $b_\tau'(0)=1-\tau^2$: the first-order leakage of scale calibration into the lags changes sign at $\tau=1$.
The lemma uses $\opnorm{T_d^0(\rho)}\le c_0<1$ only through \eqref{supp-eq:hollow-hadamard-power}; positivity of $\Id+T_d^0(\rho)$ is not needed.
On general short-memory classes, where no operator-norm ball is assumed, Lemma~\ref{lem:population-calibration} remains the relevant tool.
\end{remark}

\subsection{Oracle sparse-ruler spectral concentration}
We prove the oracle statement for the real one-bit Toeplitz model used in the main theorem.
Let $c_s=c(\rho_s;\tau)$.
For each lag $s$ and snapshot $\ell$, define the snapshot-level lag average
\begin{equation}
Y_s^{(\ell)}=\frac1{q_s}
  \sum_{(j,k)\in\Omega_s} v_j^{(\ell)}v_k^{(\ell)},
  \qquad
\widehat c_s=\frac1n\sum_{\ell=1}^nY_s^{(\ell)}.
\end{equation}
Every observed pair in $\Omega_s$ has Gaussian correlation $\rho_s$, so $\E Y_s^{(\ell)}=c_s$.
Since $|v_j|\le2$, each snapshot-level average satisfies $|Y_s^{(\ell)}|\le4$.
For fixed $s$, the variables $Y_s^{(1)},\ldots,Y_s^{(n)}$ are independent because the snapshots are independent; no independence among the $q_s$ pair products inside a single snapshot is used.

\begin{lemma}[Trigonometric grid]
\label{supp-lem:grid}
Let $p(\theta)=\sum_{|s|\le d-1}a_s e^{2\pi i s\theta}$ be a trigonometric polynomial of degree at most $d-1$.
There is a grid $\mathcal G\subset[0,1]$ with $|\mathcal G|\le C d$ such that
\begin{equation}
\sup_{\theta\in[0,1]}|p(\theta)|
  \le 2\max_{\theta\in\mathcal G}|p(\theta)|.
\end{equation}
\end{lemma}

\begin{proof}
Bernstein's inequality for trigonometric polynomials gives $\norm{p'}_\infty\le 2\pi(d-1)\norm{p}_\infty$.
Take an equally spaced grid with mesh at most $(8\pi d)^{-1}$.
If $\theta_*$ maximizes $|p|$, choose $\theta_g$ within one mesh width.
Then $|p(\theta_*)-p(\theta_g)|\le \norm{p}_\infty/2$, which proves the claim.
\end{proof}

Both Toeplitz proofs convert lagwise errors into operator norms through the following standard symbol bound, recorded here so that every step is auditable.

\begin{lemma}[Toeplitz symbol domination]
\label{supp-lem:toeplitz-symbol}
Let $e_1,\ldots,e_{d-1}\in\R$, let $\Toep_d^0(e)$ be the hollow real symmetric Toeplitz matrix with off-diagonal lags $e_s$, and define the symbol
\[
f_e(\theta)=2\operatorname{Re}\sum_{s=1}^{d-1}e_se^{2\pi is\theta}
=2\sum_{s=1}^{d-1}e_s\cos(2\pi s\theta),
\qquad\theta\in[0,1].
\]
Then, for every $x\in\mathbb C^d$,
\begin{equation}
x^*\Toep_d^0(e)\,x
=\int_0^1f_e(\theta)\,
\Bigl|\sum_{j=0}^{d-1}x_je^{2\pi ij\theta}\Bigr|^2\,d\theta,
\label{supp-eq:symbol-identity}
\end{equation}
and consequently
\begin{equation}
\norm{\Toep_d^0(e)}_2\le\sup_{\theta\in[0,1]}|f_e(\theta)| .
\label{supp-eq:symbol-bound}
\end{equation}
\end{lemma}

\begin{proof}
Write $p_x(\theta)=\sum_{j=0}^{d-1}x_je^{2\pi ij\theta}$ and expand $|p_x(\theta)|^2=\sum_{j,k}\bar x_jx_ke^{2\pi i(k-j)\theta}$.
With $e_{-s}=e_s$ and $e_0=0$, so that $f_e(\theta)=\sum_{0<|s|\le d-1}e_{|s|}e^{2\pi is\theta}$, the orthogonality relation $\int_0^1e^{2\pi i(s+k-j)\theta}\,d\theta=\one\{s=j-k\}$ gives
\[
\int_0^1f_e(\theta)|p_x(\theta)|^2\,d\theta
=\sum_{j\ne k}e_{|j-k|}\,\bar x_jx_k
=x^*\Toep_d^0(e)\,x ,
\]
which is \eqref{supp-eq:symbol-identity}.
By Parseval, $\int_0^1|p_x(\theta)|^2\,d\theta=\norm{x}_2^2$, so $|x^*\Toep_d^0(e)\,x|\le\sup_\theta|f_e(\theta)|\,\norm{x}_2^2$.
Since $\Toep_d^0(e)$ is real symmetric, its operator norm equals the supremum of $|x^*\Toep_d^0(e)\,x|$ over unit vectors, which proves \eqref{supp-eq:symbol-bound}.
\end{proof}

The next lemma records the entrywise deviations; its second part applies the contraction theorem lag by lag and places the second-order Taylor term at the coverage scale.

\begin{lemma}[Entry deviations and the coverage-scale square sum]
\label{supp-lem:per-lag}
Assume Assumption~\ref{ass:inverse-s} and let $t=\log(Cd/\delta)$.
With probability at least $1-e^{-t}$, the following hold simultaneously:
\begin{equation}
  \max_{1\le s\le d-1}|\widehat c_s-c_s|\le C\sqrt{\frac tn},
  \label{supp-eq:entry-event-oracle}
\end{equation}
and
\begin{equation}
  \sum_{s=1}^{d-1}(\widehat c_s-c_s)^2
  \le C_\eps\left\{\min\{\kobs^2\ph(\Om),\,d\}\,\frac tn+d\,\frac{t^2}{n^2}\right\}.
  \label{supp-eq:per-lag-sum}
\end{equation}
\end{lemma}

\begin{proof}
Hoeffding's inequality applied across the $n$ snapshots, followed by a union bound over $s=1,\ldots,d-1$, gives \eqref{supp-eq:entry-event-oracle} with probability at least $1-e^{-t}$ after adjusting constants.

For \eqref{supp-eq:per-lag-sum}, collect one snapshot's lag deviations into the vector $Z^{(\ell)}=(Y_s^{(\ell)}-c_s)_{s=1}^{d-1}\in\mathbb{R}^{d-1}$, where $Y_s^{(\ell)}=(v^{(\ell)})^\top A_sv^{(\ell)}$ and $A_s$ is the symmetric hollow matrix with entries $(A_s)_{jk}=(A_s)_{kj}=(2q_s)^{-1}$ on $(j,k)\in\Om_s$ and zero otherwise, so that $\Fnorm{A_s}^2=(2q_s)^{-1}$ and $\widehat c_s-c_s=n^{-1}\sum_\ell(Y_s^{(\ell)}-c_s)$.
Every active pair of $A_s$ has Gaussian correlation $\rho_s$ with $|\rho_s|\le1-2\eps$, so the contraction theorem of the main text, in its threshold-sign form, gives $\Var(Y_s^{(\ell)})\le C_\eps\kobs^2q_s^{-1}$; since also $|Y_s^{(\ell)}-c_s|\le8$, each lag satisfies $\Var(Y_s^{(\ell)})\le C_\eps\min\{\kobs^2q_s^{-1},1\}$.
Hence
\[
  \E\norm{Z^{(\ell)}}_2^2=\sum_{s=1}^{d-1}\Var(Y_s^{(\ell)})
  \le C_\eps\sum_{s=1}^{d-1}\min\{\kobs^2q_s^{-1},1\}
  \le C_\eps\min\{\kobs^2\ph(\Om),d\},
  \qquad
  \norm{Z^{(\ell)}}_2\le8\sqrt{d-1}.
\]
The $Z^{(\ell)}$ are independent and centered, so the Bernstein inequality for Hilbert-space-valued sums \citep{Pinelis1994} gives, with probability at least $1-e^{-t}$,
\[
  \norm{\widehat c-c}_2
  =\Big(\sum_{s=1}^{d-1}(\widehat c_s-c_s)^2\Big)^{1/2}
  \le C\left(\sqrt{\frac{\E\norm{Z^{(\ell)}}_2^2\,t}{n}}+\sqrt{d-1}\,\frac tn\right).
\]
Squaring and inserting the variance bound gives \eqref{supp-eq:per-lag-sum}; after adjusting the constant $C$ in $t=\log(Cd/\delta)$ both displayed events hold simultaneously with probability at least $1-e^{-t}$.
\end{proof}

\begin{proof}[Proof of the oracle upper bound]
Work on the event of Lemma~\ref{supp-lem:per-lag}.
We first dispose of large deviations.
If $C\sqrt{t/n}>b_\eps/2$, with $b_\eps$ the buffer constant of Lemma~\ref{supp-lem:inverse-link}(ii), then $t/n$ exceeds a fixed constant $c_\eps$.
The clipped oracle estimator satisfies $\widehat\rho_s^{\rm or}\in[-1+\eps,1-\eps]$ while $|\rho_s|\le1-2\eps$, so $|\widehat\rho_s^{\rm or}-\rho_s|\le2$, and the Toeplitz row-sum bound gives
\[
\norm{\widehat\Gamma_{\rm or}-\Gamma}_2
\le2\gamma_0\sum_{s=1}^{d-1}|\widehat\rho_s^{\rm or}-\rho_s|
\le4\gamma_0d
\le C_\eps\gamma_0 d\,\frac tn ,
\]
so the stated bound holds after enlarging the constant of the $L_1dt/n$ term.
We may therefore assume
\[
C\sqrt{t/n}\le b_\eps/2 .
\]
In that case Lemma~\ref{supp-lem:inverse-link}(ii) with $\tau'=\tau$ shows that, for every $s$ and every $\theta\in[0,1]$, the point $c_s+\theta(\widehat c_s-c_s)$ lies in $\mathcal I_\tau$, so the clipping is inactive, $\Pi_{\mathcal I_\tau}(\widehat c_s)=\widehat c_s$.
Lemma~\ref{supp-lem:completion} with $\eta=0$ and $e_s=\widehat c_s-c_s$ (so that $a_0\equiv0$) gives
\begin{equation}
  \widehat\gamma^{\rm or}_s-\gamma_s
  =\gamma_0\bigl[\psi(\widehat c_s;\tau)-\rho_s\bigr]
  =w_s(\widehat c_s-c_s)+r_s,
  \qquad
  |r_s|\le C_{\tau,\eps}\gamma_0(\widehat c_s-c_s)^2,
  \label{supp-eq:taylor-oracle}
\end{equation}
where $\widehat\gamma^{\rm or}_s=\gamma_0\psi(\widehat c_s;\tau)$, $\gamma_s=\gamma_0\rho_s$ and $w_s=\gamma_0\psi'(c_s;\tau)$.

The first term in \eqref{supp-eq:taylor-oracle} generates the polynomial
\begin{equation}
L_n(\theta)=\frac1n\sum_{\ell=1}^nG_\theta^{(\ell)},
\qquad
G_\theta^{(\ell)}
=2\operatorname{Re}\sum_{s=1}^{d-1}w_se^{2\pi is\theta}\bigl\{Y_s^{(\ell)}-c_s\bigr\},
\end{equation}
with deterministic weights $w_s=\gamma_0\psi'(c_s;\tau)$, so that $|w_s|\le\gamma_0L_1$.
Writing $A_\theta(w)$ for the hollow frequency matrix of the main text, $G_\theta^{(\ell)}=(v^{(\ell)})^{*}A_\theta(w)v^{(\ell)}-\E\,(v^{(\ell)})^{*}A_\theta(w)v^{(\ell)}$, so the sparse-ruler spectral variance corollary in the main text gives
\begin{equation}
\Var(G_\theta^{(\ell)})
  \le C_\eps \gamma_0^2L_1^2\kappa_{\rm obs}^2\varphi(\Omega),
\end{equation}
and boundedness gives $|G_\theta^{(\ell)}|\le C\gamma_0L_1d$.
Bernstein's inequality at a fixed $\theta$ yields
\begin{equation}
|L_n(\theta)|
  \le C\gamma_0L_1\kappa_{\rm obs}\sqrt{\frac{\varphi(\Omega)t}{n}}
    +C\gamma_0L_1d\frac{t}{n}.
\end{equation}
The boundedness term is retained in the theorem statement.
Applying Lemma~\ref{supp-lem:grid} with $t=\log(Cd/\delta)$ gives the same bound uniformly over $\theta$ with probability at least $1-\delta$.

The remainder term is controlled by the coverage-scale square sum \eqref{supp-eq:per-lag-sum}:
\begin{equation}
\sup_{\theta\in[0,1]}\left|
  \operatorname{Re}\sum_{s=1}^{d-1}
  r_s\, e^{2\pi is\theta}
  \right|
  \le \sum_{s=1}^{d-1}|r_s|
  \le C_{\tau,\eps}\gamma_0\sum_{s=1}^{d-1}(\widehat c_s-c_s)^2
  \le C_{\tau,\eps}\gamma_0\left\{\min\{\kobs^2\ph(\Om),d\}\frac{t}{n}+d\frac{t^2}{n^2}\right\}.
\end{equation}
Since the oracle estimator has exact diagonal, the error matrix is the hollow real symmetric Toeplitz matrix with lag coefficients $\widehat\gamma_s^{\rm or}-\gamma_s$, and the symbol bound of Lemma~\ref{supp-lem:toeplitz-symbol} converts the first- and second-order bounds into the oracle rate of the main upper-bound theorem.

For the probability accounting, the event of Lemma~\ref{supp-lem:per-lag} fails with probability at most $e^{-t}$, and the Bernstein--grid event for the first-order spectral polynomial fails with probability at most $Cde^{-t}$.
With $t=\log(Cd/\delta)$ and the numerical constant $C$ chosen large enough, their union fails with probability at most $\delta$, so the oracle bound holds with probability at least $1-\delta$.
\end{proof}

\subsection{Plug-in empirical centering and completion}
We prove part \textup{(b)}.
Let $\Delta_\mu=\widehat\mu-\mu$.
For a pair $(j,k)$,
\begin{equation}
  (u_j-\widehat\mu)(u_k-\widehat\mu)
  =v_jv_k-\Delta_\mu(v_j+v_k)+\Delta_\mu^2.
  \label{supp-eq:plugin-pair-identity}
\end{equation}
Averaged over the $n$ snapshots and the $q_s$ pairs of lag $s$, the linear term generates
\begin{equation}
  \ell_{s,n}
  =\frac{1}{nq_s}\sum_{\ell=1}^n\sum_{(j,k)\in\Omega_s}
  \bigl(v^{(\ell)}_j+v^{(\ell)}_k\bigr),
  \qquad
  |\ell_{s,n}|\le4 .
  \label{supp-eq:ell-def}
\end{equation}
No independence between $\Delta_\mu$ and the empirical row-sum process is used; the next lemma states the calibration bound and the recentering bound on a single event, on which the remaining argument is deterministic.

\begin{lemma}[Spectral plug-in centering control]
\label{supp-lem:plugin-centering}
Let $H_n(\theta)=2\operatorname{Re}\sum_{s=1}^{d-1}w_s\ell_{s,n}e^{2\pi is\theta}$ with deterministic weights $|w_s|\le W$.
Under the plug-in sample-size regime of the main upper-bound theorem, with probability at least $1-Ce^{-t}$ the calibration bound \eqref{supp-eq:pooled-cal-event} and the recentering bound
\begin{equation}
|\Delta_\mu|\sup_{\theta\in[0,1]}|H_n(\theta)|
\le
C W\kappa_{\rm obs}\sqrt{\frac{\varphi(\Omega)t}{n}}
  +C Wd\frac{t}{n}
\label{supp-eq:plugin-centering-bound}
\end{equation}
hold simultaneously.
\end{lemma}

\begin{proof}
Throughout the proof we may assume $t\le n$: the linear spectral polynomial has lag coefficients bounded by $4W$, because each lag average of $v_j+v_k$ is at most $4$ in absolute value and $|w_s|\le W$, so its supremum over $\theta$ is at most $8Wd$ deterministically, while $|\Delta_\mu|\le2$; hence for $t>n$ the claimed bound holds trivially after enlarging the constant of the $Wdt/n$ term.
We also use repeatedly that $\kappa_{\rm obs}\le m$, because the observed correlation matrix is positive semidefinite with unit diagonal, so its operator norm is at most its trace.

For one snapshot, the linear sparse-ruler polynomial is the exact row-sum statistic
\[
H^{(\ell)}(\theta)
=2\operatorname{Re}\sum_{s=1}^{d-1}\frac{w_se^{2\pi is\theta}}{q_s}
  \sum_{(j,k)\in\Omega_s}\bigl\{v^{(\ell)}_j+v^{(\ell)}_k\bigr\}
=\sum_{i\in\Omega}\alpha_i(\theta)v^{(\ell)}_i,
\]
where, writing $m=|\Omega|$,
\[
\alpha_i(\theta)=2\operatorname{Re}\sum_{s=1}^{d-1}
    \frac{w_s e^{2\pi i s\theta}}{q_s}c_{i,s},
\]
and $c_{i,s}\in\{0,1,2\}$ is the number of lag-$s$ observed pairs incident to $i$.
Let
\[
    A_i=\{s:c_{i,s}>0\}.
\]
Since a fixed vertex in $\Omega$ can be paired with at most $m-1$ other vertices,
\[
    |A_i|\le \sum_s c_{i,s}\le m-1 .
\]
Therefore, using $|2\operatorname{Re}z|\le2|z|$ and Cauchy--Schwarz over the active lags of vertex $i$,
\[
|\alpha_i(\theta)|^2
 \le 4|A_i|W^2\sum_s\frac{c_{i,s}^2}{q_s^2}
 \le 4(m-1)W^2\sum_s\frac{c_{i,s}^2}{q_s^2}.
\]
For each lag $s$, the total incidence count is $\sum_{i\in\Omega}c_{i,s}=2q_s$, and since $c_{i,s}\le2$,
\[
\sum_{i\in\Omega}c_{i,s}^2
   \le 2\sum_{i\in\Omega}c_{i,s}=4q_s.
\]
Consequently, uniformly in $\theta$,
\begin{equation}
\sum_{i\in\Omega}|\alpha_i(\theta)|^2
 \le 4(m-1)W^2\sum_{s=1}^{d-1}\frac{1}{q_s^2}
      \sum_{i\in\Omega}c_{i,s}^2
 \le 16(m-1)W^2\sum_{s=1}^{d-1}\frac1{q_s}
 \le C W^2|\Omega|\,\varphi(\Omega).
\label{supp-eq:active-lag-row-energy}
\end{equation}
By Lemma~\ref{supp-lem:centered-transform-covariance}, applied to the centered threshold-sign transform on the observed covariance, $\norm{\Cov(v_\Omega)}_2\le C\kappa_{\rm obs}$.
Together with \eqref{supp-eq:active-lag-row-energy}, this gives the one-snapshot variance bound $C W^2\kappa_{\rm obs}|\Omega|\varphi(\Omega)$.
Bernstein's inequality, followed by the trigonometric grid lemma, controls the averaged linear polynomial by
\begin{equation}
  \sup_{\theta\in[0,1]}|H_n(\theta)|
  \le
  C W\sqrt{\frac{\kappa_{\rm obs}|\Omega|\varphi(\Omega)t}{n}}
  +C Wd\frac{t}{n}.
  \label{supp-eq:linear-rowsum-intermediate}
\end{equation}
Combining \eqref{supp-eq:linear-rowsum-intermediate} with
\begin{equation}
|\Delta_\mu|
  \le C\left(\sqrt{\frac{\kappa_{\rm obs}t}{n|\Omega|}}+\frac{t}{n}\right)
\end{equation}
on the calibration event gives
\begin{equation}
|\Delta_\mu|\sup_{\theta\in[0,1]}|H_n(\theta)|
  \le
  C W
  \left(
  \sqrt{\frac{\kappa_{\rm obs}t}{n|\Omega|}}+\frac{t}{n}
  \right)
  \left(
  \sqrt{\frac{\kappa_{\rm obs}|\Omega|\varphi(\Omega)t}{n}}
  +d\frac{t}{n}
  \right),
\label{supp-eq:plugin-centering-products}
\end{equation}
where constants may change from line to line.
The leading product is
\[
\sqrt{\frac{\kappa_{\rm obs}t}{n|\Omega|}}
\sqrt{\frac{\kappa_{\rm obs}|\Omega|\varphi(\Omega)t}{n}}
=
\kappa_{\rm obs}\sqrt{\varphi(\Omega)}\,\frac{t}{n}
\le
\kappa_{\rm obs}\sqrt{\frac{\varphi(\Omega)t}{n}},
\]
because $t\le n$ on the stated sample-size event after adjusting constants.
The product involving the boundedness term satisfies
\[
\sqrt{\frac{\kappa_{\rm obs}t}{n|\Omega|}}\,
d\frac{t}{n}
\le
d\frac{t}{n},
\]
using $n\ge C|\Omega|t/\kappa_{\rm obs}$ and $\kappa_{\rm obs}\le|\Omega|$.
The remaining two terms are
\[
\frac{t}{n}
\sqrt{\frac{\kappa_{\rm obs}|\Omega|\varphi(\Omega)t}{n}}
\le
\kappa_{\rm obs}\sqrt{\frac{\varphi(\Omega)t}{n}},
\qquad
\frac{t}{n}\,d\frac{t}{n}
\le
d\frac{t}{n},
\]
where we use $n\ge C|\Omega|t/\kappa_{\rm obs}$, $t\le n$ and $\kappa_{\rm obs}\le |\Omega|$.
Substituting these four bounds into \eqref{supp-eq:plugin-centering-products} yields
\begin{equation}
|\Delta_\mu|\sup_{\theta\in[0,1]}|H_n(\theta)|
\le
C W\kappa_{\rm obs}\sqrt{\frac{\varphi(\Omega)t}{n}}
  +C Wd\frac{t}{n},
\end{equation}
which is \eqref{supp-eq:plugin-centering-bound}.
The calibration event fails with probability at most $e^{-t}$ and the Bernstein--grid event behind \eqref{supp-eq:linear-rowsum-intermediate} fails with probability at most $Ce^{-t}$, so the two bounds hold simultaneously with probability at least $1-Ce^{-t}$ after adjusting constants.
\end{proof}

\begin{proof}[Completion of proof of the plug-in upper bound]
Let $\Delta_\mu=\widehat\mu-\mu$ and $\eta=\widehat\gamma_0/\gamma_0-1$.
Work on the intersection of the calibration event \eqref{supp-eq:pooled-cal-event} and the event of Lemma~\ref{supp-lem:per-lag}; the proof of Lemma~\ref{supp-lem:pooled-calibration} bounds the intermediate threshold estimate as well, so on this event
\[
|\widehat\tau-\tau|+|\Delta_\mu|+|\eta|\le Cr_\mu\le Cc\eta_0 .
\]
For each lag, the unclipped plug-in lag statistic and the pair identity \eqref{supp-eq:plugin-pair-identity} give
\[
\delta_s
:=\widehat c^{\rm plug,raw}_s-c_s
=\Delta^{\rm or}_s-\Delta_\mu\ell_{s,n}+\Delta_\mu^2,
\qquad
\Delta^{\rm or}_s=\widehat c^{\rm or}_s-c_s,
\]
where
\[
\widehat c^{\rm plug,raw}_s
=\frac{1}{nq_s}\sum_{\ell=1}^n\sum_{(j,k)\in\Omega_s}
(Y^{(\ell)}_j-\widehat\mu)(Y^{(\ell)}_k-\widehat\mu)
\]
and $\ell_{s,n}$ is as in \eqref{supp-eq:ell-def}.
On the working event,
\[
\max_s|\delta_s|
\le C\sqrt{t/n}+4Cr_\mu+C^2r_\mu^2
\le C_0\bigl(\sqrt{t/n}+r_\mu\bigr),
\]
after decreasing $c$ so that $Cr_\mu\le1$.

\parhead{Large deviations and clipping} After further decreasing the constant $c$ in the condition $r_\mu\le c\eta_0$, depending only on the regularity domains, we may assume that on the working event $C_0r_\mu\le b_\eps/4$ and $|\widehat\tau-\tau|\le\eta_\eps$, with $b_\eps,\eta_\eps$ from Lemma~\ref{supp-lem:inverse-link}(ii).
Suppose first that $C_0\sqrt{t/n}>b_\eps/4$, so that $t/n$ exceeds a fixed constant.
Since $\widehat\tau\in[\tau_{\min},\tau_{\max}]$ by construction, $\widehat\gamma_0=\gamma_0(\tau/\widehat\tau)^2\le C_\tau\gamma_0$, while the clipped inverse link gives $|\widehat\rho^{\rm plug}_s|\le1-\eps$; hence every entry of $\widehat\Gamma_{\rm plug}$ is bounded by $\widehat\gamma_0\le C_\tau\gamma_0$ and every entry of $\Gamma$ by $\gamma_0$.
The row-sum bound for symmetric matrices then gives
\[
\norm{\widehat\Gamma_{\rm plug}}_2
\le d\max_{j,k}|(\widehat\Gamma_{\rm plug})_{jk}|
\le C_\tau\gamma_0 d,
\qquad
\norm{\Gamma}_2
\le d\max_{j,k}|\Gamma_{jk}|
\le\gamma_0 d,
\]
and therefore the deterministic bound
\[
\norm{\widehat\Gamma_{\rm plug}-\Gamma}_2
\le\norm{\widehat\Gamma_{\rm plug}}_2+\norm{\Gamma}_2
\le C_\tau\gamma_0 d
\le C_{\tau,\eps}\gamma_0 d\,\frac tn ,
\]
so the stated bound holds trivially.
The same deterministic bound settles the case $t>n$.
From now on we may therefore assume
\[
C_0\bigl(\sqrt{t/n}+r_\mu\bigr)\le b_\eps/2,
\qquad
t\le n .
\]
By Lemma~\ref{supp-lem:inverse-link}(ii) with $\tau'=\widehat\tau$, for every $s$ and every $\theta\in[0,1]$ the point $c_s+\theta\delta_s$ lies in $\mathcal I_{\widehat\tau}$, so the clipping is inactive, $\Pi_{\mathcal I_{\widehat\tau}}(\widehat c^{\rm plug,raw}_s)=\widehat c^{\rm plug,raw}_s$.
Since $\widehat\gamma_0=\gamma_0(1+\eta)$ and $\widehat\tau=\tau(1+\eta)^{-1/2}=\tau_\eta$ exactly, Lemma~\ref{supp-lem:completion} with $e_s=\delta_s$ gives
\begin{equation}
\widehat\gamma^{\rm plug}_s-\gamma_s
=w_s\delta_s+\gamma_0a_\eta(\rho_s)+r_s,
\qquad
|r_s|\le C_{\tau,\eps}\gamma_0\bigl(\delta_s^2+|\eta|\,|\delta_s|\bigr),
\label{supp-eq:plugin-taylor}
\end{equation}
where $\widehat\gamma^{\rm plug}_s=\widehat\gamma_0\psi(\widehat c^{\rm plug,raw}_s;\widehat\tau)$, $\gamma_s=\gamma_0\rho_s$, $w_s=\gamma_0\psi'(c_s;\tau)$ with $|w_s|\le\gamma_0L_1$, and the diagonal of the estimator contributes $\gamma_0\eta$ separately.

\parhead{Population term} On the working event $|\eta|\le Cr_\mu\le c\eta_0$, so Corollary~\ref{lem:population-calibration} gives
\begin{equation}
\norm{\gamma_0\eta \Id_d+\gamma_0\Toep_d^0\{a_\eta(\rho_s)\}_{s=1}^{d-1}}_2
  \le C_{\tau,\eps}\gamma_0\{1+S_1(d;\rho)\}r_\mu ,
\end{equation}
the calibration term in \eqref{eq:plugin-rate-s}; the diagonal contribution $\gamma_0\eta$ is included here.

\parhead{First-order term} The first-order term carries the true weights $w_s=\gamma_0\psi'(c_s;\tau)$, $|w_s|\le\gamma_0L_1$, the weight-estimation error having been folded into the remainder $r_s$.
Substituting $\delta_s=\Delta^{\rm or}_s-\Delta_\mu\ell_{s,n}+\Delta_\mu^2$,
\begin{align*}
2\operatorname{Re}\sum_{s=1}^{d-1}w_s\delta_se^{2\pi is\theta}
&=2\operatorname{Re}\sum_sw_s\Delta^{\rm or}_se^{2\pi is\theta}
-\Delta_\mu\,2\operatorname{Re}\sum_sw_s\ell_{s,n}e^{2\pi is\theta}
+\Delta_\mu^2\,2\operatorname{Re}\sum_sw_se^{2\pi is\theta}.
\end{align*}
The first term is exactly the oracle spectral polynomial and contributes the first two terms of \eqref{eq:plugin-rate-s}; the second is controlled by Lemma~\ref{supp-lem:plugin-centering} with $W=\gamma_0L_1$ and contributes the same two scales; the third has supremum over $\theta$ at most $2\gamma_0L_1d\,\Delta_\mu^2\le C\gamma_0L_1d\,t/n$, because $\Delta_\mu^2\le Cr_\mu^2\le C\{\kobs t/(nm)+t^2/n^2\}\le Ct/n$ using $\kobs\le m$ and $t\le n$.

\parhead{Remainder} By \eqref{supp-eq:plugin-taylor} the lag-$s$ remainder obeys $|r_s|\le C_{\tau,\eps}\gamma_0(\delta_s^2+|\eta||\delta_s|)$, so its Toeplitz symbol is at most $C_{\tau,\eps}\gamma_0(\sum_s\delta_s^2+|\eta|\sum_s|\delta_s|)$.
For the square sum, $\delta_s^2\le3(\Delta^{\rm or}_s)^2+3\Delta_\mu^2\ell_{s,n}^2+3\Delta_\mu^4$, so the coverage-scale square sum \eqref{supp-eq:per-lag-sum} together with $\sum_s\ell_{s,n}^2\le16(d-1)$ (as $|\ell_{s,n}|\le4$), $\ph(\Om)\ge(d-1)/m$ (as $q_s\le m$) and $\kobs\ge1$ gives
\[
\sum_{s=1}^{d-1}\delta_s^2
\le C\left\{\min\{\kobs^2\ph(\Om),d\}\frac tn+d\frac{t^2}{n^2}\right\},
\]
where $(d-1)\Delta_\mu^2\le C(d-1)\{\kobs t/(nm)+t^2/n^2\}$ enters the two branches through $(d-1)\kobs/m\le\kobs\ph(\Om)\le\kobs^2\ph(\Om)$ and $(d-1)\kobs/m\le d$, and the $\Delta_\mu^4$ term obeys the same bound since $\Delta_\mu^2\le C$.
For the cross term, Cauchy--Schwarz gives $|\eta|\sum_s|\delta_s|\le|\eta|\sqrt{d-1}\,(\sum_s\delta_s^2)^{1/2}\le\tfrac12(d-1)\eta^2+\tfrac12\sum_s\delta_s^2$; the second summand is already at the displayed scale, and $(d-1)\eta^2\le C(d-1)\{\kobs t/(nm)+t^2/n^2\}$ (from $\eta^2\le Cr_\mu^2$) obeys the same bound as the $(d-1)\Delta_\mu^2$ term above.
Therefore the remainder is at most $C_{\tau,\eps}\gamma_0\{\min\{\kobs^2\ph(\Om),d\}\,t/n+d\,t^2/n^2\}$.

\parhead{Combination} Combining the population term, the three first-order bounds and the remainder through the symbol bound of Lemma~\ref{supp-lem:toeplitz-symbol}, and recalling that the diagonal error $\gamma_0\eta$ is already counted in the population term, proves the real Toeplitz statement \eqref{eq:plugin-rate-s}.
\end{proof}

\paragraph{Probability accounting}
The plug-in proof works on the intersection of four events: the calibration event \eqref{supp-eq:pooled-cal-event}, the lag event of Lemma~\ref{supp-lem:per-lag}, the oracle spectral event produced by Bernstein's inequality and Lemma~\ref{supp-lem:grid}, and the recentering event of Lemma~\ref{supp-lem:plugin-centering}.
At level $t$, each fails with probability at most $C'de^{-t}$ for a constant $C'$ depending only on the regularity domains, so the intersection fails with probability at most $4C'de^{-t}$.
Choosing the constant $C$ in $t=\log(Cd/\delta)$ with $C\ge4C'$ makes this at most $\delta$.
All bounds in the proof hold simultaneously on this intersection.

\begin{proof}[Proof of the operator-norm-ball corollary of the main text]
The completion proof above uses Assumption~\ref{ass:short-memory-s} at exactly one point: the application of Lemma~\ref{lem:population-calibration} to the population perturbation $\gamma_0\eta\Id_d+\gamma_0\Toep_d^0\{a_\eta(\rho_s)\}$.
Replace that application by Lemma~\ref{supp-lem:hadamard-perturbation}, which applies because $\opnorm{T_d^0(\rho)}\le c_0\le1-2\eps$; after decreasing the constant $c$ in $r_\mu\le c\eta_0$ so that $Cr_\mu\le\eta_0'$ on the working event, the bound \eqref{supp-eq:no-S1} with $|\eta|\le Cr_\mu$ gives a contribution
\[
C(1+B)\,\gamma_0\,r_\mu
\;+\;
C_{\tau,\eps}\,\gamma_0\,d\,r_\mu^2 .
\]
For the quadratic part, $r_\mu^2\le C\{\kobs t/(nm)+t^2/n^2\}$, and exactly as in the second-order step --- using $\ph(\Om)\ge(d-1)/m$, $1\le\kobs\le m$ ---
\[
d\,r_\mu^2
\le Cd\left\{\frac{\kobs t}{nm}+\frac{t^2}{n^2}\right\}
\le C\left\{\min\{\kobs^2\ph(\Om),d\}\,\frac tn+d\,\frac{t^2}{n^2}\right\},
\]
which is absorbed into the curvature term of \eqref{eq:plugin-rate-s}.
No other step of the proof uses Assumption~\ref{ass:short-memory-s}, which proves the corollary.
\end{proof}

\begin{corollary}[Sobolev spectral-density class]
\label{cor:sobolev-plugin}
Assume the conditions of the main upper-bound theorem.
If, for some $\beta>1/2$,
\[
    \sum_{s=1}^{d-1}s^{2\beta}|\rho_s|^2\le A_\beta^2,
\]
then the plug-in bound \eqref{eq:plugin-rate-s} holds with the calibration term bounded by
\[
    C_{\tau,\eps,\beta}\gamma_0(1+A_\beta)
    \left\{\sqrt{\frac{\kappa_{\rm obs}t}{nm}}+\frac{t}{n}\right\}.
\]
\end{corollary}

\begin{proof}
By Cauchy--Schwarz,
\[
    S_1(d;\rho)
    \le
    \left(\sum_{s=1}^{\infty}s^{-2\beta}\right)^{1/2}
    \left(\sum_{s=1}^{d-1}s^{2\beta}|\rho_s|^2\right)^{1/2}
    \le C_\beta A_\beta .
\]
Substitution into the main upper-bound theorem gives the claim.
\end{proof}

\paragraph{Other short-memory examples}
The same substitution gives the standard exponential and finite-memory cases.
If $|\rho_s|\le C_0a^s$ with $0<a<1$, then
\[
    S_1(d;\rho)\le \frac{C_0a}{1-a},
\]
so the calibration multiplier is $1+C_0a/(1-a)$.
If $\rho_s=0$ for all $s>K$, then
\[
    S_1(d;\rho)=\sum_{s=1}^{K}|\rho_s|\le K,
\]
so the calibration multiplier is at most $1+K$.
In each case the plug-in term is obtained by substituting the displayed bound on $S_1(d;\rho)$ into \eqref{eq:plugin-rate-s}.

\begin{remark}[Bounded spectrum is not enough for plug-in calibration]
\label{rem:bounded-spectrum}
Bounded spectral density controls $\norm{\Gamma}_2/\gamma_0$, but it does not by itself control nonlinear lagwise calibration perturbations.
A scalar nonlinear transformation of autocorrelation lags need not preserve Toeplitz positive definiteness or dimension-free spectral norm.
The plug-in theorem is therefore stated over short-memory and Sobolev-type classes, where the population calibration shift is controlled directly by a Toeplitz row-sum argument.
\end{remark}

The upper bound separates three effects: sparse pair sampling through $\varphi(\Omega)$, inverse-link conditioning through $L_1,L_2$, and plug-in scale calibration through the pooled marginal error $r_\mu$ multiplied by the short-memory size $1+S_1(d;\rho)$.
On standard short-memory, Sobolev, banded and stable ARMA classes, this calibration term is a low-dimensional pooled marginal estimation error rather than a sparse-ruler pair-sampling error.

\section{Proof of the real spectral-packing lower bound}
\label{supp-sec:lower-bounds}

The lower bounds prove only the intrinsic design complexity of the leading oracle term.
Throughout this section the scale is known and the parameter space is a small identity-neighborhood submodel.
Marginal calibration, inverse-link curvature and correlation conditioning are therefore absent by construction.
The proof consists of a KL upper bound in the sparse Frobenius metric and a deterministic spectral packing whose separation is measured in Toeplitz operator norm.

We first give a global operator-norm minimax lower bound in terms of a deterministic Toeplitz spectral-packing quantity.
A coverage-log corollary then recovers the scale $\gamma_0\sqrt{\varphi(\Omega)\log d/n}$ under a natural balanced real trigonometric packing condition on the aggregation profile.

For $b=(b_1,\ldots,b_{d-1})\in\R^{d-1}$, the hollow real symmetric Toeplitz matrix is
\begin{equation}
T_d^0(b)_{jk}=
\begin{cases}
b_{|j-k|},& j\ne k,\\
0,& j=k.
\end{cases}
\label{eq:hollow-real-toeplitz}
\end{equation}
In the known-scale real one-bit Toeplitz model, let
\begin{equation}
\mathcal P_{\mathbb R}(c_0)=
\left\{
\Gamma=\gamma_0\{I_d+T_d^0(\rho)\}:
I_d+T_d^0(\rho)\succeq0,\ 
\norm{T_d^0(\rho)}_2\le c_0
\right\},
\label{eq:global-parameter-class-s}
\end{equation}
where $0<c_0<1/2$ is fixed and observations are restricted to the sparse ruler $\Omega$.
The corresponding minimax risk is
\begin{equation}
\mathfrak R_n(\Omega,\mathcal P_{\mathbb R}(c_0))
=
\inf_{\widehat \Gamma}
\sup_{\Gamma\in\mathcal P_{\mathbb R}(c_0)}
\E_\Gamma\norm{\widehat \Gamma-\Gamma}_2.
\label{eq:minimax-risk-s}
\end{equation}

\begin{lemma}[Sparse-pair Hermitian lifting]
\label{supp-lem:lifting}
Let $\Omega$ be a ruler.
For lag weights $\beta=(\beta_s)_{s=1}^{d-1}\in\mathbb{C}^{d-1}$, define the upper-triangular lifting $U_\Omega(\beta)$ by $(U_\Omega(\beta))_{jk}=\beta_{k-j}$ when $j<k$ and $(j,k)\in\Omega_{k-j}$, and zero otherwise; set $H_\Omega(\beta)=U_\Omega(\beta)+U_\Omega(\beta)^*$.
Then $H_\Omega(\beta)$ is hollow Hermitian and
\begin{equation}
  \norm{H_\Omega(\beta)}_F^2=2\sum_{s=1}^{d-1}q_s|\beta_s|^2 .
  \label{supp-eq:lifting-frobenius}
\end{equation}
If $\beta$ is real, then $H_\Omega(\beta)$ is symmetric and its row sums $r=H_\Omega(\beta)\one$ satisfy
\begin{equation}
  \sum_{i\in\Omega}r_i=2\sum_{s=1}^{d-1}q_s\beta_s,
  \qquad
  \sum_{i\in\Omega}r_i^2\ge\frac1m\Big(\sum_{i\in\Omega}r_i\Big)^2,
  \qquad m=|\Omega| .
  \label{supp-eq:lifting-rowsum}
\end{equation}
\end{lemma}

\begin{proof}
Each unordered pair $(j,k)$ with $k-j=s$ places $\beta_s$ in entry $(j,k)$ and $\overline{\beta_s}$ in entry $(k,j)$, contributing $2|\beta_s|^2$ to $\norm{H_\Omega(\beta)}_F^2$; summing over the $q_s$ pairs at each lag gives \eqref{supp-eq:lifting-frobenius}.
For real $\beta$ the matrix is symmetric, $r_i=\sum_{j\ne i}(H_\Omega(\beta))_{ij}$, and $\sum_i r_i=2\sum_{i<j}(H_\Omega(\beta))_{ij}=2\sum_s q_s\beta_s$; the inequality is Cauchy--Schwarz over the $m$ entries of $r$.
\end{proof}

\begin{lemma}[KL control by the sparse Frobenius metric]
\label{lem:kl-sparse-frobenius}
Let $B=T_d^0(b)$ be hollow real symmetric Toeplitz.
Let $P_{uB}^{(n)}$ denote the law of $n$ one-bit sparse observations with covariance $\gamma_0(I_d+uB)$, and let $P_0^{(n)}$ denote the law under covariance $\gamma_0 I_d$.
If $|u|\norm{B}_2\le1/2$, then
\begin{equation}
D(P_{uB}^{(n)}\|P_0^{(n)})
\le
Cnu^2\norm{B_{\Omega,\Omega}}_F^2.
\label{eq:kl-sparse-frobenius}
\end{equation}
Moreover,
\begin{equation}
\norm{B_{\Omega,\Omega}}_F^2
=2\sum_{s=1}^{d-1}q_sb_s^2.
\label{eq:sparse-frobenius-toeplitz}
\end{equation}
\end{lemma}

\begin{proof}
Let $\widetilde P_{uB}^{(n)}$ and $\widetilde P_0^{(n)}$ be the corresponding unquantized real Gaussian laws on the observed coordinates.
The one-bit observation is a measurable function of the unquantized vector, so the data-processing inequality gives
\begin{equation}
D(P_{uB}^{(n)}\|P_0^{(n)})
\le
D(\widetilde P_{uB}^{(n)}\|\widetilde P_0^{(n)}).
\end{equation}
The quantizer and threshold are fixed throughout this oracle submodel; the data-processing bound is threshold-independent because thresholding is a measurable map of the observed Gaussian vector.
For one snapshot, the real Gaussian KL divergence is given by the standard entropy formula \citep[see, e.g.,][]{CoverThomas2006}:
\begin{equation}
D\!\left(
\mathcal N(0,\gamma_0(I+uB_{\Omega,\Omega}))
\middle\|
\mathcal N(0,\gamma_0 I)
\right)
=
\frac12\left\{
\operatorname{tr}(uB_{\Omega,\Omega})
-\log\det(I+uB_{\Omega,\Omega})
\right\}.
\end{equation}
If $\lambda_1,\ldots,\lambda_m$ are the eigenvalues of $B_{\Omega,\Omega}$, then the hollow structure gives $\sum_{r=1}^m\lambda_r=\operatorname{tr}(B_{\Omega,\Omega})=0$ and hence
\begin{equation}
D
=
\frac12
\sum_{r=1}^m\{u\lambda_r-\log(1+u\lambda_r)\}.
\end{equation}
The condition $|u|\norm{B}_{2}\le1/2$ implies $|u\lambda_r|\le1/2$, and $x-\log(1+x)\le Cx^2$ on this interval. Thus one snapshot has KL divergence at most
\begin{equation}
Cu^2\sum_{r=1}^m\lambda_r^2
=
Cu^2\norm{B_{\Omega,\Omega}}_F^2.
\end{equation}
Independence over snapshots gives \eqref{eq:kl-sparse-frobenius}.
Finally, $B_{\Omega,\Omega}=H_\Omega(b)$ is the Hermitian lifting of the real lag weights $b$, so \eqref{supp-eq:lifting-frobenius} gives
\begin{equation}
\norm{B_{\Omega,\Omega}}_F^2
=
2\sum_{s=1}^{d-1}q_sb_s^2.
\end{equation}
\end{proof}

\begin{theorem}[Spectral-packing lower bound for the oracle sparse-pair submodel]
\label{thm:spectral-packing-minimax-s}
Let $\mathcal V=\{b^1,\ldots,b^M\}\subset\R^{d-1}$ with $M\ge3$, and define
\begin{align}
D_\Omega(\mathcal V)
&=
2\max_{1\le a\le M}\sum_{s=1}^{d-1}q_s(b_s^a)^2,\\
R_T(\mathcal V)
&=
\max_{1\le a\le M}\norm{T_d^0(b^a)}_2,\\
\Delta_T(\mathcal V)
&=
\min_{a\ne a'}\norm{T_d^0(b^a-b^{a'})}_2.
\end{align}
There is a constant $c>0$, depending only on $c_0$, such that
\begin{equation}
\mathfrak R_n(\Omega,\mathcal P_{\mathbb R}(c_0))
\ge
c\gamma_0
\Delta_T(\mathcal V)
\min\left\{
\frac{1}{R_T(\mathcal V)},\,
\sqrt{\frac{\log M}{nD_\Omega(\mathcal V)}}
\right\}.
\label{eq:spectral-packing-bound-s}
\end{equation}
Consequently, the same lower bound holds after taking the supremum over all finite packings $\mathcal V$.
\end{theorem}

\begin{proof}
Set
\begin{equation}
u=\alpha
\min\left\{
\frac{1}{R_T(\mathcal V)},\,
\sqrt{\frac{\log M}{nD_\Omega(\mathcal V)}}
\right\},
\end{equation}
where $\alpha>0$ will be chosen sufficiently small.
For each $a=1,\ldots,M$, define
\begin{equation}
\Gamma_a=\gamma_0\{I_d+uT_d^0(b^a)\}.
\end{equation}
Since $uR_T(\mathcal V)\le\alpha$, taking $\alpha\le c_0$ ensures $\norm{uT_d^0(b^a)}_2\le c_0$ and $I_d+uT_d^0(b^a)\succeq0$.
Hence $\Gamma_a\in\mathcal P_{\mathbb R}(c_0)$.

Let $P_a^{(n)}$ be the one-bit sparse observation law induced by $\Gamma_a$, and let $P_0^{(n)}$ be the law under $\gamma_0 I_d$.
Lemma~\ref{lem:kl-sparse-frobenius} gives
\begin{align}
D(P_a^{(n)}\|P_0^{(n)})
&\le
Cnu^2\norm{T_d^0(b^a)_{\Omega,\Omega}}_F^2\\
&\le
Cnu^2D_\Omega(\mathcal V)
\le
C\alpha^2\log M.
\end{align}
Choose $\alpha$ so that $C\alpha^2\le1/16$.
Then
\begin{equation}
\frac1M\sum_{a=1}^M D(P_a^{(n)}\|P_0^{(n)})
\le
\frac1{16}\log M.
\end{equation}

For $a\ne a'$,
\begin{equation}
\norm{\Gamma_a-\Gamma_{a'}}_2
=
\gamma_0 u\norm{T_d^0(b^a-b^{a'})}_2
\ge
\gamma_0u\Delta_T(\mathcal V).
\end{equation}
Let $s=\gamma_0u\Delta_T(\mathcal V)/2$.
If an estimator satisfies $\norm{\widehat \Gamma-\Gamma_a}_2<s$, then the nearest-neighbor classifier
\[
\widehat a=\mathop{\rm argmin}_{1\le j\le M}\norm{\widehat \Gamma-\Gamma_j}_2
\]
is correct.
Therefore
\begin{equation}
\Pr_a\{\norm{\widehat \Gamma-\Gamma_a}_2\ge s\}
\ge
\Pr_a\{\widehat a\ne a\}.
\end{equation}
Let $a$ be uniform on $\{1,\ldots,M\}$ and let $I(a;Y)$ denote the mutual information between $a$ and the observed data $Y$ under the joint law.
For any fixed reference probability measure $Q$,
\begin{equation}
I(a;Y)\le\frac1M\sum_{a=1}^M D(P_a^{(n)}\|Q),
\end{equation}
because the mutual information equals the minimum of the right-hand side over $Q$ \citep[Chapter~2]{Tsybakov2009}.
Taking $Q=P_0^{(n)}$ gives $I(a;Y)\le\frac1{16}\log M$.
Fano's inequality \citep{CoverThomas2006} then yields
\begin{equation}
\inf_{\widehat a}\max_{1\le a\le M}\Pr_a\{\widehat a\ne a\}
\ge
1-\frac{I(a;Y)+\log2}{\log M}
\ge
1-\frac1{16}-\frac{\log2}{\log3}
=:c_1,
\end{equation}
and $c_1>0$ is an absolute constant because $M\ge3$.
Thus
\begin{equation}
\sup_{1\le a\le M}\E_a\norm{\widehat \Gamma-\Gamma_a}_2
\ge
c\gamma_0u\Delta_T(\mathcal V).
\end{equation}
Substituting the definition of $u$ proves \eqref{eq:spectral-packing-bound-s}.
\end{proof}

Theorem~\ref{thm:spectral-packing-minimax-s} is a deterministic spectral-packing principle.
The sparse observation law is controlled by the weighted Frobenius metric $D_\Omega(\mathcal V)$, whereas the loss is governed by the Toeplitz operator separation $\Delta_T(\mathcal V)$.
The coverage-log rate arises when this abstract packing is instantiated by real cosine alternatives.
The operator-separation calculation then creates difference-frequency, sum-frequency and doubled-frequency terms that must be controlled.

\begin{assumption}[Real balanced coverage spectral packing]
\label{ass:spectral-packing}
There exist a lag set $S\subset\{1,\ldots,d-1\}$, a frequency set $\Theta\subset[0,1]$ with $|\Theta|=M\ge3$, and constants $a_0,b_0\in(0,1)$ and $\zeta\in(0,1/3)$ such that
\begin{equation}
\varphi_S(\Omega):=\sum_{s\in S}q_s^{-1}\ge a_0\varphi(\Omega),
\qquad
\Phi_S(\Omega):=
\sum_{s\in S}\left(1-\frac{s}{d}\right)q_s^{-1}
\ge b_0\varphi_S(\Omega).
\end{equation}
Let
\begin{equation}
K_S(t)=
\sum_{s\in S}\left(1-\frac{s}{d}\right)q_s^{-1}e^{2\pi i s t}.
\end{equation}
Assume further that
\begin{equation}
\max_{\theta\in\Theta}
\frac{|K_S(2\theta)|}{\Phi_S(\Omega)}
\le \zeta,
\end{equation}
and
\begin{equation}
\max_{\theta\ne\theta'\in\Theta}
\max\left\{
\frac{|K_S(\theta-\theta')|}{\Phi_S(\Omega)},
\frac{|K_S(\theta+\theta')|}{\Phi_S(\Omega)}
\right\}
\le \zeta .
\end{equation}
\end{assumption}

\begin{corollary}[Real coverage-log lower bound]
\label{cor:coverage-log-lower-s}
Under Assumption~\ref{ass:spectral-packing},
\begin{equation}
\mathfrak R_n(\Omega,\mathcal P_{\mathbb R}(c_0))
\ge
c\gamma_0
\min\left\{
1,\,
\sqrt{\frac{\varphi(\Omega)\log M}{n}}
\right\},
\label{eq:coverage-log-lower-s}
\end{equation}
where $c>0$ depends only on $a_0,b_0,\zeta$ and $c_0$.
In particular, if $M\ge d^\alpha$ for a fixed $\alpha>0$, then
\begin{equation}
\mathfrak R_n(\Omega,\mathcal P_{\mathbb R}(c_0))
\ge
c\gamma_0
\min\left\{
1,\,
\sqrt{\frac{\varphi(\Omega)\log d}{n}}
\right\}
\label{eq:coverage-logd-lower-s}
\end{equation}
where in this display $c=c(a_0,b_0,\zeta,c_0,\alpha)>0$.
\end{corollary}

\begin{proof}
For each $\theta\in\Theta$, define the real lag vector
\begin{equation}
b_s^\theta=q_s^{-1}\cos(2\pi s\theta)\mathbf 1_{\{s\in S\}},
\qquad s=1,\ldots,d-1,
\end{equation}
and let $\mathcal V=\{b^\theta:\theta\in\Theta\}$.
First,
\begin{equation}
D_\Omega(\mathcal V)
=
2\max_{\theta\in\Theta}
\sum_{s\in S}q_sq_s^{-2}\cos^2(2\pi s\theta)
\le
2\varphi_S(\Omega).
\end{equation}
Second, by the Toeplitz row-sum bound,
\begin{equation}
R_T(\mathcal V)
\le
2\sum_{s\in S}q_s^{-1}
=
2\varphi_S(\Omega).
\end{equation}

It remains to lower bound the operator separation.
Let
\[
x_\omega=d^{-1/2}
(1,e^{2\pi i\omega},\ldots,e^{2\pi i(d-1)\omega})^\top .
\]
For real $b$,
\begin{equation}
x_\omega^*T_d^0(b)x_\omega
=
2\sum_{s=1}^{d-1}
\left(1-\frac{s}{d}\right)b_s\cos(2\pi s\omega).
\end{equation}
For distinct $\theta,\theta'\in\Theta$, evaluate at $\omega=\theta$.
Using $2\cos A\cos B=\cos(A-B)+\cos(A+B)$ gives
\begin{align}
x_\theta^*T_d^0(b^\theta-b^{\theta'})x_\theta
&=
\Phi_S(\Omega)
+\operatorname{Re}K_S(2\theta) \nonumber\\
&\quad
-\operatorname{Re}K_S(\theta-\theta')
-\operatorname{Re}K_S(\theta+\theta').
\end{align}
Assumption~\ref{ass:spectral-packing} gives
\begin{equation}
x_\theta^*T_d^0(b^\theta-b^{\theta'})x_\theta
\ge
(1-3\zeta)\Phi_S(\Omega).
\end{equation}
Since $\zeta<1/3$,
\begin{equation}
\Delta_T(\mathcal V)
\ge
(1-3\zeta)\Phi_S(\Omega)
\ge
(1-3\zeta)b_0\varphi_S(\Omega).
\end{equation}

Substituting these three bounds into Theorem~\ref{thm:spectral-packing-minimax-s} yields
\begin{equation}
\mathfrak R_n(\Omega,\mathcal P_{\mathbb R}(c_0))
\ge
c\gamma_0
\min\left\{
1,\,
\sqrt{\frac{\varphi_S(\Omega)\log M}{n}}
\right\}.
\end{equation}
Since $\varphi_S(\Omega)\ge a_0\varphi(\Omega)$, \eqref{eq:coverage-log-lower-s} follows.
If $M\ge d^\alpha$, then $\log M\ge\alpha\log d$; absorbing the fixed factor $\sqrt{\alpha}$ into the constant gives \eqref{eq:coverage-logd-lower-s}.
\end{proof}

\begin{lemma}[Large-spectrum avoidance on the torus]
\label{supp-lem:avoidance}
Let $(a_s)_{s\in S}$ be nonnegative reals on a finite set $S\subset\mathbb Z\setminus\{0\}$, and set
\[
  K(t)=\sum_{s\in S}a_se^{2\pi i st}\ \ (t\in\mathbb T=\mathbb R/\mathbb Z),
  \qquad A=\sum_{s\in S}a_s,\qquad B=\sum_{s\in S}a_s^2,\qquad N=\frac{A^2}{B}.
\]
For $\zeta\in(0,1)$ put $\mathcal B_\zeta=\{t\in\mathbb T:|K(t)|>\zeta A\}$.
Then $\mathcal B_\zeta$ is symmetric, $|\mathcal B_\zeta|\le1/(\zeta^2N)$, and there is a set $\Theta\subset\mathbb T$ with $|\Theta|\ge c_\zeta N$ such that
\[
  2\theta\notin\mathcal B_\zeta
  \quad\text{and}\quad
  \theta\pm\theta'\notin\mathcal B_\zeta
  \qquad\text{for all }\theta,\theta'\in\Theta,
\]
provided $N$ exceeds a constant depending only on $\zeta$.
\end{lemma}

\begin{proof}
Parseval's identity gives $\int_{\mathbb T}|K|^2=B$, so Markov's inequality yields $|\mathcal B_\zeta|\le B/(\zeta A)^2=1/(\zeta^2N)$.
The $a_s$ are real, so $K(-t)=\overline{K(t)}$ and $\mathcal B_\zeta$ is symmetric.
Build $\Theta$ greedily: with $\theta_1,\ldots,\theta_k$ already chosen, a new $\theta$ is forbidden if $2\theta\in\mathcal B_\zeta$ or, for some $r\le k$, $\theta-\theta_r\in\mathcal B_\zeta$ or $\theta+\theta_r\in\mathcal B_\zeta$.
The doubling map preserves Haar measure, so $\{\theta:2\theta\in\mathcal B_\zeta\}$ has measure $|\mathcal B_\zeta|$, and each sum or difference constraint is a single translate of $\mathcal B_\zeta$; the forbidden measure is therefore at most $(1+2k)|\mathcal B_\zeta|$.
A further point exists while $(1+2k)|\mathcal B_\zeta|<1$, that is for $k$ up to a constant multiple of $\zeta^2N$, giving $|\Theta|\ge c_\zeta N$ after adjusting constants for small $N$.
\end{proof}

\begin{proposition}[Effective-support certificate for real spectral packing]
\label{prop:effective-support-packing}
Let $S\subset\{1,\ldots,d-1\}$ and set
\begin{equation}
a_s=\left(1-\frac{s}{d}\right)q_s^{-1},
\qquad s\in S.
\end{equation}
Define
\begin{equation}
A_S=\sum_{s\in S}a_s=\Phi_S(\Omega),
\qquad
B_S=\sum_{s\in S}a_s^2,
\qquad
N_{\rm eff}(S)=\frac{A_S^2}{B_S}.
\label{eq:effective-support}
\end{equation}
Fix $\zeta\in(0,1/3)$.
If
\begin{equation}
\varphi_S(\Omega)\ge a_0\varphi(\Omega),
\qquad
\Phi_S(\Omega)\ge b_0\varphi_S(\Omega),
\end{equation}
then there exists a frequency set $\Theta\subset[0,1]$ satisfying Assumption~\ref{ass:spectral-packing} and
\begin{equation}
|\Theta|\ge c_\zeta N_{\rm eff}(S),
\end{equation}
provided $N_{\rm eff}(S)$ is larger than a constant depending only on $\zeta$.
Consequently, if $N_{\rm eff}(S)\ge c_{\rm eff}d^\alpha$ for fixed $c_{\rm eff}>0$ and $\alpha>0$, then $\log M\gtrsim\log d$.
\end{proposition}

\begin{proof}
Apply Lemma~\ref{supp-lem:avoidance} with $a_s=(1-s/d)q_s^{-1}$ for $s\in S$, for which $A=A_S$, $B=B_S$ and $N=N_{\rm eff}(S)$, writing $K_S(t)=\sum_{s\in S}a_se^{2\pi i st}$ and $\mathcal B_\zeta=\{|K_S|>\zeta A_S\}$.
Since $\zeta\in(0,1/3)\subset(0,1)$, the lemma gives $\Theta\subset[0,1]$ with $|\Theta|\ge c_\zeta N_{\rm eff}(S)$ such that $2\theta$, $\theta-\theta'$ and $\theta+\theta'$ avoid $\mathcal B_\zeta$ for all $\theta,\theta'\in\Theta$; that is, $|K_S(2\theta)|,|K_S(\theta-\theta')|,|K_S(\theta+\theta')|\le\zeta A_S$, the three incoherence conditions of Assumption~\ref{ass:spectral-packing}.
Hence $\Theta$ satisfies that assumption with $M=|\Theta|\ge c_\zeta N_{\rm eff}(S)$, and when $N_{\rm eff}(S)\ge c_{\rm eff}d^\alpha$ this gives $\log M\gtrsim\log d$.
\end{proof}

\begin{corollary}[Flat coverage over many lags implies real spectral packing]
\label{cor:flat-coverage-packing}
Assume there exists $S\subset\{1,\ldots,d-1\}$ such that
\begin{equation}
\varphi_S(\Omega)\ge a_0\varphi(\Omega),
\qquad
S\subset\{1,\ldots,\lfloor(1-b_0)d\rfloor\},
\end{equation}
\begin{equation}
\max_{s\in S}q_s^{-1}
\le
C_{\rm flat}\frac{\varphi_S(\Omega)}{|S|},
\qquad
|S|\ge c_Sd^\alpha.
\end{equation}
Then Assumption~\ref{ass:spectral-packing} holds with $M\ge c d^\alpha$, where $c>0$ depends only on the displayed constants and $\zeta$.
\end{corollary}

\begin{proof}
For $s\in S$, the boundary condition gives $a_s=(1-s/d)q_s^{-1}\ge b_0q_s^{-1}$, so $A_S\ge b_0\varphi_S(\Omega)$.
Also $a_s\le q_s^{-1}$ and the flatness assumption gives
\begin{equation}
B_S\le \left(\max_{s\in S}a_s\right)A_S
\le C\frac{A_S^2}{|S|}.
\end{equation}
Hence $N_{\rm eff}(S)\ge c|S|\ge cc_Sd^\alpha$.
Proposition~\ref{prop:effective-support-packing} applies.
\end{proof}

\begin{corollary}[Bounded-redundancy rulers]
\label{supp-cor:bounded-redundancy}
Let $\Omega$ cover all lags $1,\ldots,d-1$, let $m=|\Omega|$, and define the redundancy
\begin{equation}
R(\Omega)=\binom{m}{2}\Big/(d-1).
\label{eq:ruler-redundancy}
\end{equation}
Fix $\zeta\in(0,1/3)$.
If $R(\Omega)\le\bar R$ and $d\ge9$, then the conditions of Corollary~\ref{cor:flat-coverage-packing} hold with
\[
a_0=\frac{1}{16\bar R},
\qquad
b_0=\frac14,
\qquad
C_{\rm flat}=4\bar R,
\qquad
c_S=\frac18,
\qquad
\alpha=1 .
\]
Consequently, Assumption~\ref{ass:spectral-packing} holds with $M\ge c(\bar R,\zeta)\,d$ once $d\ge C(\bar R,\zeta)$, and the coverage-log lower bound \eqref{eq:coverage-logd-lower-s} applies with $\log M\asymp\log d$.
\end{corollary}

\begin{proof}
Every observed pair contributes to exactly one positive lag, so
\[
\sum_{s=1}^{d-1}q_s=\binom{m}{2}=R(\Omega)(d-1)\le\bar R(d-1).
\]
By Markov's inequality, the set of lags with $q_s>4\bar R$ has cardinality at most $(d-1)/4$.
Let
\[
S=\bigl\{1\le s\le\lfloor 3(d-1)/4\rfloor:\ q_s\le4\bar R\bigr\},
\]
so that $|S|\ge 3(d-1)/4-1-(d-1)/4=(d-1)/2-1\ge(d-1)/4\ge d/8$ for $d\ge9$.
On $S$ each term satisfies $q_s^{-1}\ge1/(4\bar R)$, whence
\[
\varphi_S(\Omega)\ge\frac{|S|}{4\bar R},
\qquad
\frac{\varphi_S(\Omega)}{|S|}\ge\frac1{4\bar R},
\qquad
\max_{s\in S}q_s^{-1}\le1\le4\bar R\,\frac{\varphi_S(\Omega)}{|S|}.
\]
Since $q_s\ge1$ for every lag, $\varphi(\Omega)\le d-1$, and therefore
\[
\varphi_S(\Omega)\ge\frac{|S|}{4\bar R}\ge\frac{d-1}{16\bar R}\ge\frac{\varphi(\Omega)}{16\bar R}.
\]
Finally, $S\subset\{1,\ldots,\lfloor3d/4\rfloor\}$, so the boundary condition of Corollary~\ref{cor:flat-coverage-packing} holds with $b_0=1/4$.
All conditions of Corollary~\ref{cor:flat-coverage-packing} are now verified with the displayed constants, and the conclusion follows.
\end{proof}

For minimum-redundancy rulers the redundancy \eqref{eq:ruler-redundancy} is bounded by an absolute constant \citep{Moffet1968,Leech1956}.
For a two-level nested ruler with balanced level sizes $N_1=N_2=k$ \citep{PalVaidyanathan2010}, one has $m=2k$ and $d-1=k(k+1)-1$, so $R(\Omega)=(2k^2-k)/(k^2+k-1)\le2$.
Both canonical families are therefore covered by Corollary~\ref{supp-cor:bounded-redundancy}.

\begin{corollary}[Coverage-log lower bound under effective support]
\label{cor:effective-support-lower}
Assume the conditions of Proposition~\ref{prop:effective-support-packing} and $N_{\rm eff}(S)\ge c_{\rm eff}d^\alpha$ for fixed $c_{\rm eff}>0$ and $\alpha>0$.
Then
\begin{equation}
\mathfrak R_n(\Omega,\mathcal P_{\mathbb R}(c_0))
\ge
c\gamma_0
\min\left\{
1,\,
\sqrt{\frac{\varphi(\Omega)\log d}{n}}
\right\},
\label{eq:effective-support-logd-lower}
\end{equation}
where $c>0$ depends only on $a_0,b_0,c_{\rm eff},\alpha,\zeta$ and $c_0$.
\end{corollary}

\begin{proof}
Proposition~\ref{prop:effective-support-packing} gives Assumption~\ref{ass:spectral-packing} with $M\ge c d^\alpha$.
Corollary~\ref{cor:coverage-log-lower-s} then gives \eqref{eq:effective-support-logd-lower} after adjusting constants.
\end{proof}

\begin{remark}[Boundary lags and the taper]
\label{rem:boundary-taper}
The taper condition is not cosmetic.
Lags near the aperture boundary carry small Rayleigh weight $1-s/d$.
A ruler whose inverse-coverage mass is concentrated almost entirely on such boundary lags may have a large $\varphi(\Omega)$ without having matching operator-norm spectral packing.
The effective-support certificate is therefore intended for balanced sparse rulers whose coverage difficulty is spread over many nonboundary lags.
\end{remark}

Thus, on the known-scale identity-neighborhood submodel, the factor $\varphi(\Omega)\log d$ in the leading term of the oracle upper bound is matched by the lower bound under real cosine spectral packing with $M\ge d^\alpha$, and the effective-support certificates above turn that packing requirement into a checkable coverage-spread condition.
This lower bound establishes the intrinsic coverage complexity of the aggregation map on this submodel.
It does not assert sharpness of the correlation-stability factor $\kappa_{\rm obs}$, the inverse-link constants, the Taylor remainder, or the plug-in calibration term.

\subsection{Expectation bound for the clipped oracle estimator}

The regime-restricted minimax corollary in the main text compares a lower bound stated in expectation with an oracle upper bound stated in high probability.
The following lemma records the integration step.

\begin{lemma}[From high probability to expectation, truncated form]
\label{supp-lem:expectation-integration}
Let $X\ge0$ be a random variable with $X\le D$ almost surely, let $u(t)=a\sqrt t+bt+ct^2$ with $a,b,c\ge0$, and suppose there are $C\ge e$ and $T_*\in[\,t_0,\infty]$, $t_0=\log(Cd)$, such that
\[
\Prb\bigl\{X>u(t)\bigr\}\le Cd\,e^{-t}
\qquad\text{for all }t\in[t_0,T_*].
\]
Then
\[
\E X\le 5\,u(t_0)+D\,Cd\,e^{-T_*},
\]
with the convention $e^{-\infty}=0$.
In particular, for the clipped oracle estimator the hypothesis holds with $T_*=\infty$ and $D=C_0\gamma_0d$, and taking $a=C_\eps\gamma_0L_1\kappa_{\rm obs}\sqrt{\varphi(\Omega)/n}$, $b=C\gamma_0\{L_1d+C_\eps L_2\min(\kappa_{\rm obs}^2\varphi(\Omega),d)\}/n$ and $c=C_\eps\gamma_0L_2d/n^2$ in the oracle bound \eqref{eq:oracle-rate-s} gives
\[
\E\norm{\widehat\Gamma_{\rm or}-\Gamma}_2
\le
C'\left(\gamma_0L_1\kappa_{\rm obs}\sqrt{\frac{\varphi(\Omega)\log d}{n}}
+\gamma_0\bigl\{L_1d+L_2\min(\kappa_{\rm obs}^2\varphi(\Omega),d)\bigr\}\,\frac{\log d}{n}
+\gamma_0L_2\,\frac{d\log^2d}{n^2}\right),
\]
with $C'$ depending only on the regularity domains.
The finite-$T_*$ form is used for the plug-in estimator, whose admissibility conditions depend on $t$.
\end{lemma}

\begin{proof}
The almost-sure bound makes $\E X$ finite.
Split $[0,D]$ at $u(t_0)\wedge D$ and $u(T_*)\wedge D$.
The first piece contributes at most its length $u(t_0)\wedge D\le u(t_0)$; on the last piece, which is empty when $u(T_*)\ge D$, monotonicity gives $\Prb\{X>u\}\le\Prb\{X>u(T_*)\}$ and the length is at most $D$.
Hence every term below is nonnegative and
\[
\E X=\int_0^{D}\Prb\{X>u\}\,du
\le u(t_0)
+\int_{u(t_0)\wedge D}^{u(T_*)\wedge D}\Prb\{X>u\}\,du
+D\,\Prb\{X>u(T_*)\}.
\]
The middle integral is empty unless $u(t_0)<D$, in which case its range is contained in $[u(t_0),u(T_*)]$; substituting $u=u(t)$ there and extending the limits to $[t_0,\infty)$,
\[
\int_{u(t_0)}^{u(T_*)}\Prb\{X>u\}\,du
\le Cd\int_{t_0}^\infty e^{-t}\Bigl(\frac{a}{2\sqrt t}+b+2ct\Bigr)dt .
\]
The factor $a/(2\sqrt t)+b$ is decreasing in $t$, and $\int_{t_0}^\infty te^{-t}\,dt=(t_0+1)e^{-t_0}\le2t_0e^{-t_0}$ for $t_0\ge1$, so, using $Cd\,e^{-t_0}=1$,
\[
Cd\int_{t_0}^\infty e^{-t}\Bigl(\frac{a}{2\sqrt t}+b+2ct\Bigr)dt
\le \frac{a}{2\sqrt{t_0}}+b+4ct_0
\le a+b+4ct_0\le4u(t_0) .
\]
The last term of the decomposition is at most $D\,\Prb\{X>u(T_*)\}\le D\,Cd\,e^{-T_*}$.
Hence $\E X\le5u(t_0)+DCde^{-T_*}$.
\end{proof}

\subsection{Minimax rate for the plug-in estimator}
\label{subsec:supp-plugin-minimax}

\begin{proof}[Proof of the unknown-scale plug-in minimax corollary of the main text]
Throughout, set $\eps=(1-c_0)/4$, so that $c_0\le1-2\eps$, and recall that on the class every induced threshold $\tau(\gamma)=\lambda/\sqrt\gamma$, $\gamma\in\Lambda$, lies in $[\tau_{\min},\tau_{\max}]$; Assumption~\ref{ass:inverse-s} therefore holds with constants depending only on $c_0$ and the threshold window, and $1\le\kobs\le1+c_0$, $m\ph(\Om)\ge d-1$; moreover $\ph(\Om)\ge1$, because the ruler is complete and only $d-s$ pairs exist at lag $s$, so $q_s\le d-s$ and $\ph(\Om)=\sum_{s=1}^{d-1}q_s^{-1}\ge\sum_{s=1}^{d-1}(d-s)^{-1}\ge1$.
Also throughout, $d\ge2$, and absolute thresholds in $n$ depending only on $c_1$ below are absorbed by enlarging $C_\star$.

\emph{Lower bound.}
The class $\mathcal Q$ contains the known-scale slice $\{\gamma_{\max}(\Id_d+T_d^0(\rho))\}$, so the coverage-log lower bound applies verbatim with $\gamma_0=\gamma_{\max}$, and on the non-saturated branch, which the stated regime guarantees, it reads $c\,\gamma_{\max}\sqrt{\ph(\Om)\log d/n}$.

\emph{Upper bound: high probability.}
By the operator-norm-ball corollary, for every $\Gamma\in\mathcal Q$ and every admissible $t=\log(Cd/\delta)$, with probability at least $1-\delta$,
\begin{equation}
\norm{\widehat\Gamma_{\rm plug}-\Gamma}_2
\le C\gamma\Bigl\{
\sqrt{\frac{\ph(\Om)t}{n}}
+d\,\frac tn
+\min\{\ph(\Om),d\}\,\frac tn
+d\,\frac{t^2}{n^2}
+(1+B)\,r_\mu(t)\Bigr\},
\label{supp-eq:plugin-class-bound}
\end{equation}
where $\gamma\le\gamma_{\max}$ and all constants, including $L_1,L_2,\kobs$ and $B$, have been absorbed using their uniform bounds on the class.
Here admissibility means
\begin{equation}
n\ge C\,\frac{mt}{\kobs},
\qquad
r_\mu(t)=C_\tau\Bigl\{\sqrt{\frac{\kobs t}{nm}}+\frac tn\Bigr\}\le c\eta_0 .
\label{supp-eq:admissibility}
\end{equation}

\emph{Domination at fixed $t$.}
By $m\ph(\Om)\ge d-1$ and $\kobs\le1+c_0$,
\[
\frac{r_\mu(t)}{\sqrt{\ph(\Om)t/n}}
\le C\Bigl(\sqrt{\frac{\kobs}{m\ph(\Om)}}+\sqrt{\frac{t}{n\ph(\Om)}}\Bigr)
\le \frac{C}{\sqrt d}+C\sqrt{\frac{t}{n\ph(\Om)}} ,
\]
so at $t\asymp\log d$ the calibration term is dominated by the leading term in the stated regime, since $\ph(\Om)\ge1$ makes the second ratio at most $C/\sqrt{C_\star}$.
The boundedness and curvature terms are dominated at $t_0\asymp\log d$ by the same two branches of the regime:
\[
\frac{d\,t_0/n}{\sqrt{\ph(\Om)t_0/n}}
=d\sqrt{\frac{t_0}{n\ph(\Om)}}
\le\frac{C}{\sqrt{C_\star}},
\qquad
\frac{\min\{\ph(\Om),d\}\,t_0/n}{\sqrt{\ph(\Om)t_0/n}}
\le\sqrt{\frac{\ph(\Om)t_0}{n}}
\le\frac{C}{\sqrt{C_\star}},
\]
the first using $n\ge C_\star d^2\log d/\ph(\Om)$ and the second using $n\ge C_\star\ph(\Om)\log d$, while $d\,t_0^2/n^2\le(d\,t_0/n)(t_0/n)\le d\,t_0/n$ because $t_0\le n$ in the regime.

\emph{Admissibility window.}
Set $T_*:=c_1n/m$ with $c_1=c_1(c_0,\tau_{\min},\tau_{\max},\eta_0)$ small, and let $t\le T_*$.
The first condition in \eqref{supp-eq:admissibility} holds because $\kobs\ge1$ and $c_1\le1/C$.
For the second, $r_\mu$ is nondecreasing in $t$, and, using $m\ge1$ and $\kobs\le1+c_0$,
\[
r_\mu(t)\le r_\mu(T_*)
=C_\tau\Bigl\{\sqrt{\frac{\kobs T_*}{nm}}+\frac{T_*}{n}\Bigr\}
=C_\tau\Bigl\{\frac{\sqrt{\kobs c_1}}{m}+\frac{c_1}{m}\Bigr\}
\le C_\tau\bigl\{\sqrt{(1+c_0)c_1}+c_1\bigr\},
\]
which is at most $c\eta_0$ once $c_1$ is small enough.
In the stated regime $n\ge C_\star d\log d$, since $\max\{\ph(\Om),d^2/\ph(\Om)\}\ge d$; with $m\le d$ this gives $T_*\ge c_1C_\star\log d\ge t_0=\log(Cd)$ once $C_\star$ is large enough.
Hence \eqref{supp-eq:plugin-class-bound} holds for all $t\in[t_0,T_*]$.

\emph{Integration to expectation.}
The clipped plug-in estimator obeys the deterministic bound $\norm{\widehat\Gamma_{\rm plug}-\Gamma}_2\le C_\tau\gamma_{\max}d=:D$, as in the large-deviation step of the plug-in proof.
Lemma~\ref{supp-lem:expectation-integration} with this $D$, $t_0\asymp\log d$, $T_*=c_1n/m$, and $a,b,c$ read off \eqref{supp-eq:plugin-class-bound} --- the calibration term $C\gamma(1+B)r_\mu(t)$ contributes $C'\gamma_{\max}/\sqrt{nm}$ to $a$ and $C'\gamma_{\max}/n$ to $b$, by $\kobs\le1+c_0$ --- gives
\[
\E\norm{\widehat\Gamma_{\rm plug}-\Gamma}_2
\le5u(t_0)+C_\tau\gamma_{\max}\,d\cdot Cd\,e^{-T_*} .
\]
The truncation tail is controlled in two cases.
If $n\le d^2$, then $T_*\ge c_1n/d\ge c_1C_\star\log d\ge4\log d$ for $C_\star$ large, so $d^2e^{-T_*}\le d^{-2}\le\sqrt{\log d/n}$.
If $n>d^2$, then $T_*\ge c_1n/d\ge c_1\sqrt n$, so $d^2e^{-T_*}\le ne^{-c_1\sqrt n}\le n^{-1/2}$ once $n\ge N(c_1)$, which the regime threshold guarantees after enlarging $C_\star$.
In both cases, using $\ph(\Om)\ge1$,
\[
C_\tau\gamma_{\max}\,d\cdot Cd\,e^{-T_*}
\le C\gamma_{\max}\sqrt{\frac{\ph(\Om)\log d}{n}} .
\]
Finally, $u(t_0)$ is dominated by the leading term in the regime by the fixed-$t$ domination step, so $\sup_{\mathcal Q}\E\norm{\widehat\Gamma_{\rm plug}-\Gamma}_2 \le C\gamma_{\max}\sqrt{\ph(\Om)\log d/n}$, matching the lower bound up to constants.
\end{proof}

\section{Proof of the vertex-projection obstruction}
\label{sec:supp-obstruction}

This section proves the two identity-covariance statements used in the main text to justify centering.

\begin{proposition}[Vertex-projection obstruction on weighted edge designs]
\label{supp-prop:raw-obstruction}
Let $G_i$, $i\in V$, be independent standard normal variables.
Let $u_i=\sign(G_i-\tau)=\mu+v_i$, where $\tau\ne0$, $\E v_i=0$ and $\Var(v_i)=\sigma_v^2$.
For real symmetric weights $a_{ij}=a_{ji}$ on an undirected edge design, define
\[
S_{\rm raw}=2\sum_{i<j}a_{ij}\{u_i u_j-\E(u_i u_j)\},
\qquad
S_{\rm cen}=2\sum_{i<j}a_{ij}v_i v_j,
\]
and $r_i=\sum_{j:j\ne i}a_{ij}$.
Then
\begin{equation}
S_{\rm raw}=S_{\rm cen}+2\mu\sum_i r_i v_i,
\label{supp-eq:weighted-raw-decomposition}
\end{equation}
and the two terms on the right are uncorrelated.
Consequently,
\begin{align}
\Var(S_{\rm raw})
&=4\sigma_v^4\sum_{i<j}a_{ij}^2
  +4\mu^2\sigma_v^2\sum_i r_i^2,
\label{supp-eq:weighted-raw-variance}\\
\Var(S_{\rm cen})
&=4\sigma_v^4\sum_{i<j}a_{ij}^2.
\label{supp-eq:weighted-centered-variance}
\end{align}
Equivalently, if $A$ is the associated symmetric hollow matrix, then
\begin{equation}
\Var(S_{\rm cen})=2\sigma_v^4\norm{A}_F^2.
\label{supp-eq:weighted-centered-frobenius}
\end{equation}
\end{proposition}

\begin{proof}
The identity $u_i=\mu+v_i$ gives
\begin{equation}
u_i u_j-\E[u_i u_j]
  =v_i v_j-\E[v_i v_j]+\mu(v_i+v_j).
\end{equation}
Summing with weights gives \eqref{supp-eq:weighted-raw-decomposition}.
The centered quadratic term and the linear term are uncorrelated.
Indeed, every mixed moment has the form $\E[v_i v_jv_k]$ with $i\ne j$; if $k$ is distinct from both endpoints, independence gives a centered first moment, and if $k=i$ or $k=j$, the other endpoint still has centered first moment.

For the centered statistic, two distinct edges have zero covariance: disjoint edges are independent, and edges sharing exactly one endpoint again leave a centered first moment.
Only identical edges contribute, so
\[
\Var(S_{\rm cen})=4\sum_{i<j}a_{ij}^2\Var(v_i v_j)
=4\sigma_v^4\sum_{i<j}a_{ij}^2.
\]
The linear term has variance $4\mu^2\sigma_v^2\sum_i r_i^2$.
Combining these two uncorrelated contributions proves \eqref{supp-eq:weighted-raw-variance}.
Since $\norm{A}_F^2=2\sum_{i<j}a_{ij}^2$, \eqref{supp-eq:weighted-centered-frobenius} follows.
\end{proof}

\begin{corollary}[Sparse-ruler coverage and row-sum separation]
\label{supp-cor:sparse-ruler-obstruction}
Let $\Omega$ be any ruler covering all lags $1,\ldots,d-1$, set $m=|\Omega|$, and define
\[
    \zeta_s^{\rm raw}
      =q_s^{-1}\sum_{(j,k)\in\Omega_s}\{u_j u_k-\E(u_j u_k)\},
    \qquad
    \zeta_s^{\rm cen}=q_s^{-1}\sum_{(j,k)\in\Omega_s}v_jv_k .
\]
Then
\begin{equation}
    \Var\left(2\sum_{s=1}^{d-1}\zeta_s^{\rm raw}\right)
       \ge 16\mu^2\sigma_v^2\frac{(d-1)^2}{m},
\label{supp-eq:raw-lower-bound}
\end{equation}
whereas
\begin{equation}
    \Var\left(2\sum_{s=1}^{d-1}\zeta_s^{\rm cen}\right)
       =4\sigma_v^4\varphi(\Omega).
\label{supp-eq:centered-exact}
\end{equation}
\end{corollary}

\begin{proof}
Assign to each undirected covered edge $\{i,j\}$ the weight $a_{ij}=q_{|i-j|}^{-1}$; the associated hollow symmetric matrix is the lifting $H_\Omega(\beta)$ with $\beta_s=q_s^{-1}$.
By Lemma~\ref{supp-lem:lifting},
\[
\sum_{i<j}a_{ij}^2=\tfrac12\norm{H_\Omega(\beta)}_F^2=\sum_{s=1}^{d-1}q_s q_s^{-2}=\varphi(\Omega),
\]
and the row sums satisfy
\[
  \sum_{i\in\Omega}r_i=2\sum_{s=1}^{d-1}q_s\beta_s=2(d-1),
  \qquad
  \sum_{i\in\Omega}r_i^2\ge\frac1m\Big(\sum_{i\in\Omega}r_i\Big)^2=\frac{4(d-1)^2}{m}.
\]
The conclusion follows from Proposition~\ref{supp-prop:raw-obstruction}.
\end{proof}

\section{Numerical details}
\label{sec:experiments}

The numerical experiments are rate checks, not algorithmic benchmarks.
They separately track the row-sum obstruction before centering, the oracle $n^{-1/2}$ behavior, the $\sqrt{\varphi(\Omega)}$ coverage dependence and the marginal plug-in calibration gap.
The constants are not calibrated to the theorem statements: every figure below is a rate diagnostic measured against the theorem's upper-bound scales, not a fit to an equality model.
Code reproducing all numerical experiments and figures, including the fixed random seeds, will be made publicly available upon publication.

Experiments are implemented in standardized coordinates: snapshots are drawn with unit marginal variance and correlation matrix $C$, the threshold is applied at the normalized value $\tau$, and the scale factor $\gamma_0=2$ is multiplied back when reporting covariance-scale errors.
Because the threshold signs satisfy $\sign(X_j-\lambda)=\sign(X_j/\sqrt{\gamma_0}-\tau)$ with $\lambda=\tau\sqrt{\gamma_0}$, this is an exact reparametrization of the experiment at scale $\gamma_0$.
Error bars and fitted slopes are computed from independent Monte Carlo repetitions.
The main figures report medians; the saved experiment tables also contain replicate-level errors, normalized errors, clipping frequencies and coverage summaries.
The paper-scale run used 2,500 Monte Carlo draws for the pure variance experiment, 200 repetitions for the oracle rate experiment, 80 repetitions per ruler for the coverage experiment, 160 repetitions for plug-in calibration and 100 repetitions per threshold cell.

\begin{table}[t]
\centering
\caption{Summary of the numerical rate checks.}
\label{tab:rate-checks}
\begin{tabular}{@{}lll@{}}
\toprule
quantity & theoretical prediction & empirical check\\
\midrule
Centered variance & $n\,{\rm Var}/\varphi(\Omega)=O(1)$ & median normalized value $0.983$\\
Raw products & no $\varphi(\Omega)$ collapse & raw/coverage median $19.44$\\
Raw row-sum scale & row-sum normalization & raw/row median $1.064$\\
Oracle rate in $n$ & $n^{-1/2}$ & slopes $-0.496$ to $-0.512$\\
Coverage geometry & error $\propto\sqrt{\varphi(\Omega)}$ & coverage slope $0.538$\\
Plug-in gap & $(1+S_1)r_\mu$ & slope $0.936$\\
Global regression & $\beta_n=-1/2$, $\beta_\varphi=1/2$ & $-0.502$, $0.462$\\
\bottomrule
\end{tabular}
\end{table}

\subsection{Raw-product row sums and coverage variance}

The obstruction diagnostic is shown in the main text.
Across apertures $d\in\{64,128,256,512\}$, sample sizes $n\in\{50,\ldots,3200\}$ and four ruler families, the median normalized centered variance is $0.983$, close to the predicted value one.
Raw products do not collapse under $\varphi(\Omega)$ normalization: the median raw-to-coverage ratio is $19.44$, while the row-sum-normalized median ratio is $1.064$.

\subsection{Oracle operator-norm rate}

\begin{figure}[t]
\centering
\includegraphics[width=.95\linewidth]{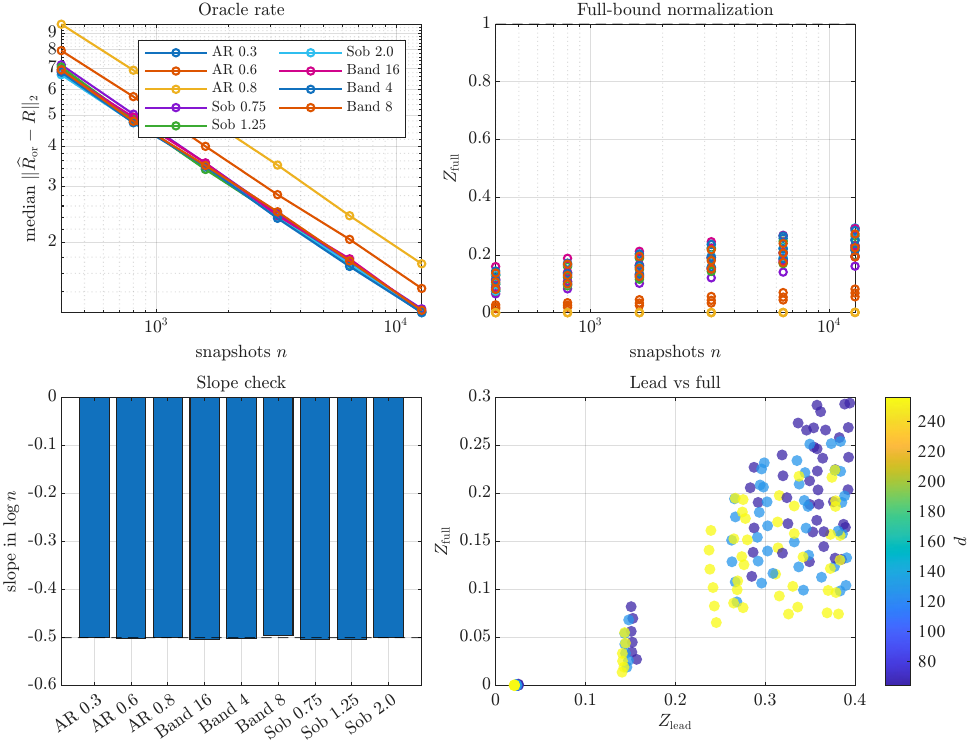}
\caption{Oracle operator-norm rate.  The panels show log-log error decay, full-bound normalized errors, per-class slopes in $\log n$, and the relation between leading-term and full-bound normalizations.}
\label{fig:cov-recovery}
\end{figure}

Figure~\ref{fig:cov-recovery} provides a diagnostic check of the oracle term in the main upper-bound theorem.
We used AR(1) correlations with $a\in\{0.3,0.6,0.8\}$, Sobolev spectral densities with $\beta\in\{0.75,1.25,2.0\}$, and banded short-memory correlations with bandwidths $K\in\{4,8,16\}$.
For these nine classes, the fitted slopes of $\log\norm{\widehat\Gamma_{\rm or}-\Gamma}_2$ against $\log n$ range from $-0.496$ to $-0.512$.
The median slopes are therefore indistinguishable, at the scale of the simulation, from the theoretical $n^{-1/2}$ exponent.
The normalized error
\[
Z_{\rm full}=
\frac{\norm{\widehat\Gamma_{\rm or}-\Gamma}_2}
{\gamma_0L_1\kappa_{\rm obs}\sqrt{\varphi(\Omega)\log d/n}
  +\gamma_0(L_1+L_2)d\log d/n}
\]
does not show systematic growth over the tested $n$ range, which supports the combined leading-plus-Taylor normalization.
The normalization uses the conservative remainder $(L_1+L_2)d\log d/n$, which dominates the curvature term $\min\{\kappa_{\rm obs}^2\varphi(\Omega),d\}\log d/n$ of the theorem, so $Z_{\rm full}$ is a valid upper-bound diagnostic.

\subsection{Virtual coverage geometry}

\begin{figure}[t]
\centering
\includegraphics[width=.95\linewidth]{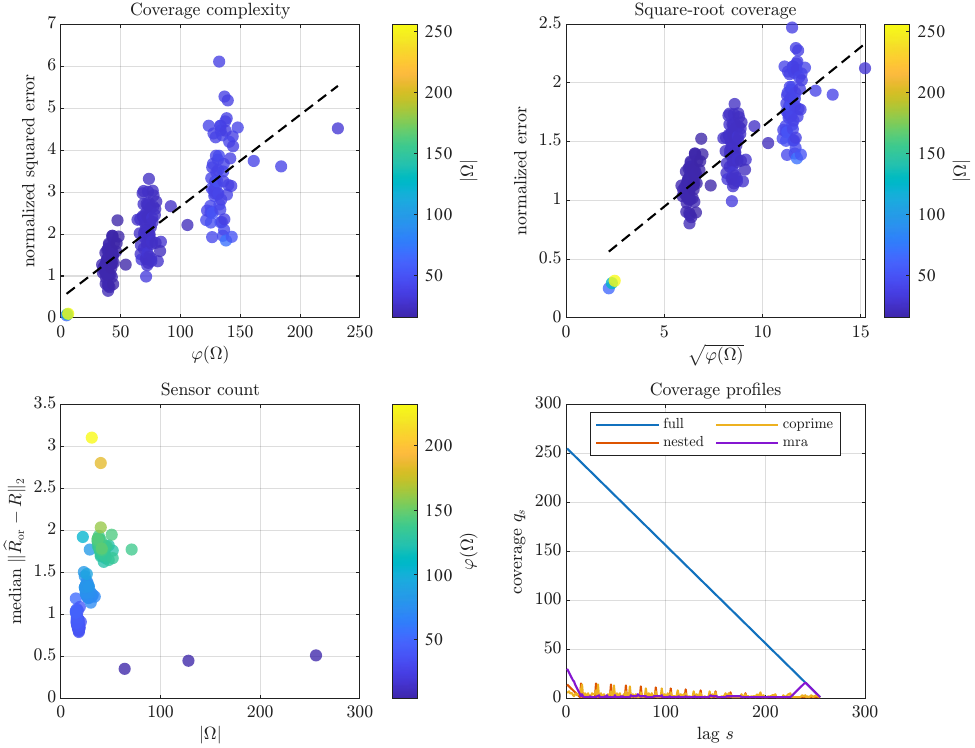}
\caption{Coverage geometry aligns with operator error more closely than sensor count.  The first two panels plot normalized oracle errors against $\varphi(\Omega)$ and $\sqrt{\varphi(\Omega)}$, the third plots the raw median operator error against $|\Omega|$, and the fourth shows representative coverage profiles $q_s$.}
\label{fig:geometry}
\end{figure}

Figure~\ref{fig:geometry} varies the design at fixed covariance by using full, nested, co-prime, minimum-redundancy-style and randomly generated complete rulers.
All rulers cover every lag by construction: the nested, co-prime and minimum-redundancy-style seed geometries are completed greedily, adding observation points until every lag $1,\ldots,d-1$ is covered and then pruning redundant points.
In particular, the standard co-prime geometry, whose difference set alone does not cover all lags, enters only as a seed.
A regression of $\log\norm{\widehat\Gamma_{\rm or}-\Gamma}_2$ on $\log\varphi(\Omega)$ gives slope $0.538$, close to the theoretical square-root dependence.
The competing scatter against $|\Omega|$ is visibly less aligned because rulers with comparable numbers of observed coordinates can have very different coverage profiles.
This experiment supports the interpretation of $\varphi(\Omega)$ as the design complexity that enters the operator-norm risk.

\subsection{Marginal bit calibration}

\begin{figure}[!t]
\centering
\includegraphics[width=.95\linewidth]{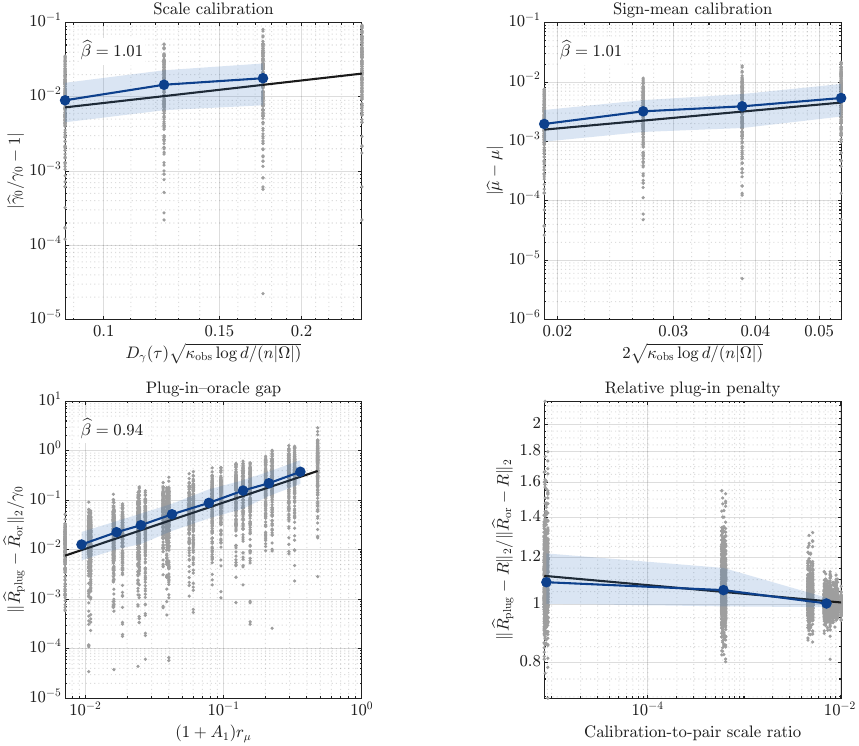}
\caption{Plug-in calibration rate diagnostic.  The top two panels compare the marginal scale and sign-mean calibration errors with their predicted pooled marginal rates.  Points are Monte Carlo replicates, solid curves are binned medians, and bands show interquartile ranges.  The bottom-left panel compares the plug-in--oracle operator gap with the short-memory calibration scale $(1+S_1)r_\mu$.  The bottom-right panel shows that the relative plug-in penalty remains close to one when the calibration scale is small relative to the centered pair-estimation scale.}
\label{fig:plugin-calibration}
\end{figure}

Figure~\ref{fig:plugin-calibration} isolates marginal calibration from centered pair estimation.
The scale and sign-mean errors are plotted against their pooled marginal rates on logarithmic axes, with binned medians and interquartile ranges.
For AR(1) correlations with $a\in\{0.1,0.3,0.5,0.7,0.85,0.9\}$, the plug-in--oracle operator gap has fitted log-log slope $0.936$ against $(1+S_1)r_\mu$.
The fitted slope is consistent with plug-in calibration behaving as a lower-dimensional marginal estimation component rather than another sparse-pair averaging error.

\subsection{Threshold conditioning}

\begin{figure}[!t]
\centering
\includegraphics[width=.95\linewidth]{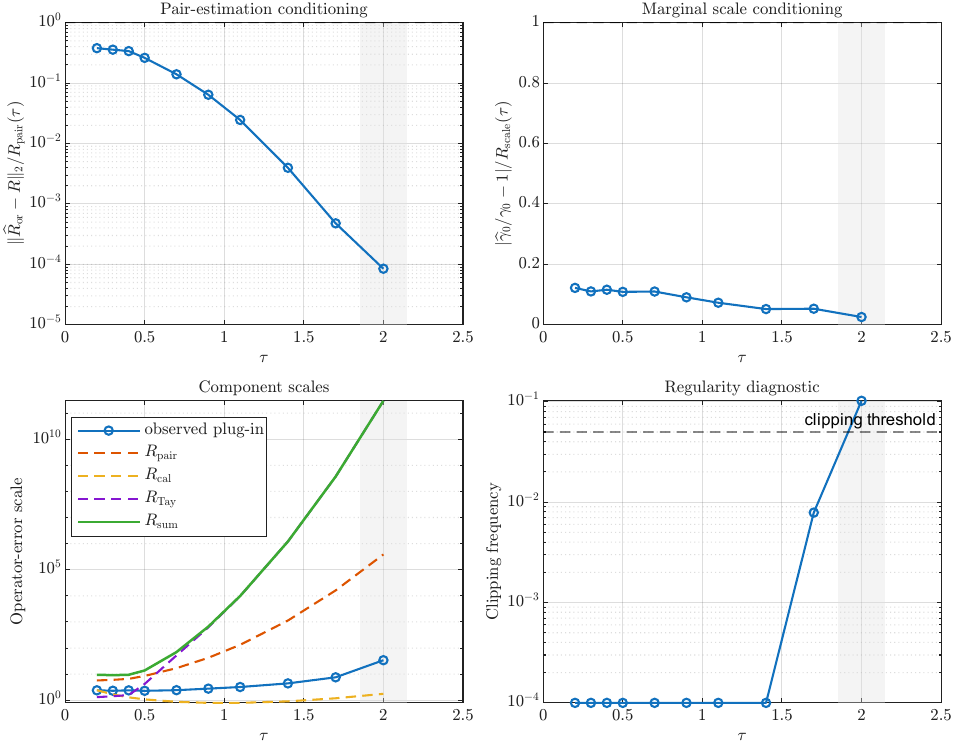}
\caption{Threshold-conditioning decomposition.  The top-left panel normalizes the oracle operator error by the inverse-link pair-estimation scale.  The top-right panel normalizes the marginal scale error by the pooled scale-calibration sensitivity.  The bottom-left panel shows the observed plug-in error and the theoretical component scales on a logarithmic axis.  The bottom-right panel reports the clipping frequency used to identify the regular threshold range; shaded regions are not used for rate interpretation.}
\label{fig:threshold-conditioning}
\end{figure}

Figure~\ref{fig:threshold-conditioning} decomposes threshold sensitivity into pair-estimation, marginal-calibration and curvature components.
We interpret the decomposition only on the regular threshold set where the expected number of exceedances is not too small and inverse-link clipping is rare.
Over the regular threshold range, the observed errors remain below the corresponding component-normalized scales; outside this range, clipping and rare-event effects dominate.
The experiment is not intended to select a universal optimal threshold.

\subsection{Global rate collapse}

Pooling the oracle cells from the rate and geometry experiments, we fit
\[
\log\norm{\widehat\Gamma_{\rm or}-\Gamma}_2
=\alpha+\beta_n\log n+\beta_\varphi\log\varphi(\Omega)
  +\beta_{\log d}\log\log d+\beta_\kappa\log\kappa_{\rm obs}
  +\varepsilon .
\]
The two primary exponents are close to the theorem, with $\widehat\beta_n=-0.502$ and $\widehat\beta_\varphi=0.462$.
The $\log\log d$ coefficient is noisier, as expected from the limited number of tested dimensions, and the $\kappa_{\rm obs}$ coefficient is only diagnostic because the experiment does not independently vary conditioning.
Across the pooled diagnostics, centered variance collapses under $\varphi(\Omega)/n$, oracle error decays at $n^{-1/2}$, coverage geometry enters through $\sqrt{\varphi(\Omega)}$, and the plug-in--oracle gap scales with $(1+S_1)r_\mu$.

\FloatBarrier

\bibliographystyle{plainnat}
\bibliography{references}

@book{Janson1997,
  author = {Janson, Svante},
  title = {Gaussian Hilbert Spaces},
  publisher = {Cambridge University Press},
  year = {1997},
  doi = {10.1017/CBO9780511526169}
}

@book{Nualart2006,
  author = {Nualart, David},
  title = {The Malliavin Calculus and Related Topics},
  edition = {2},
  publisher = {Springer},
  year = {2006},
  doi = {10.1007/3-540-28329-3}
}

@book{PeccatiTaqqu2011,
  author = {Peccati, Giovanni and Taqqu, Murad S.},
  title = {Wiener Chaos: Moments, Cumulants and Diagrams: A Survey With Computer Implementation},
  publisher = {Springer Milan},
  year = {2011},
  doi = {10.1007/978-88-470-1679-8}
}

@book{Bhatia2007,
  author = {Bhatia, Rajendra},
  title = {Positive Definite Matrices},
  publisher = {Princeton University Press},
  year = {2007},
  doi = {10.1515/9781400827787}
}

@article{Hoeffding1948,
  author = {Hoeffding, Wassily},
  title = {A Class of Statistics With Asymptotically Normal Distribution},
  journal = {The Annals of Mathematical Statistics},
  volume = {19},
  number = {3},
  pages = {293--325},
  year = {1948},
  doi = {10.1214/aoms/1177730196}
}

@book{DeLaPenaGine1999,
  author = {de la Pe{\~n}a, Victor H. and Gin{\'e}, Evarist},
  title = {Decoupling: From Dependence to Independence},
  publisher = {Springer},
  year = {1999},
  doi = {10.1007/978-1-4612-0537-1}
}

@article{BreuerMajor1983,
  author = {Breuer, P{\'e}ter and Major, P{\'e}ter},
  title = {Central Limit Theorems for Non-Linear Functionals of {Gaussian} Fields},
  journal = {Journal of Multivariate Analysis},
  volume = {13},
  number = {3},
  pages = {425--441},
  year = {1983},
  doi = {10.1016/0047-259X(83)90019-2}
}

@article{Arcones1994,
  author = {Arcones, Miguel A.},
  title = {Limit Theorems for Nonlinear Functionals of a Stationary {Gaussian} Sequence of Vectors},
  journal = {The Annals of Probability},
  volume = {22},
  number = {4},
  pages = {2242--2274},
  year = {1994},
  doi = {10.1214/aop/1176988503}
}

@article{CasazzaKutyniokLi2008,
  author = {Casazza, Peter G. and Kutyniok, Gitta and Li, Shidong},
  title = {Fusion Frames and Distributed Processing},
  journal = {Applied and Computational Harmonic Analysis},
  volume = {25},
  number = {1},
  pages = {114--132},
  year = {2008},
  doi = {10.1016/j.acha.2007.10.001}
}

@techreport{Bussgang1952,
  author = {Bussgang, Julian J.},
  title = {Crosscorrelation Functions of Amplitude-Distorted {Gaussian} Signals},
  institution = {Research Laboratory of Electronics, Massachusetts Institute of Technology},
  number = {216},
  pages = {1--14},
  year = {1952},
  url = {http://hdl.handle.net/1721.1/4847}
}

@article{Plackett1954,
  author = {Plackett, R. L.},
  title = {A Reduction Formula for Normal Multivariate Integrals},
  journal = {Biometrika},
  volume = {41},
  number = {3--4},
  pages = {351--360},
  year = {1954},
  doi = {10.1093/biomet/41.3-4.351}
}

@article{Price1958,
  author = {Price, Robert},
  title = {A Useful Theorem for Nonlinear Devices Having {Gaussian} Inputs},
  journal = {IRE Transactions on Information Theory},
  volume = {4},
  number = {2},
  pages = {69--72},
  year = {1958},
  doi = {10.1109/TIT.1958.1057444}
}

@article{VanVleckMiddleton1966,
  author = {Van Vleck, J. H. and Middleton, D.},
  title = {The Spectrum of Clipped Noise},
  journal = {Proceedings of the IEEE},
  volume = {54},
  number = {1},
  pages = {2--19},
  year = {1966},
  doi = {10.1109/PROC.1966.4567}
}

@article{DirksenMalyRauhut2022,
  author = {Dirksen, Sjoerd and Maly, Johannes and Rauhut, Holger},
  title = {Covariance Estimation Under {One-Bit} Quantization},
  journal = {The Annals of Statistics},
  volume = {50},
  number = {6},
  pages = {3538--3562},
  year = {2022},
  doi = {10.1214/22-AOS2239}
}

@article{DirksenMaly2024,
  author = {Dirksen, Sjoerd and Maly, Johannes},
  title = {Tuning-Free {One-Bit} Covariance Estimation Using Data-Driven Dithering},
  journal = {IEEE Transactions on Information Theory},
  volume = {70},
  number = {7},
  pages = {5228--5247},
  year = {2024},
  doi = {10.1109/TIT.2024.3358994}
}

@article{BickelLevina2008Regularized,
  author = {Bickel, Peter J. and Levina, Elizaveta},
  title = {Regularized Estimation of Large Covariance Matrices},
  journal = {The Annals of Statistics},
  volume = {36},
  number = {1},
  pages = {199--227},
  year = {2008},
  doi = {10.1214/009053607000000758}
}

@article{BickelLevina2008Thresholding,
  author = {Bickel, Peter J. and Levina, Elizaveta},
  title = {Covariance Regularization by Thresholding},
  journal = {The Annals of Statistics},
  volume = {36},
  number = {6},
  pages = {2577--2604},
  year = {2008},
  doi = {10.1214/08-AOS600}
}

@article{CaiZhangZhou2010,
  author = {Cai, T. Tony and Zhang, Cun-Hui and Zhou, Harrison H.},
  title = {Optimal Rates of Convergence for Covariance Matrix Estimation},
  journal = {The Annals of Statistics},
  volume = {38},
  number = {4},
  pages = {2118--2144},
  year = {2010},
  doi = {10.1214/09-AOS752}
}

@article{CaiYuan2012,
  author = {Cai, T. Tony and Yuan, Ming},
  title = {Adaptive Covariance Matrix Estimation Through Block Thresholding},
  journal = {The Annals of Statistics},
  volume = {40},
  number = {4},
  pages = {2014--2042},
  year = {2012},
  doi = {10.1214/12-AOS999}
}

@article{CaiZhou2012,
  author = {Cai, T. Tony and Zhou, Harrison H.},
  title = {Optimal Rates of Convergence for Sparse Covariance Matrix Estimation},
  journal = {The Annals of Statistics},
  volume = {40},
  number = {5},
  pages = {2389--2420},
  year = {2012},
  doi = {10.1214/12-AOS998}
}

@article{ChenWangNgWang2023,
  author = {Chen, Junren and Wang, Cheng-Long and Ng, Michael K. and Wang, Di},
  title = {High Dimensional Statistical Estimation Under Uniformly Dithered {One-Bit} Quantization},
  journal = {IEEE Transactions on Information Theory},
  volume = {69},
  number = {8},
  pages = {5151--5187},
  year = {2023},
  doi = {10.1109/TIT.2023.3266271}
}

@article{Xiao2023,
  author = {Xiao, Yu-Hang and Huang, Lei and Ram{\'i}rez, David and Qian, Cheng and So, Hing Cheung},
  title = {Covariance Matrix Recovery From {One-Bit} Data With {Non-Zero} Quantization Thresholds: Algorithm and Performance Analysis},
  journal = {IEEE Transactions on Signal Processing},
  volume = {71},
  pages = {4060--4076},
  year = {2023},
  doi = {10.1109/TSP.2023.3325664}
}

@article{ChapeauBlondeau2008,
  author = {Chapeau-Blondeau, Francois and Blanchard, Solenna and Rousseau, David},
  title = {Fisher Information and Noise-Aided Power Estimation From One-Bit Quantizers},
  journal = {Digital Signal Processing},
  volume = {18},
  number = {3},
  pages = {434--443},
  year = {2008},
  doi = {10.1016/j.dsp.2007.04.012}
}

@inproceedings{LiuLin2021,
  author = {Liu, Chun-Lin and Lin, Zi-Min},
  title = {One-Bit Autocorrelation Estimation With Non-Zero Thresholds},
  booktitle = {Proceedings of the IEEE International Conference on Acoustics, Speech and Signal Processing (ICASSP)},
  pages = {4520--4524},
  year = {2021},
  doi = {10.1109/ICASSP39728.2021.9414732}
}

@inproceedings{LiuChou2025,
  author = {Liu, Chun-Lin and Chou, Yi-Hung},
  title = {Approximation and Analysis of the One-Bit Hermite Law},
  booktitle = {Proceedings of the IEEE International Conference on Acoustics, Speech and Signal Processing (ICASSP)},
  pages = {1--5},
  year = {2025},
  doi = {10.1109/ICASSP49660.2025.10888049}
}

@article{Eamaz2023,
  author = {Eamaz, Arian and Yeganegi, Farhang and Soltanalian, Mojtaba},
  title = {Covariance Recovery for One-Bit Sampled Stationary Signals With Time-Varying Sampling Thresholds},
  journal = {Signal Processing},
  volume = {206},
  pages = {108899},
  year = {2023},
  doi = {10.1016/j.sigpro.2022.108899}
}

@article{Moffet1968,
  author = {Moffet, Alan T.},
  title = {Minimum-Redundancy Linear Arrays},
  journal = {IEEE Transactions on Antennas and Propagation},
  volume = {16},
  number = {2},
  pages = {172--175},
  year = {1968},
  doi = {10.1109/TAP.1968.1139138}
}

@article{LinebargerSudboroughTollis1993,
  author = {Linebarger, D. A. and Sudborough, I. H. and Tollis, I. G.},
  title = {Difference Bases and Sparse Sensor Arrays},
  journal = {IEEE Transactions on Information Theory},
  volume = {39},
  number = {2},
  pages = {716--721},
  year = {1993},
  doi = {10.1109/18.212309}
}

@article{PalVaidyanathan2010,
  author = {Pal, Piya and Vaidyanathan, P. P.},
  title = {Nested Arrays: A Novel Approach to Array Processing With Enhanced Degrees of Freedom},
  journal = {IEEE Transactions on Signal Processing},
  volume = {58},
  number = {8},
  pages = {4167--4181},
  year = {2010},
  doi = {10.1109/TSP.2010.2049264}
}

@article{VaidyanathanPal2011,
  author = {Vaidyanathan, P. P. and Pal, Piya},
  title = {Sparse Sensing With {Co-Prime} Samplers and Arrays},
  journal = {IEEE Transactions on Signal Processing},
  volume = {59},
  number = {2},
  pages = {573--586},
  year = {2011},
  doi = {10.1109/TSP.2010.2089682}
}

@article{CaiRenZhou2013,
  author = {Cai, T. Tony and Ren, Zhao and Zhou, Harrison H.},
  title = {Optimal Rates of Convergence for Estimating Toeplitz Covariance Matrices},
  journal = {Probability Theory and Related Fields},
  volume = {156},
  number = {1--2},
  pages = {101--143},
  year = {2013},
  doi = {10.1007/s00440-012-0422-7}
}

@article{KlockmannKrivobokova2024,
  author = {Klockmann, Karolina and Krivobokova, Tatyana},
  title = {Efficient Nonparametric Estimation of Toeplitz Covariance Matrices},
  journal = {Biometrika},
  volume = {111},
  number = {3},
  pages = {843--864},
  year = {2024},
  doi = {10.1093/biomet/asae002}
}

@article{xu2026bit,
  author = {Xu, Hongwei and Yang, Zai},
  title = {Bit-Efficient {Toeplitz} Covariance Estimation},
  journal = {IEEE Transactions on Information Theory},
  volume = {72},
  number = {7},
  pages = {5293--5316},
  year = {2026},
  doi = {10.1109/TIT.2026.3697612},
  publisher = {IEEE}
}

@misc{XuZhengYang2025,
  author = {Xu, Hongwei and Zheng, Weichao and Yang, Zai},
  title = {Compressive {Toeplitz} Covariance Estimation From Few-Bit Quantized Measurements With Applications to {DOA} Estimation},
  howpublished = {arXiv preprint arXiv:2512.22527},
  year = {2025},
  doi = {10.48550/arXiv.2512.22527}
}

@inproceedings{LiuVaidyanathan2017,
  author = {Liu, Chun-Lin and Vaidyanathan, P. P.},
  title = {One-Bit Sparse Array {DOA} Estimation},
  booktitle = {Proceedings of the IEEE International Conference on Acoustics, Speech and Signal Processing (ICASSP)},
  pages = {3126--3130},
  year = {2017},
  doi = {10.1109/ICASSP.2017.7952732}
}

@article{AriananadaLeus2012,
  author = {Ariananda, Dyonisius Dony and Leus, Geert},
  title = {Compressive Wideband Power Spectrum Estimation},
  journal = {IEEE Transactions on Signal Processing},
  volume = {60},
  number = {9},
  pages = {4775--4789},
  year = {2012},
  doi = {10.1109/TSP.2012.2201153}
}

@article{Leech1956,
  author = {Leech, John},
  title = {On the Representation of {$1,2,\ldots,n$} by Differences},
  journal = {Journal of the London Mathematical Society},
  volume = {s1-31},
  number = {2},
  pages = {160--169},
  year = {1956},
  doi = {10.1112/jlms/s1-31.2.160}
}

@article{BarShalomWeiss2002,
  author = {Bar-Shalom, O. and Weiss, A. J.},
  title = {{DOA} Estimation Using One-Bit Quantized Measurements},
  journal = {IEEE Transactions on Aerospace and Electronic Systems},
  volume = {38},
  number = {3},
  pages = {868--884},
  year = {2002},
  doi = {10.1109/TAES.2002.1039405}
}

@article{Sedighi2021,
  author = {Sedighi, Saeid and {Bhavani Shankar}, M. R. and Soltanalian, Mojtaba and Ottersten, Bj{\"o}rn},
  title = {On the Performance of One-Bit {DoA} Estimation via Sparse Linear Arrays},
  journal = {IEEE Transactions on Signal Processing},
  volume = {69},
  pages = {6165--6182},
  year = {2021},
  doi = {10.1109/TSP.2021.3122290}
}

@article{ChengChen2020,
  author = {Cheng, Zhiyong and Chen, Shengyao and Shen, Qibin and He, Jin and Liu, Zhong},
  title = {Direction Finding of Electromagnetic Sources on a Sparse Cross-Dipole Array Using One-Bit Measurements},
  journal = {IEEE Access},
  volume = {8},
  pages = {83131--83143},
  year = {2020},
  doi = {10.1109/ACCESS.2020.2989525}
}

@article{LinLiu2026,
  author = {Lin, Yi-Heng and Liu, Chun-Lin},
  title = {Counter-Based Scatter Matrix Estimation for One-Bit Centered Bivariate {Cauchy} Signals},
  journal = {IEEE Transactions on Signal Processing},
  year = {2026},
  volume = {74},
  pages = {2417--2433},
  doi = {10.1109/TSP.2026.3694498}
}

@article{SteinBarNossekTabrikian2018,
  author = {Stein, Manuel S. and Bar, Shahar and Nossek, Josef A. and Tabrikian, Joseph},
  title = {Performance Analysis for Channel Estimation With 1-Bit {ADC} and Unknown Quantization Threshold},
  journal = {IEEE Transactions on Signal Processing},
  volume = {66},
  number = {10},
  pages = {2557--2571},
  year = {2018},
  doi = {10.1109/TSP.2018.2815022}
}

@article{LevinaVershynin2012,
  author = {Levina, Elizaveta and Vershynin, Roman},
  title = {Partial Estimation of Covariance Matrices},
  journal = {Probability Theory and Related Fields},
  volume = {153},
  number = {3--4},
  pages = {405--419},
  year = {2012},
  doi = {10.1007/s00440-011-0349-4}
}

@article{Lounici2014,
  author = {Lounici, Karim},
  title = {High-dimensional covariance matrix estimation with missing observations},
  journal = {Bernoulli},
  volume = {20},
  number = {3},
  pages = {1029--1058},
  year = {2014},
  doi = {10.3150/12-BEJ487}
}

@article{CaiZhang2016,
  author = {Cai, T. Tony and Zhang, Anru},
  title = {Minimax rate-optimal estimation of high-dimensional covariance matrices with incomplete data},
  journal = {Journal of Multivariate Analysis},
  volume = {150},
  pages = {55--74},
  year = {2016},
  doi = {10.1016/j.jmva.2016.05.002}
}

@article{Abdalla2024,
  author = {Abdalla, Pedro},
  title = {Covariance estimation under missing observations and {$L_4$--$L_2$} moment equivalence},
  journal = {Electronic Journal of Statistics},
  volume = {18},
  number = {1},
  pages = {2665--2686},
  year = {2024},
  doi = {10.1214/24-EJS2264}
}

@misc{KabanavaRauhut2017,
  author = {Kabanava, Maryia and Rauhut, Holger},
  title = {Masked {Toeplitz} Covariance Estimation},
  howpublished = {arXiv preprint arXiv:1709.09377},
  year = {2017},
  doi = {10.48550/arXiv.1709.09377}
}

@article{ChenGittensTropp2012,
  author = {Chen, Richard Y. and Gittens, Alex and Tropp, Joel A.},
  title = {The Masked Sample Covariance Estimator: An Analysis Using Matrix Concentration Inequalities},
  journal = {Information and Inference},
  volume = {1},
  number = {1},
  pages = {2--20},
  year = {2012},
  doi = {10.1093/imaiai/ias001}
}

@article{TerrinTaqqu1990,
  author = {Terrin, Norma and Taqqu, Murad S.},
  title = {A Noncentral Limit Theorem for Quadratic Forms of {Gaussian} Stationary Sequences},
  journal = {Journal of Theoretical Probability},
  volume = {3},
  number = {3},
  pages = {449--475},
  year = {1990},
  doi = {10.1007/BF01061262}
}

@article{GiraitisTaqqu1998,
  author = {Giraitis, Liudas and Taqqu, Murad S.},
  title = {Central Limit Theorems for Quadratic Forms with Time-Domain Conditions},
  journal = {The Annals of Probability},
  volume = {26},
  number = {1},
  pages = {377--398},
  year = {1998},
  doi = {10.1214/aop/1022855425}
}

@incollection{MalyYangDirksenRauhutCaire2022,
  author = {Maly, Johannes and Yang, Tianyu and Dirksen, Sjoerd and Rauhut, Holger and Caire, Giuseppe},
  title = {New Challenges in Covariance Estimation: Multiple Structures and Coarse Quantization},
  booktitle = {Compressed Sensing in Information Processing},
  publisher = {Springer International Publishing},
  address = {Cham},
  pages = {77--104},
  year = {2022},
  doi = {10.1007/978-3-031-09745-4_3}
}

@article{RomeroArianandaTianLeus2016,
  author = {Romero, Daniel and Ariananda, Dyonisius D. and Tian, Zhi and Leus, Geert},
  title = {Compressive Covariance Sensing: Structure-Based Compressive Sensing Beyond Sparsity},
  journal = {IEEE Signal Processing Magazine},
  volume = {33},
  number = {1},
  pages = {78--93},
  year = {2016},
  doi = {10.1109/MSP.2015.2486805}
}

@article{RomeroLopezLeus2015,
  author = {Romero, Daniel and Lopez-Valcarce, Roberto and Leus, Geert},
  title = {Compression Limits for Random Vectors With Linearly Parameterized Second-Order Statistics},
  journal = {IEEE Transactions on Information Theory},
  volume = {61},
  number = {3},
  pages = {1410--1425},
  year = {2015},
  doi = {10.1109/TIT.2015.2394784}
}

@inproceedings{Eldar2020,
  author = {Eldar, Yonina C. and Li, Jerry and Musco, Cameron and Musco, Christopher},
  title = {Sample Efficient {Toeplitz} Covariance Estimation},
  booktitle = {Proceedings of the ACM-SIAM Symposium on Discrete Algorithms (SODA)},
  pages = {378--397},
  year = {2020},
  doi = {10.1137/1.9781611975994.23}
}

@inproceedings{Lawrence2020,
  author = {Lawrence, Hannah and Li, Jerry and Musco, Cameron and Musco, Christopher},
  title = {Low-Rank {Toeplitz} Matrix Estimation via Random Ultra-Sparse Rulers},
  booktitle = {Proceedings of the IEEE International Conference on Acoustics, Speech and Signal Processing (ICASSP)},
  pages = {4796--4800},
  year = {2020},
  doi = {10.1109/ICASSP40776.2020.9053026}
}

@book{CoverThomas2006,
  author = {Cover, Thomas M. and Thomas, Joy A.},
  title = {Elements of Information Theory},
  edition = {2},
  publisher = {Wiley},
  year = {2006},
  doi = {10.1002/047174882X}
}

@book{Tsybakov2009,
  author = {Tsybakov, Alexandre B.},
  title = {Introduction to Nonparametric Estimation},
  publisher = {Springer},
  year = {2009},
  doi = {10.1007/b13794}
}

@article{RudelsonVershynin2013,
  author = {Rudelson, Mark and Vershynin, Roman},
  title = {Hanson--{Wright} Inequality and Sub-{Gaussian} Concentration},
  journal = {Electronic Communications in Probability},
  volume = {18},
  number = {82},
  pages = {1--9},
  year = {2013},
  doi = {10.1214/ECP.v18-2865}
}

@article{Adamczak2015,
  author = {Adamczak, Rados{\l}aw},
  title = {A Note on the {Hanson--Wright} Inequality for Random Vectors with Dependencies},
  journal = {Electronic Communications in Probability},
  volume = {20},
  number = {72},
  pages = {1--13},
  year = {2015},
  doi = {10.1214/ECP.v20-3829}
}

@article{GotzeSambaleSinulis2019,
  author = {G{\"o}tze, Friedrich and Sambale, Holger and Sinulis, Arthur},
  title = {Higher Order Concentration for Functions of Weakly Dependent Random Variables},
  journal = {Electronic Journal of Probability},
  volume = {24},
  number = {85},
  pages = {1--19},
  year = {2019},
  doi = {10.1214/19-EJP338}
}

@article{Pinelis1994,
  author = {Pinelis, Iosif},
  title = {Optimum Bounds for the Distributions of Martingales in {B}anach Spaces},
  journal = {The Annals of Probability},
  volume = {22},
  number = {4},
  pages = {1679--1706},
  year = {1994},
  doi = {10.1214/aop/1176988477}
}

\end{document}